\documentclass[bj, preprint]{imsart}

\startlocaldefs

\usepackage[utf8]{inputenc}
\usepackage[a4paper]{geometry}
\usepackage{amsmath}
\usepackage{amsfonts}
\usepackage{amssymb}
\usepackage{amsthm}
\usepackage{mathrsfs}
\usepackage{dsfont}
\usepackage{paralist}
\usepackage{natbib}
\usepackage{bm}
\usepackage{color}
\usepackage[colorlinks=true,linkcolor=blue,citecolor=blue,pdfborder={0 0 0}]{hyperref}
\usepackage{graphicx}
\usepackage{enumerate}
\usepackage{changes}

\newcount\Comments  % 0 suppresses notes to selves in text
\newcommand{\kibitz}[2]{\ifnum\Comments=1\textcolor{#1}{#2}\fi}

\AtBeginDocument{%
   \def\MR#1{}
}

\newtheoremstyle{normal}% name
{2ex}               % Space above, empty = `usual value'
{3ex}               % Space below
{}                  % Body font
{}                  % Indent amount (empty = no indent, \parindent = para indent)
{\bfseries} % Thm head font
{}                  % Punctuation after thm head
{2pt}   % Space after thm head.
{\thmname{#1}\thmnumber{ #2.} \thmnote{(#3)}}% Thm head spec

\newtheoremstyle{italic}% name
{2ex}%      Space above, empty = `usual value'
{3ex}%      Space below
{\itshape}% Body font
{}%         Indent amount (empty = no indent, \parindent = para indent)
{\bfseries} % Thm head font
{}%        Punctuation after thm head
{2pt}%     Space after thm head.
{\thmname{#1}\thmnumber{ #2.} \thmnote{(#3)}}% Thm head spec

\theoremstyle{normal}
\newtheorem{definition}{Definition}[section]
\newtheorem{remark}[definition]{Remark}

\newtheorem{condition}[definition]{Condition}

\theoremstyle{italic}
\newtheorem{theorem}[definition]{Theorem}
\newtheorem{lemma}[definition]{Lemma}
\newtheorem{prop}[definition]{Proposition}

\newtheorem{add}[definition]{Addendum}

\newcommand\G{\mathbb{G}}
\newcommand\N{\mathbb{N}}

\newcommand\R{\mathbb{R}}
\newcommand\Z{\mathbb{Z}}
\newcommand\eps{\varepsilon}
\newcommand\Prob{\mathbb{P}}    %probability
\newcommand\Exp{\mathbb{E}}     %expected value
\newcommand\ind{\mathds{1}}     %indicators
\newcommand{\1}{\ind}

\newcommand\Ac{\mathcal{A}}

\newcommand\Fc{\mathcal{F}}
\newcommand\Hc{\mathcal{H}}

\newcommand\Nc{\mathcal{N}}

\newcommand\Xc{\mathcal{X}}

\newcommand\Eb{\mathbb{E}}

\newcommand\Gb{\mathbb{G}}
\newcommand\Hb{\mathbb{H}}

\newcommand{\Gn}{\mathbb{G}_n}
\newcommand{\Pn}{\mathbb{P}_n}

\DeclareMathOperator{\var}{var}
\DeclareMathOperator{\Var}{Var}

\DeclareMathOperator{\cov}{cov}

\newcommand{\Frechet}{\operatorname{\text{Fr\'echet}}}

\newcommand{\argmax}{\operatornamewithlimits{\arg\max}}

\newcommand{\diff}{\mathrm{d}}
\newcommand{\iu}{\mathrm{i}}

\newcommand{\mom}{\nu}

\newcommand\weak{\rightsquigarrow}

\newcommand{\norm}[1]{\lVert{#1}\rVert}
\newcommand{\abs}[1]{\left\lvert{#1}\right\rvert}
\newcommand{\floor}[1]{\lfloor{#1}\rfloor}

\newcommand{\ip}[1]{\lfloor #1 \rfloor}
\newcommand{\sumk}{\frac{1}{k}\sum_{i=1}^k}

\numberwithin{equation}{section}

\allowdisplaybreaks[1]

\endlocaldefs

 \setattribute{journal}{name}{}	%No "submitted to" on the title page, for arxiv version?

\begin{document}

\begin{frontmatter}

\title{Maximum likelihood estimation for the Fr\'echet distribution based on block maxima extracted from a time series}
\runtitle{Maximum likelihood estimation for the Fr\'echet distribution}

\begin{aug}
  \author{\fnms{Axel} \snm{B\"ucher}\thanksref{a1}\ead[label=e1]{axel.buecher@ruhr-uni-bochum.de}}
  \and
  \author{\fnms{Johan} \snm{Segers}\thanksref{a2}\ead[label=e2]{johan.segers@uclouvain.be}}

  \runauthor{A.\ B\"ucher and J.\ Segers}

  \affiliation{Ruhr-Universit\"at Bochum and Universit\'{e} catholique de Louvain}

  \address[a1]{Fakult\"at f\"ur Mathematik, Ruhr-Universit\"at Bochum, Universit\"atsstr.~150, 44780 Bochum, Germany. \printead{e1}}

  \address[a2]{Universit\'{e} catholique de Louvain,
	Institut de Statistique, Biostatistique et Sciences Actuarielles,
	Voie du Roman Pays~20,
	B-1348 Louvain-la-Neuve, Belgium. \printead{e2}}

\end{aug}

\begin{abstract}
The block maxima method in extreme-value analysis proceeds by fitting an extreme-value distribution to a sample of block maxima extracted from an observed stretch of a time series. The method is usually validated under two simplifying assumptions: the block maxima should be distributed exactly according to an extreme-value distribution and the sample of block maxima should be independent. Both assumptions are only approximately true. The present paper validates that the simplifying assumptions can in fact be safely made.

For general triangular arrays of block maxima attracted to the Fr\'echet distribution, consistency and asymptotic normality is established for the maximum likelihood estimator of the parameters of the limiting Fr\'echet distribution. The results are specialized to the common setting of block maxima extracted from a strictly stationary time series. The case where the underlying random variables are independent and identically distributed is further worked out in detail. The results are illustrated by theoretical examples and Monte Carlo simulations.
\end{abstract}


\end{frontmatter}

\section{Introduction}
% \label{sec:intro}

For the analysis of extreme values, two fundamental approaches can be distinguished. First, the \textit{peaks-over-threshold method} consists of extracting those values from the observation period which exceed a high threshold. To model such threshold excesses, asymptotic theory suggests the use of the Generalized Pareto distribution \citep{Pic75}. Second, the \textit{block maxima method} consists of dividing the observation period into a sequence of non-overlapping intervals and restricting attention to the largest observation in each time interval. Thanks to the extremal types theorem, the probability distribution of such block maxima is approximately Generalized Extreme-Value (GEV), popularized by \cite{Gum58}. The block maxima method is particularly common in environmental applications, since appropriate choices of the block size yield a simple but effective way to deal with seasonal patterns.

For both methods, honest theoretical justifications must take into account two distinct features. First, the postulated models for either threshold excesses or block maxima arise from asymptotic theory and are not necessarily accurate at sub-asymptotic thresholds or at finite block lengths. Second, if the underlying data exhibit serial dependence, then the same will likely be true for the extreme values extracted from those data.

How to deal with both issues is well-understood for the peaks-over-threshold method. The model approximation can be justified under a second-order condition (see, e.g., \citealp{dHF06} for a vast variety of applications), while serial dependence is taken care of in \cite{Hsi91,Dre00} or \cite{Roo09}, among others. Excesses over large thresholds often occur in clusters, and such serial dependence usually has an impact on the asymptotic variances of estimators based on these threshold excesses.

Surprisingly, perhaps, is that for the block maxima method, no comparable analysis has yet been done. With the exception of some recent articles, which we will discuss in the next paragraph, the commonly used assumption is that the block maxima constitute an independent random sample from a GEV distribution. The heuristic justification for assuming independence over time, even for block maxima extracted from time series data, is that for large block sizes, the occurrence times of the consecutive block maxima are likely to be well separated.

A more accurate framework is that of a triangular array of block maxima extracted from a sequence of random variables, the block size growing with the sample size. While \cite{Dom15} shows consistency of the maximum likelihood estimator \citep{prescott+w:1980} for the parameters of the GEV distribution, \cite{dHF15} show both consistency and asymptotic normality of the probability weighted moment estimators \citep{hosking+w+w:1985}. In both papers, however, the random variables from which the block maxima are extracted are supposed to be independent and identically distributed. In many situations, this assumption is clearly violated. To the best of our knowledge, \cite{BucSeg14} is the only reference treating both the approximation error and the time series character, providing large-sample theory of nonparametric estimators of extreme-value copulas based on samples of componentwise block maxima extracted out of multivariate stationary time series.

The aim of the paper is to show the consistency and asymptotic normality of the maximum likelihood estimator for more general sampling schemes, including the common situation of extracting block maxima from an underlying stationary time series. For technical reasons explained below, we restrict attention to the heavy-tailed case. The block maxima paradigm then suggests to use the two-parametric Fr\'echet distribution as a model for a sample of block maxima extracted from that time series.

The first (quite general) main result, Theorem~\ref{theo:asydis}, is that for triangular arrays of random variables whose empirical measures, upon rescaling, converge in an appropriate sense to a Fr\'echet distribution, the maximum likelihood estimator for the Fr\'echet parameters based on those variables is consistent and asymptotically normal. The theorem can be applied to the common set-up discussed above of block maxima extracted from an underlying time series, and the second main result, Theorem~\ref{theo:asydis:stat}, shows that, in this case, the asymptotic variance matrix is the inverse of the Fisher information of the Fr\'echet family: asymptotically, it is as if the data were an independent random sample from the Fr\'echet attractor. In this sense, our theorem confirms the soundness of the common simplifying assumption that block maxima can be treated as if they were serially independent. 
Interestingly enough, the result  allows for time series of which the strong mixing coefficients are not summable, allowing for some long range dependence~scenarios.
%Interestingly enough, the result also covers cases where the autocovariances of the underlying time series are not summable (as long as the Rosenblatt alpha mixing coefficients $\alpha_n$ decay at rate $n^{-a}$ for some $a \in (0,1)$).

%\ab{Rewrite next paragraph and cite our other paper.}

Restricting attention to the heavy-tailed case is done because of the non-standard nature of the three-parameter GEV distribution. The issue is that the support of a GEV distribution depends on its parameters. Even for the maximum likelihood estimator based on an independent random sample from a GEV distribution, asymptotic normality has not yet been established. The article usually cited in this context is \cite{Smi85}, although no formal result is stated therein. Even the differentiability in quadratic mean of the three-parameter GEV is still to be proven; \cite{marohn:1994} only shows differentiability in quadratic mean for the one-parameter GEV family (shape parameter only) at the Gumbel distribution. We feel that solving all issues simultaneously (irregularity of the GEV model, finite block size approximation error and serial dependence) is a far too ample program for one paper. For that reason, we focus on the analytically simpler Fr\'echet family, while thoroughly treating the triangular nature of the array of block maxima and the issue of serial dependence within the underlying time series. In a companion paper \citep{BucSeg16}, we consider the maximum likelihood estimator in the general GEV-model based on independent and identically distributed random variables sampled directly from the GEV distribution. The main focus of that paper is devoted to resolving the considerable technical issues arising from the dependence of the GEV support on its parameters.

We will build up the theory in three stages. First, we consider general triangular arrays of observations that asymptotically follow  a Fr\'echet distribution in Section~\ref{sec:tri}. Second, we apply the theory to the set-up of block maxima extracted from a strictly stationary time series in Section~\ref{sec:stat}. Third, we further specialize the results to the special case of block maxima formed from independent and identically distributed random variables in Section~\ref{sec:iid}. This section can hence be regarded as a continuation of \cite{Dom15} by reinforcing consistency to asymptotic normality, albeit for the Fr\'echet domain of attraction only. We work out an example and present finite-sample results from a simulation study in Section~\ref{sec:ex}. The main proofs are deferred to Appendix~\ref{sec:proofs}, while some auxiliary results concerning the Fr\'echet distribution are mentioned in Appendix~\ref{sec:aux}. The proofs of the less central results are postponed to a supplement.

\section{Triangular arrays of block maxima}
\label{sec:tri}

%If not all observed maxima are tied, the Fr\'echet likelihood admits a unique maximum (Subsection~\ref{subsec:tri:existsUnique}). In the setting of a triangular array of block maxima of which certain functionals admit a weak law of large numbers or a central limit theorem, the maximum likelihood estimator is consistent or asymptotically normal, respectively (Subsections~\ref{subsec:tri:consistency} and~\ref{subsec:tri:asydis}). Nothing is assumed about the processs underlying the block maxima; in fact, it is not even mentioned. \change{The variables $X_{n,i}$ are called ``block maxima'' for convenience only.} The proofs of the results in this section are given in Subsection~\ref{subsec:proofs:tri}.

In this section, we summarize results concerning the maximum likelihood estimator for the parameters of the Fr\'echet distribution:
given a sample of observations which are not all tied, the Fr\'echet likelihood admits a unique maximum (Subsection~\ref{subsec:tri:existsUnique}). If the observations are based on a triangular array which is approximately Fr\'echet distributed  in the sense that certain functionals admit a weak law of large numbers or a central limit theorem, the maximum likelihood estimator is consistent or asymptotically normal, respectively (Subsections~\ref{subsec:tri:consistency} and~\ref{subsec:tri:asydis}). 
Proofs  are given in Subsection~\ref{subsec:proofs:tri}.

\subsection{Existence and uniqueness}
\label{subsec:tri:existsUnique}

Let $P_{\theta}$ denote the two-parameter Fr\'echet distribution on $(0, \infty)$ with parameter $\theta = (\alpha, \sigma) \in (0, \infty)^2 = \Theta$, defined through its cumulative distribution function 
\[
  G_\theta(x) = \exp \{ - (x/\sigma)^{-\alpha} \}, \quad x>0.
\]
Its probability density function is equal to
\[
  p_\theta(x) = \frac{\alpha}{\sigma} \exp \{ - (x/\sigma)^{-\alpha} \} \, (x/\sigma)^{-\alpha-1}, \quad x>0,
\]
with log-likelihood function
\[
  \ell_\theta(x) = \log(\alpha/\sigma) - (x/\sigma)^{-\alpha} - (\alpha+1) \log(x/\sigma), \quad x>0,
\]
and score functions $\dot \ell_\theta=(\dot \ell_{\theta,1}, \dot \ell_{\theta,2})^T$, with
\begin{align}
\label{eq:score:alpha}
  \dot \ell_{\theta,1}(x) &= \partial_\alpha \ell_\theta(x)
  = \alpha^{-1} + \left( (x/\sigma)^{-\alpha} - 1 \right) \log(x/\sigma) , \\
\label{eq:score:sigma}
   \dot \ell_{\theta,2}(x) &=\partial_\sigma \ell_\theta(x)
  = \left( 1 - (x/\sigma)^{-\alpha} \right) \alpha/\sigma.
\end{align}

Let $\bm{x} = (x_1, \ldots, x_k) \in (0, \infty)^k$ be a sample vector to which the Fr\'echet distribution is to be fitted. Consider the log-likelihood function
\begin{equation}
\label{eq:loglik}
  L( \theta \mid \bm{x} ) = \sum_{i=1}^k \ell_\theta( x_i ).
\end{equation}
Further, define
\begin{align}
\label{eq:Psik}
  \Psi_k( \alpha \mid \bm{x} )
  &= \frac{1}{\alpha} + \frac{\sumk x_i^{-\alpha} \log( x_i )}{\sumk x_i^{-\alpha}} - \sumk \log(x_i), \\
\label{eq:sigmahat}
  {\sigma}( \alpha \mid \bm{x} ) 
  &= \left( \frac{1}{k} \sum_{i=1}^k x_i^{-\alpha} \right)^{-1/\alpha}.
\end{align}

\begin{lemma}[Existence and uniqueness]
\label{lem:existsUnique}
If the scalars $x_1, \ldots, x_k \in (0, \infty)$ are not all equal ($k \ge 2$), then there exists a unique maximizer
\[ 
  \hat{\theta}( \bm{x} ) 
  =
  \bigl( \hat{\alpha}( \bm{x} ),\hat \sigma(\bm x) \bigr)
  = \argmax_{\theta \in \Theta} L( \theta \mid \bm{x} ). 
\]
We have $\hat \sigma(\bm x) = \sigma( \hat{\alpha}(\bm{x}) \mid \bm{x} )$ while $\hat{\alpha}( \bm{x} )$ is the unique zero of the strictly decreasing function $\alpha \mapsto \Psi_k( \alpha \mid \bm{x})$:
\begin{equation}
\label{eq:estimEq}
  \Psi_k \bigl( \hat{\alpha}( \bm{x} ) \mid \bm{x} \bigr) = 0.
\end{equation}
\end{lemma}

It is easily verified that the estimating equation for $\alpha$ is scale invariant: for any $c \in (0, \infty)$, we have $\Psi_k( \alpha \mid c\bm{x} ) = \Psi_k( \alpha \mid \bm{x} )$. As a consequence, the maximum likelihood estimator for the shape parameter is scale invariant:
\[
  \hat{\alpha}( c\bm{x} ) = \hat{\alpha}(\bm{x}).
\]
Moreover, the estimator for $\sigma$ is a scale parameter in the sense that
\[
  \hat \sigma(c\bm x) = {\sigma}( \hat{\alpha}(c\bm{x}) \mid c\bm{x} )
  =
  c \, {\sigma}( \hat{\alpha}(\bm{x}) \mid \bm{x} )
  = 
  c \, \hat\sigma(\bm x).
\]

Until now, the maximum likelihood estimator is defined only in case not all $x_i$ values are identical. For definiteness, if $x_1 = \ldots = x_k$, define $\hat{\alpha}( \bm{x} ) = \infty$ and $\hat{\sigma}(\bm x ) = \min(x_1, \ldots, x_k) = x_1$.

\subsection{Consistency}
\label{subsec:tri:consistency}

We derive a general condition under which the maximum likelihood estimator for the parameters of the Fr\'echet distribution is consistent. The central result, Theorem~\ref{theo:consistency} below, shows that, apart from a not-all-tied condition, the only thing that is required for consistency is a weak law of large numbers for the functions appearing in the estimating equation \eqref{eq:estimEq} for the shape parameter.

Suppose that for each positive integer $n$, we are given a random vector $\bm{X}_n = (X_{n,1}, \ldots, X_{n,k_n})$ taking values in $(0, \infty)^{k_n}$, where $k_n \ge 2$ is a positive integer sequence such that $k_n \to \infty$ as $n \to \infty$. One may think of $X_{n,i}$ as being (approximately) Fr\'echet distributed with shape parameter $\alpha_0>0$ and scale parameter $\sigma_n>0$. This statement is made precise in Condition~\ref{cond:LLN} below.
On the event that the $k_n$ variables $X_{n,i}$ are not all equal, Lemma~\ref{lem:existsUnique} allows us to define
\begin{equation}
\label{eq:MLE:shape}
  \hat{\alpha}_n 
  = 
  \hat{\alpha}( \bm{X}_n ),
\end{equation}
the unique zero of the function $0 < \alpha \mapsto \Psi_{k_n}( \alpha \mid \bm{X}_n )$. Further, as in \eqref{eq:sigmahat}, put
\begin{equation}
\label{eq:MLE:scale}
  \hat{\sigma}_n 
  = 
  {\sigma}( \hat{\alpha}_n \mid \bm{X}_n ) 
  = 
  \left( \frac{1}{k_n} \sum_{i=1}^{k_n} X_{n,i}^{-\hat{\alpha}_n} \right)^{-1/\hat{\alpha}_n}.
\end{equation}
For definiteness, put $\hat{\alpha}_n = \infty$ and $\hat{\sigma}_n = X_{n,1}$ on the event $\{ X_{n,1} = \ldots = X_{n,k_n} \}$. Subsequently, we will assume that this event is asymptotically negligible:
\begin{equation}
\label{eq:noties}
  \lim_{n \to \infty} \Pr(X_{n,1} = \ldots = X_{n,k_n}) = 0.
\end{equation}
We refer to $(\hat{\alpha}_n, \hat{\sigma}_n)$ as the maximum likelihood estimator.

% \begin{condition}
% \label{cond:setup}
% We have 
% \begin{equation*}
%   \lim_{n \to \infty} \Pr(X_{n,1} = \ldots = X_{n,k_n}) = 0.
% \end{equation*}
% \end{condition}

The fundamental condition guaranteeing consistency of the maximum likelihood estimator concerns the asymptotic behavior of sample averages of $f(X_{n,i}/\sigma_n)$ for certain functions $f$. 
For $0 < \alpha_- < \alpha_+ < \infty$, consider the function class
\begin{equation}
\label{eq:Falpha-+}
  \mathcal{F}_1(\alpha_-, \alpha_+)
  =
  \{ x \mapsto \log x\}
  \cup
  \{ x \mapsto x^{-\alpha} : \alpha_- < \alpha < \alpha_+ \}
  \cup
  \{ x \mapsto x^{-\alpha} \log x : \alpha_- < \alpha < \alpha_+ \},
\end{equation}
all functions being from $(0, \infty)$ into $\R$.
% Consider the function classes
% \begin{align}
% \label{eq:F1}
%   \mathcal{F}_0
%   &=
%   \{ x \mapsto x^{-\alpha} : \alpha \in (0, \infty) \} , \qquad
%   \mathcal{F}_1
%   =
%   \{ x \mapsto x^{-\alpha} \log(x) : \alpha \in [0, \infty) \},
% \end{align}
% all functions being from $(0, \infty)$ into $\R$. Since $\alpha = 0$ is allowed on the right-hand side of \eqref{eq:F1}, the logarithm belongs to $\mathcal{F}_1$ too. 
Let the arrow `$\weak$' denote weak convergence. 

\begin{condition}
\label{cond:LLN}
There exist $0 < \alpha_- < \alpha_0 < \alpha_+ < \infty$ and a positive sequence $(\sigma_n)_{n \in \N}$ such that, for all $f \in \mathcal{F}_1( \alpha_-, \alpha_+ )$,
\begin{equation}
\label{eq:LLN}
  \frac{1}{k_n} \sum_{i=1}^{k_n} f( X_{n,i}/\sigma_n )
  \weak
  \int_0^\infty f(x) \, p_{\alpha_0, 1}(x) \, \diff x,
  \qquad
  n \to \infty.
\end{equation}
\end{condition}

% The following theorem shows that the maximum likelihood estimator is consistent.

\begin{theorem}[Consistency]
\label{theo:consistency}
Let $\bm{X}_n = (X_{n,1}, \ldots, X_{n,k_n})$ be a sequence of random vectors in $(0, \infty)^{k_n}$, where $k_n \to \infty$. Assume that Equation~\eqref{eq:noties} and Condition~\ref{cond:LLN} hold. 
On the complement of the event $\{ X_{n,1} = \ldots = X_{n,k_n} \}$, the random vector $(\hat{\alpha}_n, \hat{\sigma}_n)$ is the unique maximizer of the log-likelihood $(\alpha, \sigma) \mapsto L(\alpha, \sigma \mid X_{n,1}, \ldots, X_{n,k_n})$. Moreover, the maximum likelihood estimator is consistent in the sense that
\[
  ( \hat{\alpha}_n, \hat{\sigma}_n / \sigma_n ) 
  \weak (\alpha_0, 1), \qquad n \to \infty.
\]
\end{theorem}

\subsection{Asymptotic distribution}
\label{subsec:tri:asydis}

We formulate a general condition under which the estimation error of the maximum likelihood estimator for the Fr\'echet parameter vector converges weakly. The central result is Theorem~\ref{theo:asydis} below.

For $0 < \alpha_- < \alpha_+ < \infty$, recall the function class $\mathcal{F}_1( \alpha_-, \alpha_+ )$ in \eqref{eq:Falpha-+} and define another one:
\begin{equation}
\label{eq:F2}
  \mathcal{F}_2(\alpha_-, \alpha_+) 
  =
  \mathcal{F}_1(\alpha_-, \alpha_+)
  \cup
  \{ x \mapsto x^{-\alpha} (\log x)^2 : \alpha_- < \alpha < \alpha_+ \}.
\end{equation}
% Recall the function classes $\mathcal{F}_0$ and $\Fc_1$ in \eqref{eq:F1} and define yet another one,
% \begin{equation}
% \label{eq:F2}
%   \mathcal{F}_2 = \{ x \mapsto x^{-\alpha} (\log x)^2 : \alpha \in [0, \infty) \}.
% \end{equation}
Furthermore, fix $\alpha_0>0$ and consider the following triple of real-valued functions on $(0,\infty)$:
\begin{equation}
\label{eq:H}
  \mathcal{H} 
  = 
  \{ f_1, f_2, f_3 \} 
  = 
  \{ x \mapsto x^{-\alpha_0} \log(x), \, x \mapsto x^{-\alpha_0}, \, x \mapsto \log x \}.
\end{equation}
The following condition strengthens Condition~\ref{cond:LLN}.

\begin{condition}
\label{cond:CLT}
There exist $\alpha_0 \in (0, \infty)$ and a positive sequence $(\sigma_n)_{n \in \N}$ such that the following two statements hold:
\begin{enumerate}[(i)]
\item
There exist $0 < \alpha_- < \alpha_0 < \alpha_+ < \infty$ such that Equation~\eqref{eq:LLN} holds for all $f \in \mathcal{F}_2(\alpha_-, \alpha_+)$.

\item
There exists a sequence $0 < v_n \to \infty$ and a random vector $\bm Y=(Y_1, Y_2, Y_3)^T$ such that, denoting
\begin{equation}
\label{eq:Gn}
  \Gn f = 
  v_n 
  \left( 
    \frac{1}{k_n} \sum_{i=1}^{k_n} f( X_{n,i} / \sigma_n )
    -
    \int_0^\infty f(x) \, p_{\alpha_0,1}(x) \, \diff x
  \right),
\end{equation}
we have, for $f_j$ as in \eqref{eq:H},
\begin{equation}
\label{eq:Y}
%   \bigl(\Gn x^{-\alpha_0} \log(x), \G_n x^{-\alpha_0}, \Gn \log(x) \bigr)
  (\Gn f_1, \Gn f_2, \, \Gn f_3)^T
  \weak
  \bm{Y}
%   (Y_1, Y_2, Y_3)
  , \qquad n \to \infty.
\end{equation}
\end{enumerate}
\end{condition}

% \begin{condition}
% \label{cond:CLT}
% There exists $\alpha_0 \in (0, \infty)$, a positive sequence $v_n \to \infty$  and a random vector $\bm Y=(Y_1, Y_2, Y_3)^T$ such that, denoting
% \[
%   \Gn f = 
%   v_n 
%   \left( 
%     \frac{1}{k_n} \sum_{i=1}^{k_n} f( X_{n,i} / \sigma_n )
%     -
%     \int_0^\infty f(x) \, \diff e^{-x^{1/\alpha_0}}
%   \right),
% \]
% we have
% \[
%   \bigl(\Gn x^{-\alpha_0} \log(x), \G_n x^{-\alpha_0}, \Gn \log(x) \bigr)
%   \weak
%   (Y_1, Y_2, Y_3), \qquad n \to \infty.
% \]
% % where $\Pb_n$ and $P$ are as in Condition~\ref{cond:LLN2}. 
% \end{condition}

Let $\Gamma$ be the Euler gamma function and let $\gamma = -\Gamma'(1) = 0.5772\dots$ be the Euler--Mascheroni constant. Recall $\Gamma''(2)=(1-\gamma)^2+\pi^2/6-1$. Define the matrix
\begin{equation}
\label{eq:M}
  M(\alpha_0)
  =
%  \begin{pmatrix}
%    6\alpha_0^2/\pi^2 & 6\alpha_0(1-\gamma)/\pi^2 & - 6\alpha_0^2/\pi^2 \\
%%    6(\gamma-1)/\pi^2 & -6(1-\gamma)^2/(\pi^2\alpha_0) - 1/\alpha_0 &  6(1-\gamma)/\pi^2 
%    6(\gamma-1)/\pi^2 & -6(\Gamma''(2) +1) / (\pi^2\alpha_0) &  6(1-\gamma)/\pi^2 
%  \end{pmatrix}.
 \frac{6}{\pi^2} 
 \begin{pmatrix}
    \alpha_0^2 & \alpha_0(1-\gamma) & - \alpha_0^2 \\
%    6(\gamma-1)/\pi^2 & -6(1-\gamma)^2/(\pi^2\alpha_0) - 1/\alpha_0 &  6(1-\gamma)/\pi^2 
    \gamma-1 & -(\Gamma''(2) +1) / \alpha_0 &  1-\gamma
  \end{pmatrix},
  \qquad \alpha_0 \in (0, \infty).
\end{equation}

\begin{theorem}[Asymptotic distribution]
\label{theo:asydis}
Let $\bm{X}_n = (X_{n,1}, \ldots, X_{n,k_n})$ be a sequence of random vectors in $(0, \infty)^{k_n}$, where $k_n \to \infty$. Assume that Equation~\eqref{eq:noties} and Condition~\ref{cond:CLT} hold. As $n \to \infty$, the maximum likelihood estimator $(\hat{\alpha}_n, \hat{\sigma}_n)$ satisfies
\begin{equation}
\label{eq:joint:delta}
  \begin{pmatrix} 
    v_n ( \hat{\alpha}_n - \alpha_0 ) \\ 
    v_n \, ( \hat{\sigma}_n / \sigma_n - 1 ) 
  \end{pmatrix} 
  = 
  M(\alpha_0)
  \begin{pmatrix}
    \Gn x^{-\alpha_0} \log(x) \\
    \Gn x^{-\alpha_0} \\
    \Gn \log(x)
  \end{pmatrix} + o_p(1)
  \weak
  M(\alpha_0) \bm{Y},
\end{equation}
where $\bm{Y} = (Y_1, Y_2, Y_3)^T$ and $M(\alpha_0)$ are given in Equations~\eqref{eq:Y} and~\eqref{eq:M}, respectively. 
\end{theorem}

For block maxima extracted from a strongly mixing stationary time series, Condition~\ref{cond:CLT} with $v_n = \sqrt{k_n}$, where $k_n$ denotes the number of blocks, will be derived from the Lindeberg central limit theorem. In that case, the distribution of $\bm{Y}$ is trivariate Gaussian with some mean vector~$\mu_{\bm Y}$ (possibly different from $0$, see Theorem~\ref{theo:asydis:stat} below for details) and covariance matrix
\begin{align} \label{eq:Sigma}
  \Sigma_{\bm Y}
  &=
  \frac{1}{\alpha_0^2}
  \begin{pmatrix}
    1-4\gamma+\gamma^2+\pi^2/3 & \alpha_0(\gamma - 2) & \pi^2/6-\gamma \\
    \alpha_0(\gamma - 2) & \alpha_0^2 & -\alpha_0 \\
    \pi^2/6-\gamma &-\alpha_0 &  \pi^2/6
  \end{pmatrix}.
\end{align}
According to Lemma~\ref{lem:cov} below, the right-hand side in \eqref{eq:Sigma} coincides with the covariance matrix of the random vector $\bigl( X^{-\alpha_0} \log(X), \, X^{-\alpha_0}, \, \log(X) \bigr)^T$, where $X$ is Fr\'echet distributed with parameter $(\alpha_0,1)$. From Lemma~\ref{lem:fisher}, recall the inverse of the Fisher information matrix of the Fr\'echet family at $(\alpha, \sigma) = (\alpha_0, 1)$:
\begin{equation}
\label{eq:fisher}
  I_{(\alpha_0,1)}^{-1} 
  = 
  \frac{6}{\pi^2} 
  \begin{pmatrix}
    \alpha_0^2 & (\gamma-1) \\ 
    (\gamma-1) &  \alpha_0^{-2} \{ (1-\gamma)^2 +\pi^2/6 \} 
  \end{pmatrix}.
\end{equation}

\begin{add}
\label{add:asydis}
If $\bm Y$ is normally distributed with covariance matrix $\Sigma_{\bm Y}$ as in \eqref{eq:Sigma}, then the limit $M(\alpha_0) \bm{Y}$ in Theorem~\ref{theo:asydis} is also normally distributed and its covariance matrix is equal to the inverse of the Fisher information matrix of the Fr\'echet family, $M(\alpha_0) \, \Sigma_{\bm Y} \, M(\alpha_0)^T = I_{(\alpha_0, 1)}^{-1}$. % in \eqref{eq:fisher}.
\end{add}

\section{Block maxima extracted from a stationary time series}
\label{sec:stat}

Let $(\xi_t)_{t \in \Z}$ be a strictly stationary time series, that is, for any $k\in\N$ and $\tau, t_1, \dots, t_k \in \Z$, the distribution of $(\xi_{t_1+\tau}, \dots, \xi_{t_k+\tau})$ is the same as the distribution of $(\xi_{t_1}, \dots, \xi_{t_k})$. For positive integer $i$ and $r$, consider the block maximum
\[
  M_{r,i} = \max( \xi_{(i-1)r+1}, \ldots, \xi_{ir} ).
\]
Abbreviate $M_{r,1} = M_r$. The classical block maxima method consists of choosing a sufficiently large block size $r$ and fitting an extreme-value distribution to the sample of block maxima $M_{r,1}, \ldots, M_{r,k}$. The likelihood is constructed under the simplifying assumption that the block maxima are independent. The present section shows consistency and asymptotic normality of this method in an appropriate asymptotic framework.

For the block maxima distribution to approach its extreme-value limit, the block sizes must increase to infinity. Moreover, consistency can only be achieved when the number of blocks grows to infinity too. Hence, we consider a positive integer sequence $r_n$, to be thought of as a sequence of block sizes. The number of disjoint blocks of size $r_n$ that fit into a sample of size $n$ is equal to $k_n = \ip{n/r_n}$, where $\ip{x}$ denotes the integer part of a real number $x$. Assume that both $r_n \to \infty$ and $k_n \to \infty$ as $n \to \infty$.

The theory will be based on an application of Theorem~\ref{theo:asydis} to the sample of left-truncated block maxima $X_{n,i} = M_{r_n,i} \vee c$ ($i = 1, \ldots, k_n$), for some  positive constant $c$ specified below. The estimators $\hat{\alpha}_n$ and $\hat{\sigma}_n$ are thus the ones in \eqref{eq:MLE:shape} and \eqref{eq:MLE:scale}, respectively. The reason for the left-truncation is that otherwise, some of the block maxima could be zero or negative. Asymptotically, such left-truncation does not matter, since all maxima will simultaneously diverge to infinity in probability (Condition~\ref{cond:small} below).

In Section~\ref{sec:iid} below, we will specialize things further to the case where the random variables $\xi_t$ are independent. In particular, we will simplify the list of conditions given in this section.

The basic assumption is that the distribution of rescaled block maxima is asymptotically Fr\'echet. The sequence of scaling constants should possess a minimal degree of regularity. The assumption is satisfied in case the stationary distribution of the series is in the Fr\'echet domain of attraction and the series possesses a positive extremal index; see Remark~\ref{rem:extrIndex} below. 

\begin{condition}[Domain of attraction]
\label{cond:DA}
The time series $(\xi_t)_{t \in \Z}$ is strictly stationary and there exists a sequence $(\sigma_n)_{n\in\N}$ of positive numbers with $\sigma_n \to \infty$ and a positive number $\alpha_0$ such that
\begin{equation}
\label{eq:DA:stat}
  M_n / \sigma_n \weak \Frechet( \alpha_0, 1 ), \qquad n \to \infty.
\end{equation}
Moreover, $\sigma_{m_n} / \sigma_n \to 1$ for any integer sequence $(m_n)_{n\in\N}$ such that $m_n / n \to 1$ as $n \to \infty$.
\end{condition}

The domain-of-attraction condition implies that, for every scalar $c$, we have $\Pr[ M_n \le c ] = \Pr[ M_n / \sigma_n \le c / \sigma_n ] \to 0$ as $n \to \infty$. In words, the block maxima become unboundedly large as the sample size grows to infinity. Still, out of a sample of $k_n$ block maxima, the smallest of the maxima might still be small, especially when the number of blocks is large, or, equivalently, the block sizes are not large enough. The following condition prevents this from happening.

\begin{condition}[All block maxima diverge]
\label{cond:small}
For every $c \in (0, \infty)$, we have
\[
  \lim_{n \to \infty} \Pr[ \min( M_{r_n,1}, \ldots, M_{r_n, k_n} ) \le c ] = 0.
\]
\end{condition}

To control the serial dependence within the time series, we require that the Rosenblatt mixing coefficients decay sufficiently fast: for positive integer $\ell$, put
\[
  \alpha(\ell) 
  = 
  \sup \big\{
  \abs{ \Pr(A \cap B) - \Pr(A) \Pr(B) } :     A \in \sigma( \xi_t : t \le 0 ),
    B \in \sigma( \xi_t : t \ge \ell ) \big\},
\]
where $\sigma(\,\cdot\,)$ denotes the $\sigma$-field generated by its argument.

\begin{condition}[$\alpha$-Mixing with rate]
\label{cond:alpha}
We have $\lim_{\ell \to \infty} \alpha(\ell) = 0$. Moreover, there exists $\omega > 0$ 
such that
\begin{equation}
\label{eq:mixing:kdelta}
  k_n^{1+\omega} \, \alpha(r_n) \to 0, \qquad n \to \infty.
\end{equation}
%and there exists a positive integer sequence $\ell_n$ such that $\ell_n / r_n \to 0$ as well as
%\begin{equation}
%  \label{eq:mixing:ellr}
%  \frac{r_n}{\ell_n} \alpha(\ell_n) \to 0, \qquad n \to \infty,
%\end{equation}
\end{condition}

Condition~\ref{cond:alpha} can be interpreted as requiring the block sizes $r_n$ to be sufficiently large. For instance, if $\alpha(\ell) = O( \ell^{-a} )$ for some $a > 0$, then \eqref{eq:mixing:kdelta} is satisfied as soon as $r_n$ is of larger order than 
%$n^{(1+\omega)/(1+\omega+a)}$. 
$n^{(1+\eps)/(1+a)}$ for some $0<\eps<a$; in that case, one may choose $\omega=\eps$.
Note that the exponent $a$ is allowed to be smaller than one, in which case the sequence of mixing coefficients is not summable.

%The Fr\'echet distribution satisfies $\int_0^\infty x^\beta \, p_{\alpha_0,1}(x) \, \diff x < \infty$ for every real $\beta$ less than $\alpha_0$. In order to be able to integrate \eqref{eq:DA:stat} to the limit, we require an asymptotic bound on the moments of the block maxima. 
%
%\begin{condition}[Moments]
%\label{cond:moment}
%For all $c \in (0, \infty)$ and all $\beta \in (-\infty, \alpha_0)$, we have
%\[
%  \limsup_{n \to \infty}
%  \Exp \bigl[ \bigl\{ (M_n \vee c)/\sigma_n \bigr\}^\beta \bigr]
%  < \infty.
%\]
% \js{In the light of the new Condition~\ref{cond:CLT}, we can perhaps reduce the set of $\beta$.}
%\end{condition}
%
% \ab{I agree: by monotonicity and the new condition 2.4, I think that we may reduce the condition to just one function: we need to control a certain power moment to the left, and a  logarithmic moment to the right.  The exponents depend on $\omega$ from the subsequent condition (or vice versa).}

In order to be able to integrate \eqref{eq:DA:stat} to the limit, we require an asymptotic bound on certain moments of the block maxima; more precisely, on negative power moments in the left tail and on logarithmic moments in the right tail.

%\begin{condition}[Moments]
%\label{cond:moment}
%For all $c \in (0, \infty)$ there exists some $\mom>2/\omega$ with $\omega$ from Condition~\ref{cond:alpha} such that
%\[
%  \limsup_{n \to \infty}
%  \Exp \bigl[ g_{\mom, \alpha_0}\bigl( (M_n \vee c)/\sigma_n \bigr)\bigr]
%  < \infty,
%\]
%where $g_{\mom, \alpha_0}(x) = \{ x^{-\alpha_0} \ind(x\le e) + \log (x) \ind(x>e) \}^{2+\mom}$.
%\end{condition}
%In fact, an elementary argument shows that, if the condition is satisfied for $c=1$, then it automatically holds for all $c>0$. However, for the ease of reading in later sections, we prefer to keep the condition in the above form.
\begin{condition}[Moments]
\label{cond:moment}
There exists some $\mom>2/\omega$ with $\omega$ from Condition~\ref{cond:alpha} such that
\begin{align}
  \limsup_{n \to \infty}
  \Exp \bigl[ g_{\mom, \alpha_0}\bigl( (M_n \vee 1)/\sigma_n \bigr)\bigr]
  < \infty, \label{eq:mom1}
\end{align}
where $g_{\mom, \alpha_0}(x) = \{ x^{-\alpha_0} \ind(x\le e) + \log (x) \ind(x>e) \}^{2+\mom}$.
\end{condition}
An elementary argument shows that if Condition~\ref{cond:moment} holds, then $M_n \vee 1$ in the $\limsup$ may be replaced by $M_n\vee c$, for arbitrary $c > 0$.
Moreover, note that the limiting Fr\'echet distribution satisfies $\int_0^\infty x^\beta \, p_{\alpha_0,1}(x) \, \diff x < \infty$ if and only if $\beta$ is less than $\alpha_0$. 
In some scenarios, e.g., for the iid case considered in Section~\ref{sec:iid} or  for the moving maximum process considered in Section~\ref{subsec:movmax}, 
it can be shown that  the following sufficient condition for  Condition~\ref{cond:moment} is true:
\begin{align} \label{eq:momall}
  \limsup_{n \to \infty}
  \Exp \bigl[ \bigl\{ (M_n \vee c)/\sigma_n \bigr\}^\beta \bigr]
  < \infty
\end{align}
for all $c>0$ and all $\beta \in (-\infty, \alpha_0)$. In that case, Condition~\ref{cond:moment} is easily satisfied for any $\mom>0$.

By Condition~\ref{cond:small} and Lemma~\ref{lem:ties}, the probability 
% of the event 
that all block maxima $M_{r_n,1}, \ldots, M_{r_n,k_n}$ are larger than some positive constant $c$ and that they are not all equal tends to unity. On this event, we can study the maximum likelihood estimators $(\hat{\alpha}_n, \hat{\sigma}_n)$ for the parameters of the Fr\'echet distribution based on the sample of block maxima. 

Fix $c \in (0, \infty)$ and put
\[ 
  X_{n,i} = M_{r_n,i} \vee c. 
\]
Let $\Gn$ be the empirical process associated to $X_{n,1}/\sigma_{r_n}, \ldots, X_{n,k_n}/\sigma_{r_n}$ as in \eqref{eq:Gn} with $v_n = \sqrt{k_n}$. The empirical process is not necessarily centered, which is why we need a handle on its expectation.

\begin{condition}[Bias]
\label{cond:bias}
There exists $c \in (0, \infty)$ such that for every function $f$ in $\mathcal{H}$ defined in \eqref{eq:H}, the following limit exists:
\begin{equation}
\label{eq:bias}
  \lim_{n \to \infty} 
  \sqrt{k_n} 
  \left( 
    \Exp\bigl[ f \bigl( (M_{r_n} \vee c) / \sigma_{r_n} \bigr)\bigr] 
    - 
    \int_0^\infty f(x) \, p_{\alpha_0, 1}(x) \, \diff x
  \right)
  = 
  B(f).
\end{equation}
\end{condition}

\begin{theorem}
\label{theo:asydis:stat} Suppose that Conditions~\ref{cond:DA} up to~\ref{cond:bias} are satisfied and fix $c$ as in Condition~\ref{cond:bias}. 
Then, with probability tending to one, there exists a unique maximizer $(\hat{\alpha}_n, \hat{\sigma}_n)$ of the Fr\'echet log-likelihood \eqref{eq:loglik} based on the block maxima $M_{r_n,1}, \ldots, M_{r_n,k_n}$, and we have, as $n \to \infty$,
\begin{align*}
% \label{eq:joint:delta}
  \begin{pmatrix} 
    \sqrt{k_n} \, ( \hat{\alpha}_n - \alpha_0 ) \\ 
    \sqrt{k_n} \, ( \hat{\sigma}_n / \sigma_{r_n} - 1 ) 
  \end{pmatrix} 
  &= 
  M(\alpha_0)
  \begin{pmatrix}
    \Gn x^{-\alpha_0} \log(x) \\
    \Gn x^{-\alpha_0} \\
    \Gn \log(x)
  \end{pmatrix} + o_p(1) 
  \weak
%   M(\alpha_0) \bm{Y},
%   \qquad n \to \infty,
%   \bigl( Y_1 + \frac{1-\gamma}{\alpha_0} Y_2 - Y_3, \; \text{TO BE COMPLETED} \bigr), 
  \mathcal{N}_2 \bigl( 
    M(\alpha_0) \, B, \; I_{\alpha_0,1}^{-1}
  \bigr).
\end{align*}
% where $M(\alpha_0)$ is given in eq.\ (4.4) 
Here,  $M(\alpha_0)$ and $I_{\alpha_0,1}^{-1}$ are defined in Equations~\eqref{eq:M} and \eqref{eq:fisher}, respectively, while  $B = (B(f_1), B(f_2), B(f_3))^T$, where $B(f)$ is the limit in \eqref{eq:bias} and where $f_1, f_2, f_3$ are defined in \eqref{eq:H}.
% where $M(\alpha_0)$ and $\bm{Y}$ have been defined just before the statement of the theorem.
% The limiting distribution is bivariate Normal with covariance matrix equal to the inverse, $I^{-1}_{\alpha_0, 1}$, of the Fisher information matrix in \eqref{eq:fisher} and with mean vector equal to $M(\alpha_0) \, B_{\bm{Y}}$.
\end{theorem}

The proof of Theorem~\ref{theo:asydis:stat} is given in Subsection~\ref{subsec:proofs:stat}.
The conditions imposed in Theorem~\ref{theo:asydis:stat} are rather high-level. In the setting of a sequence of independent and identically distributed random variables, they can be brought down to analytical conditions on the tail of the stationary distribution function (Theorem~\ref{theo:ml}). Moreover, all conditions will be worked out in a moving maximum model in Section~\ref{subsec:movmax}. Still, we admit that for more common time series models, such as linear time series with heavy-tailed innovations or solutions to stochastic recurrence equations, checking the conditions in Theorem~\ref{theo:asydis:stat} may not be an easy matter. Especially the bias Condition~\ref{cond:bias}, which requires quite detailed knowledge on the distribution of the sample maximum, may be hard to verify. Even in the i.i.d.\ case, where the distribution of the sample maximum is known explicitly, checking Condition~\ref{cond:bias} occupies more than three pages in the proof of Theorem~\ref{theo:ml} below.

Interestingly, the asymptotic covariance matrix in Theorem~\ref{theo:asydis:stat} is unaffected by serial dependence and the asymptotic standard deviation of $\sqrt{k_n} ( \hat{\alpha}_n - \alpha_0 )$ is always equal to $(\sqrt{6}/\pi) \times \alpha_0 \approx 0.7797 \times \alpha_0$. The reason for this invariance is that even for time series, maxima over large disjoint blocks are asymptotically independent because of the strong mixing condition.
% The scaling sequence $\sigma_n$, however, may be different from the one in the case of independent variables, as explained in Remark~\ref{rem:extrIndex}.

\begin{remark}[Domain-of-attraction condition for positive extremal index]
\label{rem:extrIndex}
Let $F$ be the cumulative distribution function of $\xi_1$. Assume that there exist $0 < a_n \to \infty$ and $\alpha_0 \in (0, \infty)$ such that
\[ 
  \lim_{n \to \infty} F^n( a_n x ) = \exp( - x^{-\alpha_0} ), \qquad x \in (0, \infty). 
\]
Moreover, assume that the sequence $(\xi_t)_{t \in \Z}$ has extremal index $\vartheta \in (0, 1]$ \citep{leadbetter:1983}: If $u_n \to \infty$ is such that $F^n( u_n )$ converges, then 
\[ 
  \Pr( M_n \le u_n ) = F^{n \vartheta}( u_n ) + o(1), \qquad n \to \infty.
\]
Note that we assume that $\vartheta > 0$. Putting 
$
  \sigma_n = \vartheta^{1/\alpha_0} a_n
$ 
we obtain that Condition~\ref{cond:DA} is satisfied: for every $x \in (0, \infty)$, we have
\begin{align*}
  \Pr( M_n / \sigma_n \le x )
  &= F^{n\vartheta}( \sigma_n x ) + o(1) 
  \to \exp \bigl( - \vartheta (\vartheta^{1/\alpha_0} x)^{-\alpha_0} \bigr) = \exp ( - x^{-\alpha_0} ),
  \qquad n \to \infty.
\end{align*}
\end{remark}

\section{Block maxima extracted from an iid sample}
\label{sec:iid}

We specialize Theorem~\ref{theo:asydis:stat} to the case where the random variables $\xi_1, \xi_2, \ldots$ are independent and identically distributed with common distribution function $F$. In this setting, fitting extreme-value distributions to block maxima is also considered in \cite{Dom15} (consistency of the maximum likelihood estimator in the GEV-family with $\gamma>-1$) and \cite{dHF15} (asymptotic normality of the probability weighted moment estimator in the GEV-family with $\gamma<1/2$).
Assume that $F$ is in the maximum domain of attraction of the Fr\'echet distribution with shape parameter~$\alpha_0 \in (0, \infty)$: there exists a positive scalar sequence $(a_n)_{n \in \N}$ such that, for every $x \in (0, \infty)$,
\begin{equation}
\label{eq:DA}
%  \Pr[ \max(\xi_1, \ldots, \xi_n) / a_n \le x ]
%  = 
  F^n(a_n x) 
  \to 
  e^{-x^{-\alpha_0}}, 
  \qquad n \to \infty.
\end{equation}
% The scaling sequence $a_n$ is asymptotically equivalent to the tail quantile sequence $F^-(1-1/n)$, where $F^-(p) = \inf\{ x \in \R : F(x) \ge p \}$.

Because of serial independence, the conditions in Theorem~\ref{theo:asydis:stat} can be simplified considerably.
%the block maxima are independent as well, and their distribution function admits a simple expression. 
In addition, the mean vector of the asymptotic bivariate normal distribution of the maximum likelihood estimator can be made explicit. Required is a second-order reinforcement of \eqref{eq:DA} in conjunction with a growth restriction on the number of blocks.

% Let $\xi_1, \xi_2, \ldots$ be independent, identically distributed random variables with common distribution function $F$. 

% Suppose that we observe an i.i.d.\ sample $\xi_1, \dots, \xi_n$ from the distribution $F$. 
% Let $r=r_n$ be a positive integer sequence such that $r_n \to \infty$ and $r_n / n \to 0$ as $n \to \infty$. Define $k_n = \floor{ n / r_n }$ (integer part). For $i=1, \dots, k_n$, define the $i$th block maximum of the sample $\xi_1, \dots, \xi_n$ as
% \[
%   M_{r_n,i} = \max\{ \xi_{(i-1)r+1}, \ldots, \xi_{ir} \}.
% \]
% One should think of $r_n$ as the block size and $k_n$ as the number of blocks in a sample of size $n$. We have $r_n k_n \sim n$ and both $r_n$ and $k_n$ tend to infinity. 

% By \eqref{eq:DA}, for large block sizes $r_n$, each block maximum $M_{r_n,i}$ is approximately distributed according to the Fr\'echet distribution with shape parameter $\alpha_0$ and scale parameter $a_{r_n}$. Fitting this distribution to the sample of block maxima is commonly referred to as the \textit{annual maxima method}. The model suggests to define the maximum likelihood estimator $\hat \alpha_n = \hat \alpha_n(M_{r_n,1}, \dots, M_{r_n,k_n})$ and $\hat\sigma_n=\hat \sigma_n(M_{r_n,1}, \dots, M_{r_n,k_n})$ as in \eqref{eq:MLE:shape} and \eqref{eq:MLE:scale}, respectively, 
% as an estimator for $(\alpha_0,a_{r_n})$. 

% The results in the previous sections can be used to show consistency and asymptotic normality.  For that purpose, a second order condition quantifying the speed of convergence in \eqref{eq:DA} is needed. 

Equation \eqref{eq:DA} is equivalent to regular variation of $-\log F$ at infinity with index $-\alpha_0$ \citep{gnedenko:1943}: we have $F(x) < 1$ for all $x \in \R$ and
\begin{equation}
\label{eq:RV}
  \lim_{u \to \infty} \frac{- \log F(ux)}{ -\log F(u)} = x^{-\alpha_0}, \qquad x \in (0, \infty).
\end{equation}
The scaling constants in \eqref{eq:DA} may be chosen as any sequence $(a_n)_{n \in \mathbb{N}}$ that satisfies
\begin{equation}
\label{eq:scaling}
  \lim_{n \to \infty} n \, \{ -\log F(a_n) \} = 1.
\end{equation}
Being constructed from the asymptotic inverse of a regularly varying function with non-zero index, the sequence $(a_n)_{n \in \mathbb{N}}$ is itself regularly varying at infinity with index $1/\alpha_0$.

The following condition reinforces \eqref{eq:RV} and thus \eqref{eq:DA} from regular variation to second-order regular variation \citep[Section~3.6]{BGT87}. With $-\log F$ replaced by $1-F$, it appears for instance in \citet[Theorem~3.2.5]{dHF06} in the context of the asymptotic distribution of the Hill estimator. For $\tau \in \R$, define $h_\tau : (0, \infty) \to \R$ by
\begin{equation}
\label{eq:htau}
  h_\tau(x) 
  = 
  \int_1^x y^{\tau - 1} \, \diff y
  =
  \begin{cases}
    \dfrac{x^\tau - 1}{\tau}, & \text{if $\tau \neq 0$,} \\[1ex]
    \log(x), & \text{if $\tau = 0$.}
  \end{cases}
\end{equation}

\begin{condition}[Second-Order Condition]
\label{cond:secor}
There exists $\alpha_0 \in (0, \infty)$, $\rho \in (-\infty, 0]$, and a real function $A$ on $(0, \infty)$ of constant, non-zero sign such that $\lim_{u \to \infty} A(u) = 0$ and such that, for all $x \in (0, \infty)$,
\begin{equation}
\label{eq:SV:2}
  \lim_{u \to \infty} 
  \frac{1}{A(u)} 
  \left( 
    \frac{-\log F(ux)}{-\log F(u)} - x^{-\alpha_0} 
  \right) 
  = 
  x^{-\alpha_0} \, h_\rho(x).
\end{equation}
\end{condition}

The function $A$ can be regarded as capturing the speed of convergence in \eqref{eq:RV}.
The form of the limit function in \eqref{eq:SV:2} may seem unnecessarily specific, but actually, it is not, as explained in Remark~\ref{rem:secor} below.

% \begin{condition}[Second Order Condition]
% \label{cond:secor} 
% There exists a function $A : (0, \infty) \to (0, \infty)$ such that $\lim_{u \to \infty} A(u) = 0$ and such that the following limit exists and is not identically zero:
% \end{condition}

Let $\psi = \Gamma' / \Gamma$ denote the digamma function and recall the Euler--Mascheroni constant $\gamma = -\Gamma'(1) = 0.5772\ldots$. To express the asymptotic bias of the maximum likelihood estimators, we will employ the functions $b_1$ and $b_2$ defined by
\begin{equation}
\label{eq:b1}
  b_1(x) = 
  \begin{cases}
    (1+x) \, \Gamma ( x ) \{ \gamma  + \psi(1+x) \}, & \text{if $x > 0$,} \\
    \dfrac{\pi^2}{6}, & \text{if $x = 0$,}
    %\frac{ \lambda}{\alpha_0^2} \Big\{ \gamma^2 - 2\gamma + \pi^2/6 - 1 \Big\} & , \rho = 0,
  \end{cases}
\end{equation}
and
\begin{equation}
\label{eq:b2}
  b_2(x) = 
  \begin{cases} 
    -\dfrac{\pi^2}{6x} + 
      (1+x) \, \Gamma ( x )  
      \{ \Gamma''(2) + \gamma + (\gamma-1) \, \psi(1 + x)\},
    & \text{if $x > 0$,} \\
    0, & \text{if $x = 0$}.
  \end{cases}
\end{equation} 
See Figure~\ref{fig:asybias} for the graphs of these two functions.
For $(\alpha_0, \rho) \in (0, \infty) \times (-\infty, 0]$, define the bias function
\begin{equation}
\label{eq:bias:iid}
  B(\alpha_0, \rho) 
  = 
  - \frac{6}{\pi^2} 
  \begin{pmatrix} 
    b_1(\abs{\rho}/\alpha_0) \\ 
    b_2(\abs{\rho}/\alpha_0) / \alpha_0^2
  \end{pmatrix}.
\end{equation}
The proof of the following theorem is given in Section~\ref{subsec:proofs:iid}.

\begin{theorem}
\label{theo:ml}
Let $\xi_1, \xi_2, \ldots$ be independent random variables with common distribution function $F$ satisfiying Condition~\ref{cond:secor}.
% and that, without loss of generality, the constant $\kappa$ in Remark~\ref{rem:secor} is equal to $1$.
Let the block sizes $r_n$ be such that $r_n \to \infty$ and $k_n = \floor{ n / r_n } \to \infty$ as $n \to \infty$ and assume that 
\begin{equation}
\label{eq:ka}
  \lim_{n\to \infty} \sqrt{k_n} \, A(a_{r_n}) = \lambda \in \R.
\end{equation}
Then, with probability tending to one, there exists a unique maximizer $(\hat{\alpha}_n, \hat{\sigma}_n)$ of the Fr\'echet log-likelihood \eqref{eq:loglik} based on the block maxima $M_{r_n,1}, \ldots, M_{r_n,k_n}$, and we have
\begin{equation}
\label{eq:ml}
  \sqrt{k_n} 
  \begin{pmatrix}
    \hat{\alpha}_n - \alpha_0 \\
    \hat{\sigma}_n / a_{r_n} - 1
  \end{pmatrix}
  \weak 
  \Nc_2 \left( \lambda \, B(\alpha_0, \rho), \, I_{(\alpha_0,1)}^{-1} \right),
  \qquad n \to \infty,
\end{equation}
where $I_{(\alpha_0,1)}^{-1}$ denotes the inverse of the Fisher information of the Fr\'echet family as in~\eqref{eq:fisher}
% , where $\rho$ is as in Remark~\ref{rem:secor} 
and with $B( \alpha_0, \rho )$ as in \eqref{eq:bias:iid}.
%where $B(\alpha_0) =(B_1(\alpha_0), B_2(\alpha_0))^T$  with
%\begin{align*}
%B_1(\alpha_0) &= 
%\begin{cases}
%\frac{6\lambda}{\pi^2} \left(1+\tfrac{\abs \rho}{\alpha_0}\right) \Gamma \left( \frac{\abs{\rho}}{\alpha_0}\right)  
%\left\{ \gamma  +\psi\left(1+ \frac{\abs{\rho}}{\alpha_0}\right) \right\} & , \rho <0, \\
%\lambda &, \rho=0
%%\frac{ \lambda}{\alpha_0^2} \Big\{ \gamma^2 - 2\gamma + \pi^2/6 - 1 \Big\} & , \rho = 0,
%\end{cases} \\
%B_2(\alpha_0) &= 
%\begin{cases}
% \frac{-\lambda}{ \alpha_0^2} \frac{\alpha_0}{\abs \rho} + \frac{6 \lambda}{\pi^2 \alpha_0^2}  \left(1+\tfrac{\abs \rho}{\alpha_0}\right) \Gamma \left( \frac{\abs{\rho}}{\alpha_0}\right)  
% \left\{ \Gamma''(2) + \gamma + (\gamma-1)  \psi\left(1+ \frac{\abs{\rho}}{\alpha_0}\right) \right\}&, \rho <0, \\
%0 &, \rho=0.
%\end{cases}
%\end{align*} 
\end{theorem}

We conclude this section with a series of remarks on the second-order Condition~\ref{cond:secor} and its link to the block-size condition in~\eqref{eq:ka} and the mean vector of the limiting distribution in~\eqref{eq:ml}.

\begin{figure}
\centering
\vspace{-0.8cm}
\includegraphics[width=0.48\textwidth]{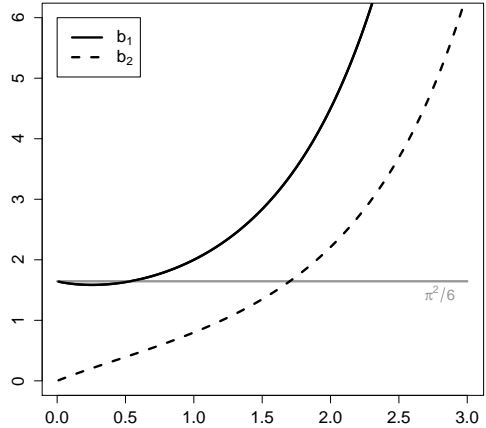}
\vspace{-.9cm}
\caption{Graphs of the functions $b_1$ and $b_2$ in \eqref{eq:b1} and \eqref{eq:b2}.}
\label{fig:asybias}
\vspace{-.4cm}
\end{figure}

\begin{remark}[Second-order regular variation]
\label{rem:secor}
Let $F$ satisfy \eqref{eq:RV}. For $x>0$ sufficiently large such that $F(x)>0$, define $L(x)$ by
\begin{align}
\label{eq:defL}
  - \log F(x) = x^{-\alpha_0} \, L(x).
\end{align}
In view of \eqref{eq:RV}, the function $L$ is slowly varying at infinity, that is,
\begin{equation*}
%\label{eq:SV}
  \lim_{u \to \infty} \frac{L(ux)}{L(u)} = 1, \qquad x \in (0, \infty).
\end{equation*}
A second-order refinement of this would be that there exist $A : (0, \infty) \to (0, \infty)$ and $h : (0, \infty) \to \R$, the latter not identically zero, such that $\lim_{u \to \infty} A(u) = 0$ and
\begin{equation*}
%\label{eq:SV:2}
  \lim_{u \to \infty} \frac{1}{A(u)} \left( \frac{L(ux)}{L(u)} - 1 \right) = h(x), \qquad x \in (0, \infty).
\end{equation*}
Writing $g(u) = A(u) \, L(u)$, Theorem~B.2.1 in \cite{dHF06} (see also \citealp{BGT87}, Section~3.6) implies that there exists $\rho \in \R$ such that $g$ and thus $A = g/L$ are regularly varying at infinity with index~$\rho$. Since $A$ vanishes at infinity, necessarily $\rho \le 0$. Furthermore, there exists $\kappa \in \R \setminus \{ 0 \}$ such that $h(x) = \kappa \, h_\rho(x)$ for $x \in (0, \infty)$, with $h_\rho$ as in \eqref{eq:htau}.
% \[ 
%   h_\rho(x) = \int_1^x y^{\rho-1} dy =
%   \begin{cases}
%     \dfrac{x^{\rho} - 1}{\rho} & \text{if $\rho < 0$,} \\
%     \log(x) & \text{if $\rho = 0$.}
%   \end{cases}
% \]
Incorporating the constant $\kappa$ into the function $A$, we can assume without loss of generality that $\kappa=1$ and we arrive at Condition~\ref{cond:secor}. The function $A$ then possibly takes values in $(-\infty,0)$ rather than in $(0, \infty)$.
\end{remark}

\begin{remark}[Asymptotic mean squared error]
According to \eqref{eq:ka} and \eqref{eq:ml}, the distribution of the estimation error $\hat{\alpha}_n - \alpha_0$ is approximately equal to
\[
  \mathcal{N} 
  \left( 
    -A(a_{r_n}) \, \frac{6}{\pi^2} \, b_1( \abs{\rho}/\alpha_0 ), \;
   \frac{r_n}{n} \, \frac{6}{\pi^2} \, \alpha_0^2
  \right).
\]
The asymptotic mean squared error is therefore equal to
\[
  \operatorname{AMSE}( \hat{\alpha}_n )
  =
  \operatorname{ABias}^2( \hat{\alpha}_n )
  +
  \operatorname{AVar}( \hat{\alpha}_n )
  =
  \abs{A(a_{r_n})}^2 \, \frac{36}{\pi^4} \, b_1( \abs{\rho}/\alpha_0 )^2 
  +
 \frac{r_n}{n} \frac{6}{\pi^2} \alpha_0^2.
\]
The choice of the block size $r_n$ (or, equivalently, the number of blocks $k_n$), thus involves a bias--variance trade-off; see Section~\ref{sec:ex}. Alternatively, if $\rho$ and $A(a_{r_n})$ could be estimated, then one could construct bias-reduced estimators, just as in the case of the Hill estimator (see, e.g., \citealp{Pen98}, among others) or probability weighted moment estimators \citep{CaiDehZho13}.
\end{remark}

\begin{remark}[On the number of blocks]
A version of \eqref{eq:ka} is used in \cite{dHF15} to prove asymptotic normality of probability weighted moment estimators. Equation~\eqref{eq:ka} also implies the following limit relation, which is imposed in \cite{Dom15} and which we will be needing later on as well:
\begin{equation}
\label{eq:klogk}
  \lim_{n \to \infty} \frac{k_n \log(k_n)}{n} = 0.
\end{equation}
Indeed, in view of Remark~\ref{rem:secor} and regular variation of $(a_n)_{n \in \mathbb{N}}$, the sequence $(|A(a_r)|)_{r \in \mathbb{N}}$ is regularly varying at infinity. Potter's theorem \citep[Theorem~1.5.6]{BGT87} then implies that there exists $\beta > 0$ such that $r^{-\beta} = o(|A(a_r)|)$ as $r \to \infty$. But then $\sqrt{k_n} (r_n)^{-\beta} = \sqrt{k_n} \, o(|A(a_{r_n})|) = o(1)$ by \eqref{eq:ka} and thus $k_n^{1/2+\beta} / n^\beta = o(1)$ as $n \to \infty$. We obtain that $k_n^{1+1/(2\beta)}/n = o(1)$, which implies~\eqref{eq:klogk}.
\end{remark}

\begin{remark}[No asymptotic bias]
If $\lambda = 0$ in \eqref{eq:ka}, then the limiting normal distribution in \eqref{eq:ml} is centered and the maximum likelihood estimator is said to be asymptotically unbiased. If the index, $\rho$, of regular variation of the auxiliary function $\lvert A \rvert$ is strictly negative (see Remark~\ref{rem:secor}), then a sufficient condition for $\lambda = 0$ to occur is that $k_n = O( n^{\beta} )$ for some $\beta < \lvert\rho\rvert / (\alpha_0/2 + \lvert\rho\rvert)$.
\end{remark}

\section{Examples and finite-sample results}
\label{sec:ex}

\subsection{Verification of conditions in a moving maximum model} \label{subsec:movmax}

For many stationary time series models, the distribution of the sample maximum is a difficult object to work with. This is true even for linear time series models, since the maximum operator is non-linear. In such cases, checking the conditions of Section~\ref{sec:stat} may be a hard or even impossible task. An exception occurs for moving maximum models, where the sample maximum can be linked directly to maxima of the innovation sequence.

Let $(Z_t)_{t\in \Z}$ be a sequence of independent and identically distributed random variables with common distribution function $F$ in the domain of attraction of the Fr\'echet distribution with shape parameter $\alpha_0>0$, that is, such that \eqref{eq:DA} is satisfied for some sequence $a_n \to \infty$. Let $p \in \N$, $p \ge 2$, be fixed and let $b_1, \dots, b_p$ be nonnegative constants, $b_1 \ne 0 \ne b_p$, such that $\sum_{i=1}^p b_i=1$.  We consider the moving maximum process $\xi_t$ of order $p$, defined by
\[
\xi_t = \max\{b_1 Z_t, b_2 Z_{t-1}, \dots, b_p Z_{t-p+1}\}, \qquad t\in \Z.
\]
A simple calculation (see also the proof of Lemma~\ref{lem:movmax} for the stationary distribution of $\xi_t$) shows that the extremal index of $(\xi_t)_{t\in\Z}$ is equal to
\[
\theta= \left\{ \textstyle \sum_{i=1}^p b_i^{\alpha_0} \right\}^{-1} {b_{(p)}^{\alpha_0}} ,
\]
where $b_{(p)}= \max_{i=1}^p b_i$.
Let $\sigma_n= b_{(p)} a_n$.   The proof of the following lemma is given in Section~\ref{subsec:proofs:ex} in the supplementary material.

\begin{lemma} \label{lem:movmax}
The stationary time series $(\xi_t)_{t\in\Z}$ satisfies Conditions~\ref{cond:DA}, \ref{cond:alpha} and \ref{cond:moment}. If additionally \eqref{eq:klogk} is met, then Condition~\ref{cond:small} is satisfied as well. Finally, if~$F$ satisfies the Second-Order Condition \ref{cond:secor},  if \eqref{eq:ka} is met and if $k_n=o(n^{2/3})$ as $n \to \infty$,  then Condition~\ref{cond:bias} is also satisfied, with $B(f)$ denoting the same limit as  in the iid case, that is, $B(f) = \bm \beta$ with $\bm \beta$ as  in \eqref{eq:beta}.
\end{lemma}

As a consequence, Theorem \ref{theo:asydis:stat} may be applied and the asymptotic bias of the maximum likelihood estimator is the same as specified in Theorem~\ref{theo:ml} for the case of independent and identically distributed random variables.

\subsection{Simulation results}
\label{subsec:simul}
We report on the results of a simulation study, highlighting some interesting features regarding the finite-sample performance of the maximum likelihood estimator. Attention is restricted to the estimation of the shape parameter, and particular emphasis is given to a comparison with the common Hill estimator, which is based on the competing peaks-over-threshold method. Its variance is of the order $O(k^{-1})$, where $k$ is the number of upper order statistics taken into account for its calculation. The Hill estimator's asymptotic variance is given by $\alpha_0^2$, which is larger than the asymptotic variance $(6/\pi^2) \times \alpha_0^2$ of the block maxima maximum likelihood estimator. Furthermore, numerical experiments (not shown) involving the probability weighted moment estimator showed a variance that was higher, in all cases considered, than the one of the maximum likelihood estimator.

We consider three time series models for $(\xi_t)_{t\in\Z}$: independent and identically distributed random variables, the moving maximum process from Section~\ref{subsec:movmax}, and the absolute values of a GARCH(1,1) time series. In the first two models, three choices are considered for the distribution function $F$ of either the variables $\xi_t$ in the first model and the innovations $Z_t$ in the second model: absolute values of a Cauchy-distribution, the standard Pareto distribution and the Fr\'echet(1,1) distribution itself. All three distribution functions are attracted to the Fr\'echet distribution with $\alpha_0=1$. For the moving maximum process, we fix $p=4$ and $b_j=j/10$ for $j\in\{1,2,3,4\}$. The GARCH(1,1) model is based on standard normal innovations, that is, $\xi_t = \abs{Z_t}$, where $Z_t$ is the stationary solution of the equations
\begin{equation}
\label{eq:GARCH}
  \left\{
    \begin{array}{rcl}
      Z_t &=& \varepsilon_t \sigma_t , \\
      \sigma_t^2 &=& \lambda_0 + \lambda_1 Z_{t-1}^2 + \lambda_2 \sigma_{t-1}^2,
    \end{array}
  \right.
\end{equation}
with $\varepsilon_t$, $t \in \Z$, independent standard normal random variables.  The parameter vector $(\lambda_0, \lambda_1, \lambda_2)$ is set to either $(0.5, 0.367, 0.367)$ or $(0.5, 0.08, 0.91)$. By \cite{MikSta00}, the stationary distribution associated to any of these two models is attracted to the Fr\'echet distribution with shape parameter being (approximately) equal to $\alpha_0=5$.  

We generate samples from all of the afore-mentioned models for a fixed sample size of $n=1\,000$. Based on $N=3\,000$ Monte Carlo repetitions, we obtain empirical estimates of the finite sample bias, variance and mean squared error (MSE) of the competing estimators.  The results are summarized in Figure~\ref{fig:mse} for the iid and the moving maxima model, and in Figure~\ref{fig:mse2} for the GARCH-model. Additional details for the case of independent random sampling from the absolute value of a Cauchy distribution are provided in the Supplement, Section~\ref{sec:simulextra}.

In general, (most of) the graphs nicely reproduce the bias-variance tradeoff, its characteristic form however varying from model to model. Consider the iid scenario:
since the Hill estimator is essentially the maximum likelihood estimator in the Pareto family, it is to be expected that it outperforms the block maxima estimator. On the other hand, by max-stability of the Fr\'echet family, the block maxima estimator should outperform the Hill estimator for that family.  These expectations are confirmed by the simulation results in the left column of Figure~\ref{fig:mse}. For the Cauchy distribution, it turns out that the block maxima maximum likelihood estimator shows a better performance.

Now, consider the moving maxima time series scenarios (right column in Figure~\ref{fig:mse}). Compared to the iid case, we observe an increase in the mean squared error (note that the scale on the axis of ordinates is row-wise identical). The block maxima method clearly outperforms the Hill estimator, except for the Pareto model. The big increase in relative performance is perhaps not too surprising, as the data points from a moving maximum process are already  (weighted) maxima, which principally favors the block maxima method with small block sizes.

Finally, consider the GARCH models in Figure~\ref{fig:mse2}. While, as in line with the theoretical findings, the variance of the block maxima estimator is smaller than the one of the Hill estimator, the squared bias turns out to be substantially higher for a large range of values for $k$. The MSE-optimal point is smaller for the Hill estimator.

\begin{figure}[h!]
\begin{center}
%\hspace*{-2.5cm}
\vspace{-0.5cm}

\mbox{
\includegraphics[width=0.44\textwidth]{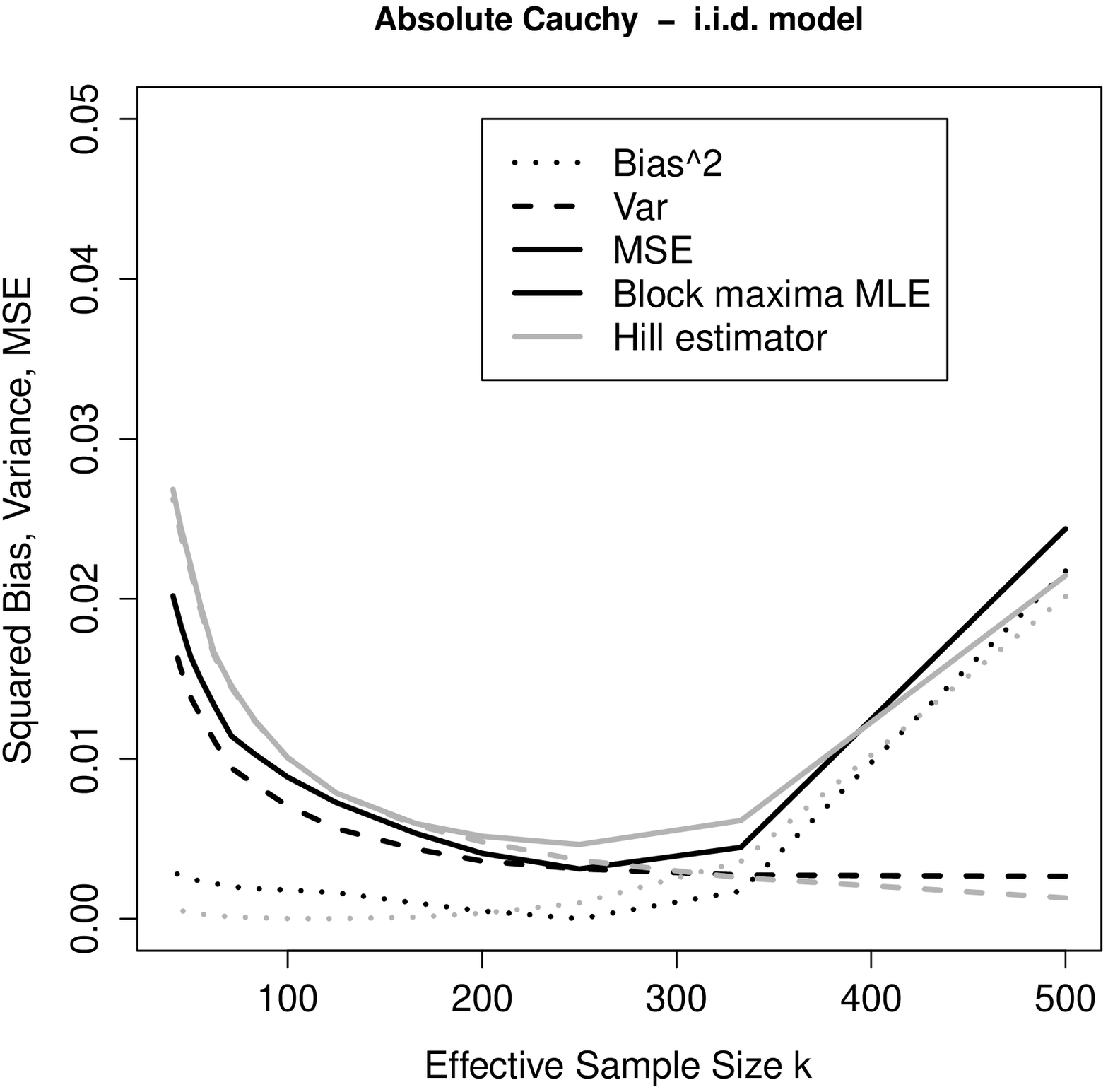}
\includegraphics[width=0.44\textwidth]{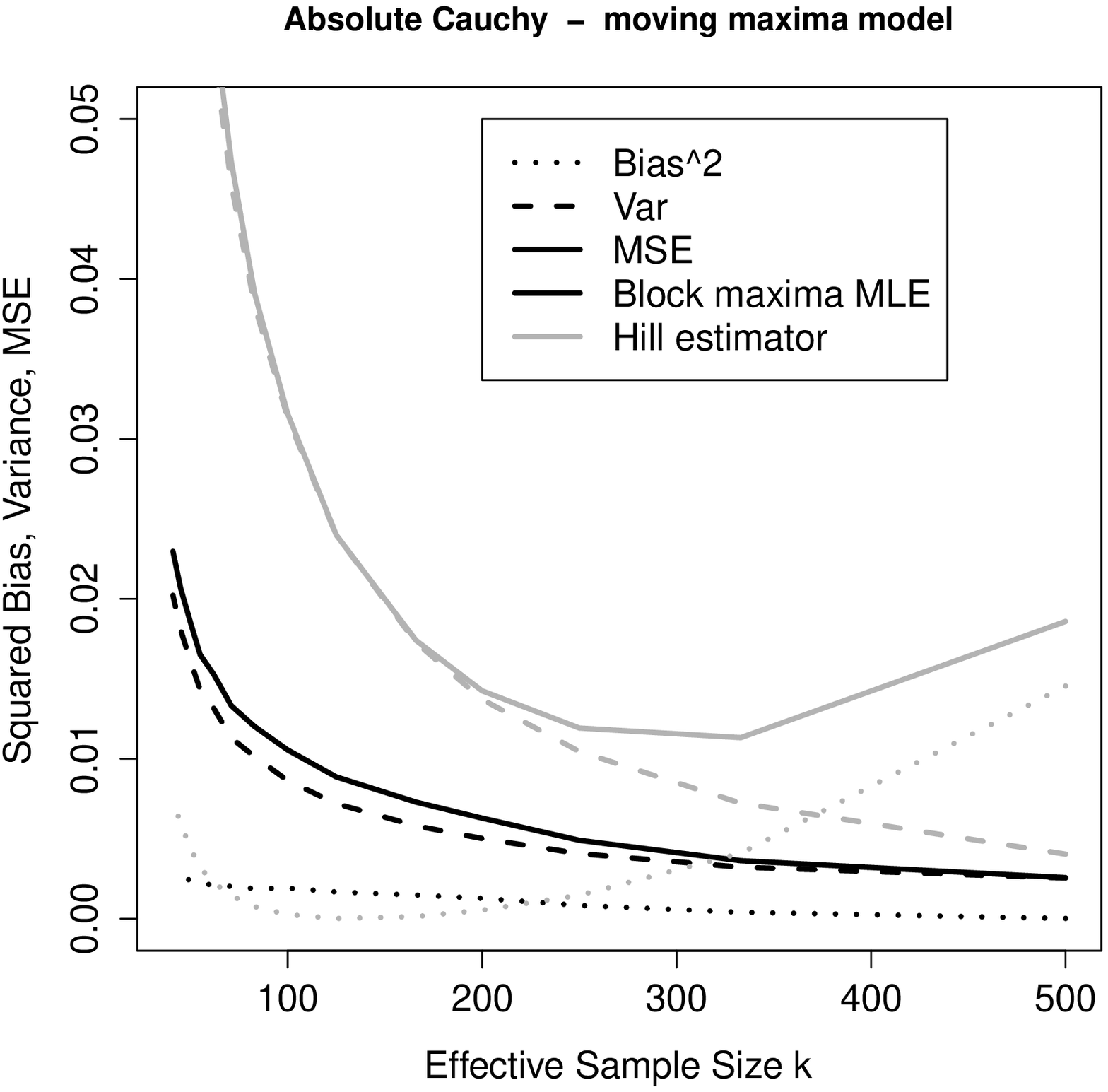}
}
\vspace{-0.2cm}

%\hspace*{-2.5cm}
\mbox{
\includegraphics[width=0.44\textwidth]{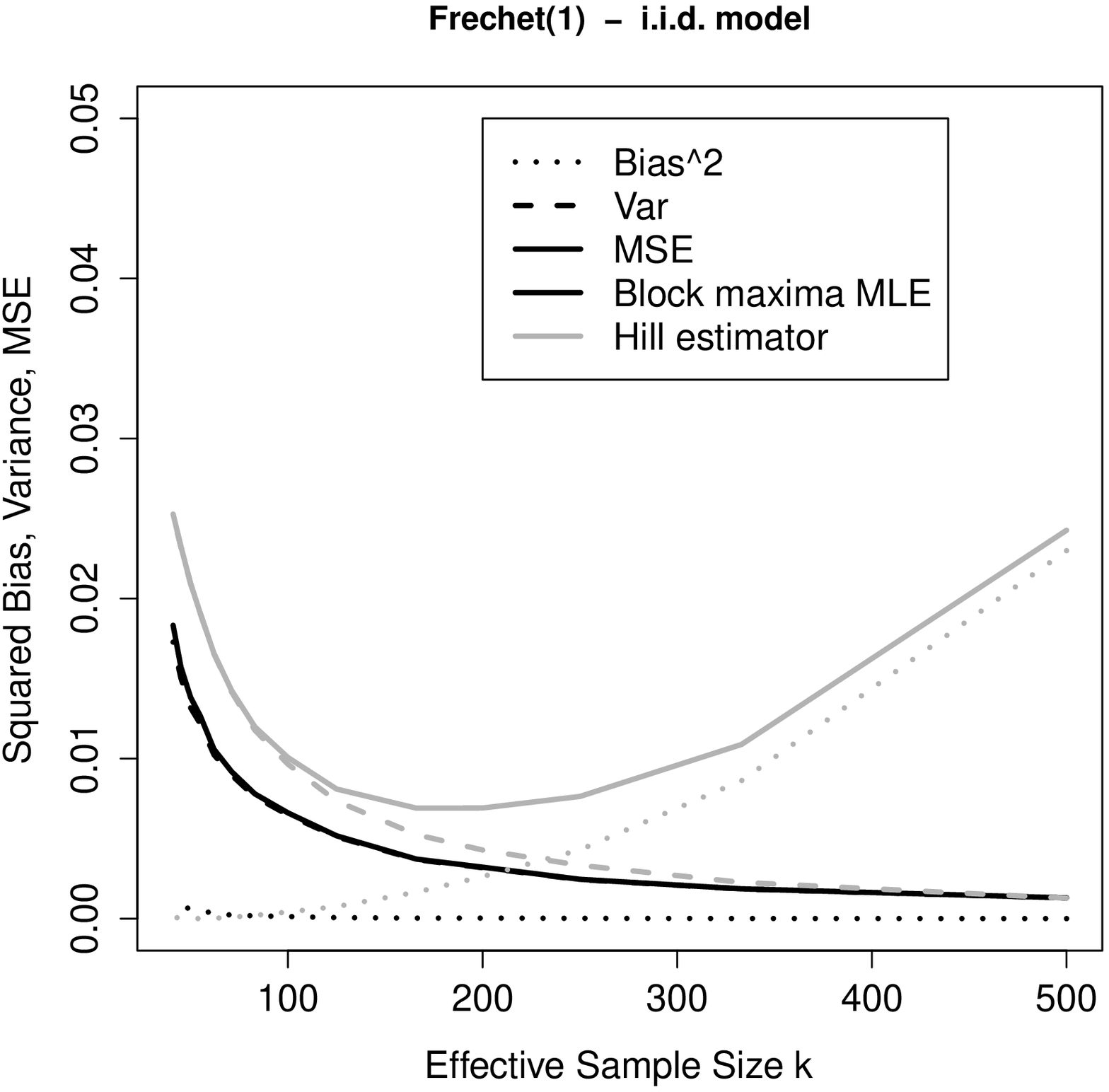}
\includegraphics[width=0.44\textwidth]{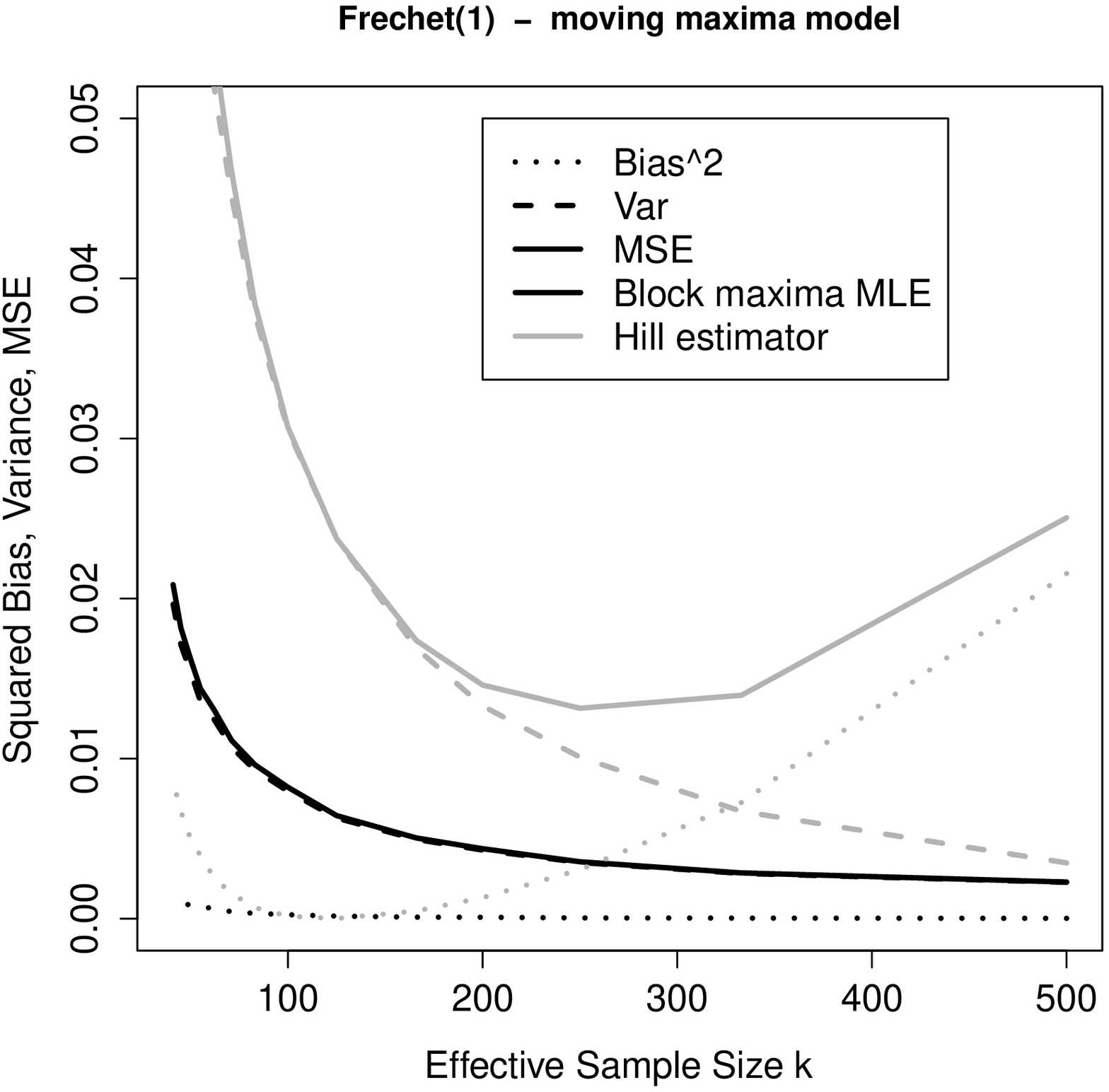}
}
\vspace{-0.2cm}

%\hspace*{-2.5cm}
\mbox{\includegraphics[width=0.44\textwidth]{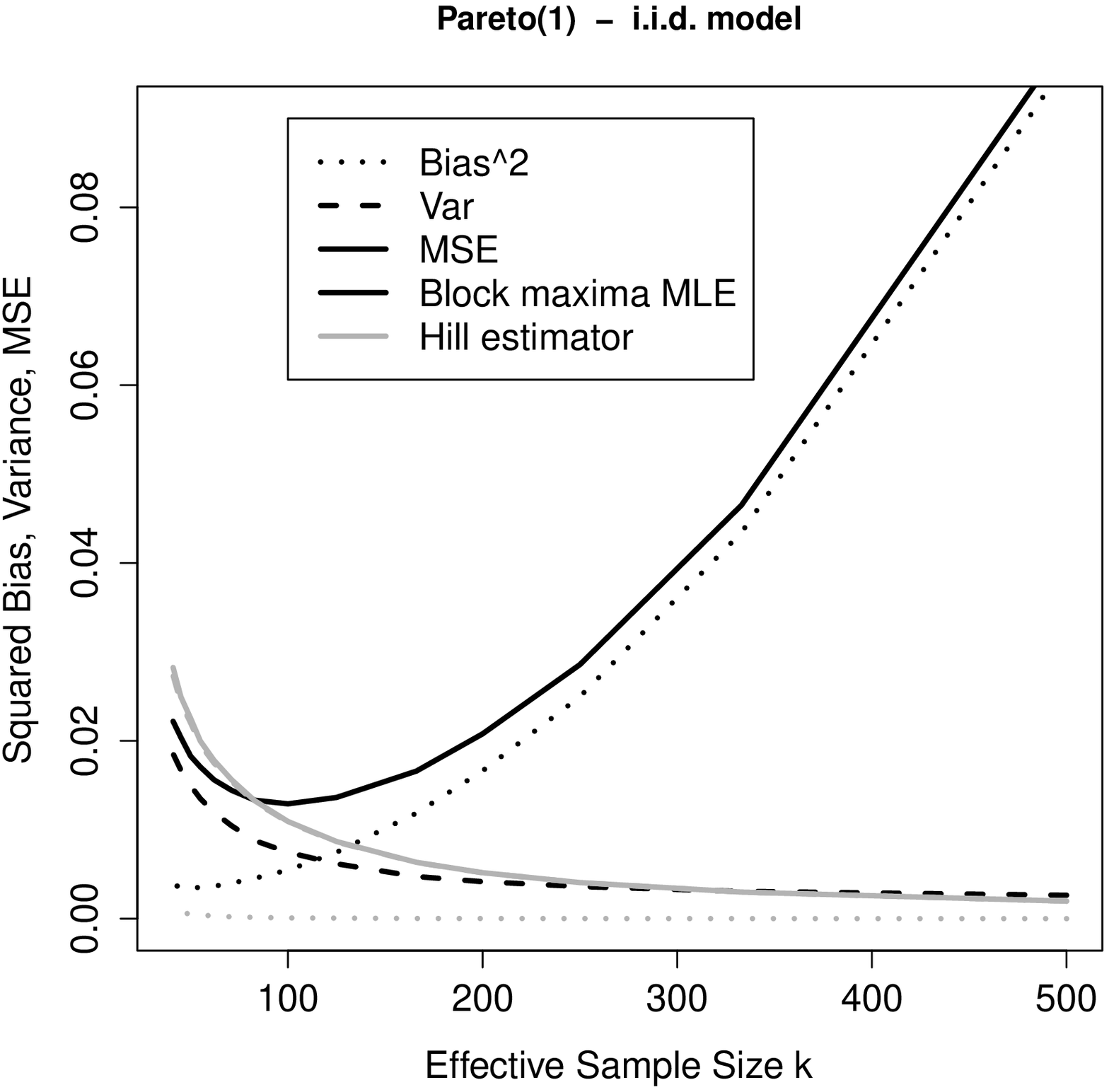}
\includegraphics[width=0.44\textwidth]{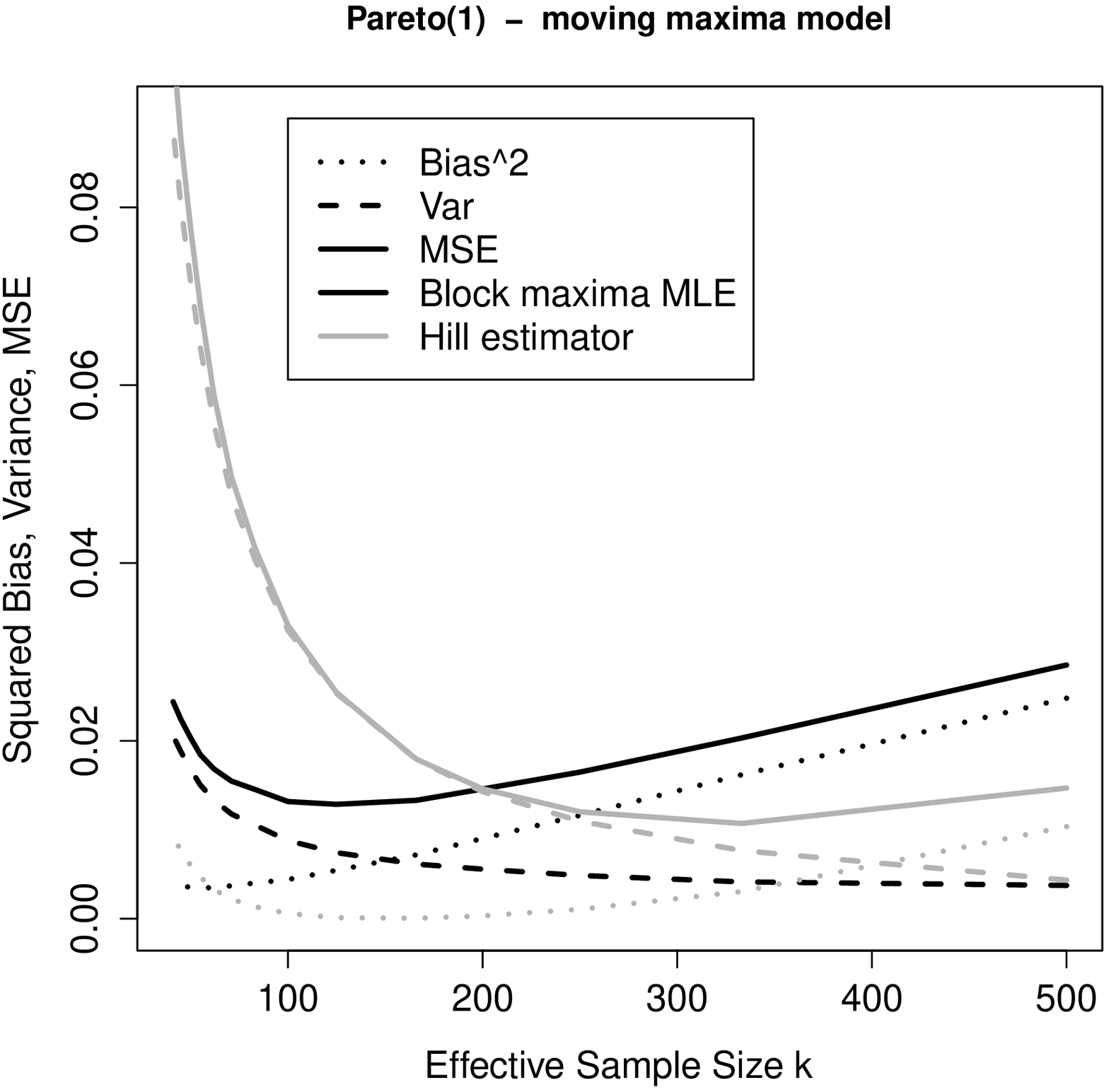}
}
\end{center}

\vspace{-0.65cm}
\caption{Simulation results (Section~\ref{subsec:simul}). Effective sample size refers to the number of blocks (block maxima MLE) or the number of upper order statistics (Hill estimator). Time series models:  iid (left) and moving maximum model (right). Innovations: absolute values of Cauchy (top), unit Fr\'echet (middle) and unit Pareto (bottom) random variables. Block sizes $r \in \{2, 3, \dots , 24\}$, resulting in $k\in \{500, 333, \dots, 41\}$ blocks.}
\label{fig:mse}
\vspace{-.3cm}
\end{figure}

\begin{figure}[h!]
\begin{center}
%\hspace*{-2.5cm}
\includegraphics[width=0.45\textwidth]{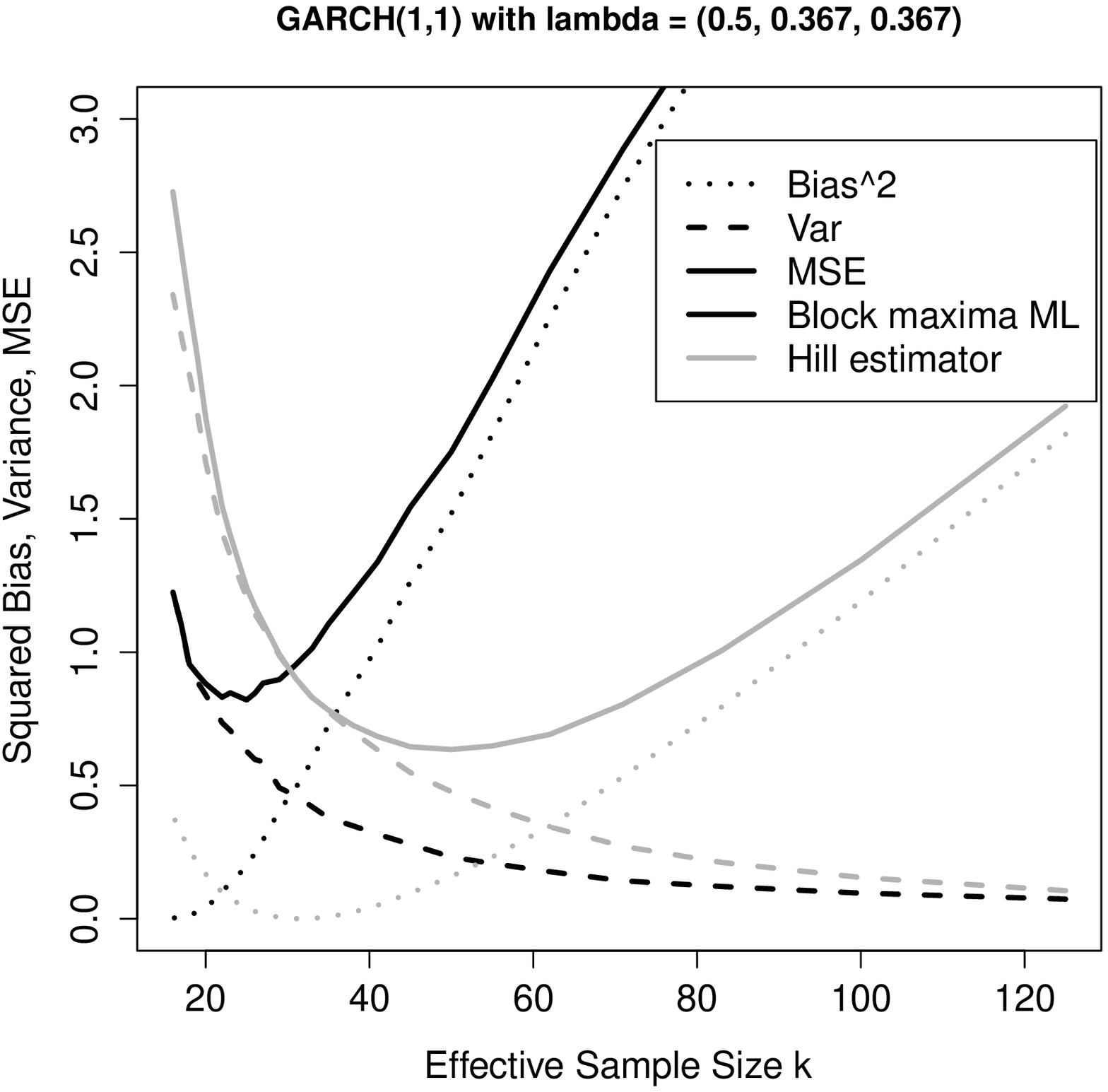}
\includegraphics[width=0.45\textwidth]{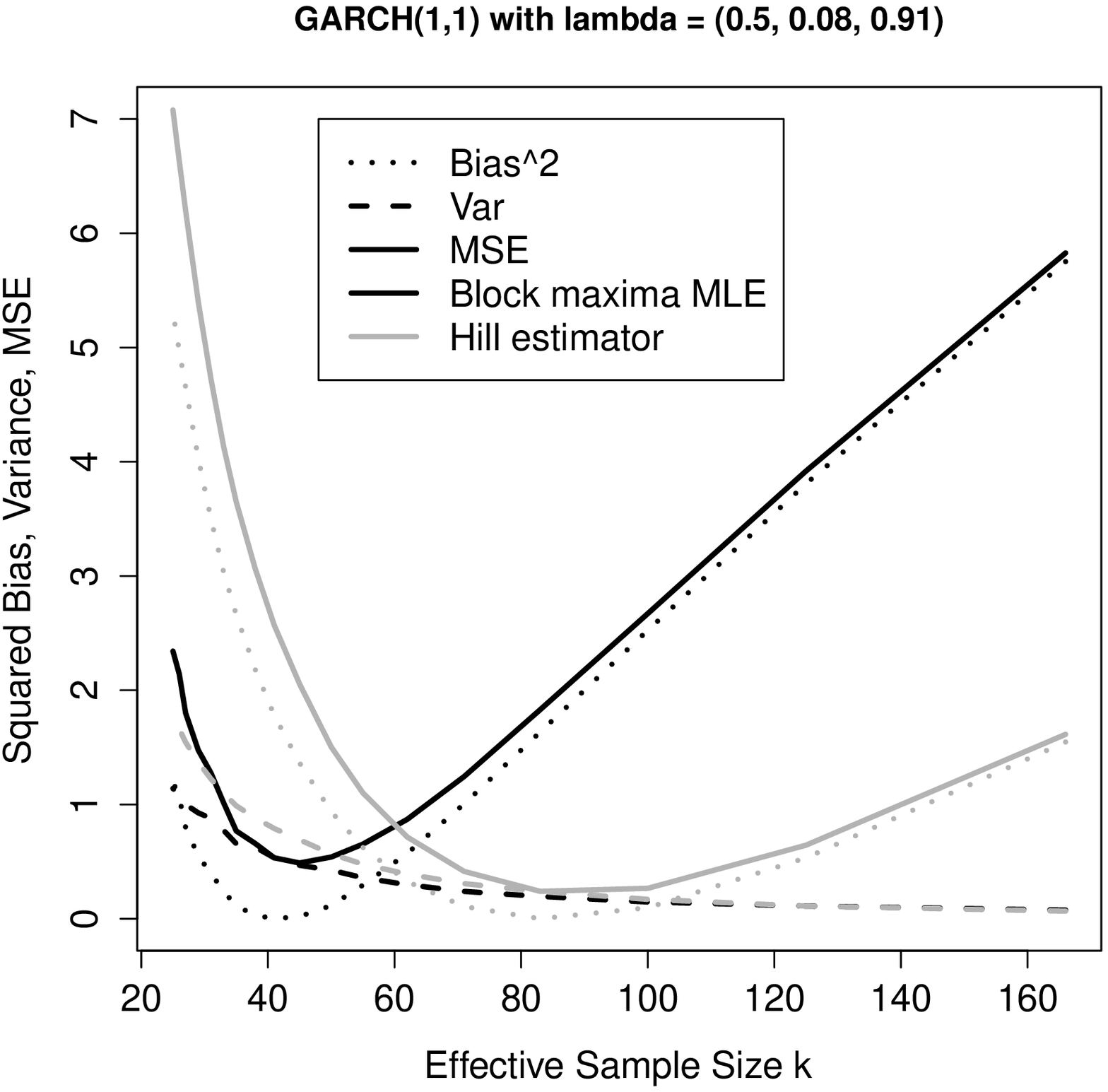}
\end{center}
\vspace{-0.7cm}

\caption{Simulation results (Section~\ref{subsec:simul}). Effective sample size refers to the number of blocks (block maxima MLE) or the number of upper order statistics (Hill estimator). 
Both panels refer to the GARCH(1,1) model in \eqref{eq:GARCH}, with $(\lambda_0, \lambda_1, \lambda_2)$ equal to $(0.5, 0.367, 0.367)$ on the left and to $(0.5, 0.08, 0.91)$ on the right.}
\label{fig:mse2}
\vspace{-.3cm}
\end{figure}

\appendix

\section{Proofs}
\label{sec:proofs}

\subsection{Proofs for Section~\ref{sec:tri}}
\label{subsec:proofs:tri}

\begin{proof}[Proof of Lemma~\ref{lem:existsUnique}]
The proof extends the development in Section~2 of \cite{BalKat08}. First, fix $\alpha > 0$ and consider the function $0 < \sigma \mapsto L(\alpha, \sigma \mid \bm{x})$. By Equation~\eqref{eq:score:sigma}, its derivative is equal to
\[
  \partial_\sigma L(\alpha, \sigma \mid \bm{x})
  = \sum_{i=1}^k \partial_\sigma \ell_\theta( x_i )
  = (\alpha/\sigma) \left(k - \sigma^\alpha \sum_{i=1}^k x_i^{-\alpha}\right).
\]
We find that $\partial_\sigma L(\alpha, \sigma \mid \bm{x})$ is positive, zero, or negative according to whether $\sigma$ is smaller than, equal to, or larger than ${\sigma}( \alpha \mid \bm{x} )$, respectively. In particular, for fixed $\alpha$, the expression $L(\alpha, \sigma \mid \bm{x})$ is maximal at $\sigma$ equal to ${\sigma}( \alpha \mid \bm{x} )$. Hence we need to find the maximum of the function $0 < \alpha \mapsto L( \alpha, {\sigma}(\alpha \mid \bm{x}) \mid \bm{x} )$. By \eqref{eq:score:alpha}, its derivative is given by
\begin{align*}
  \frac{d}{d \alpha} L( \alpha, {\sigma}(\alpha \mid \bm{x}) \mid \bm{x} )
%   &= \sum_{i=1}^k \frac{d}{d \alpha} \ell_{\alpha, \hat{\sigma}(\alpha \mid \bm{x})}( x_i ) 
  = \sum_{i=1}^k \partial_\alpha \ell_{\alpha, \sigma}(x_i) \bigg|_{\sigma = {\sigma}(\alpha \mid \bm{x})}
  + \sum_{i=1}^k \partial_\sigma \ell_{\alpha, \sigma}(x_i) \biggl|_{\sigma = {\sigma}(\alpha \mid \bm{x})} \times \frac{d}{d\alpha} {\sigma}( \alpha \mid \bm{x} ).
\end{align*}
The second sum is equal to zero, by definition of ${\sigma}( \alpha \mid \bm{x} )$. We obtain
\[
  \frac{d}{d \alpha} L( \alpha, {\sigma}(\alpha \mid \bm{x}) \mid \bm{x} )
  = k \, \Psi_k( \alpha \mid \bm{x} ),
\]
with $\Psi_k$ as in \eqref{eq:Psik}. This is the same expression as Eq.~(2.3) in \cite{BalKat08}, with their $x_i$ replaced by our $x_i^{-1}$. 
Differentiating once more with respect to $\alpha$, we obtain that
\begin{align}
\label{eq:ddLalpha}
  \frac{d^2}{d \alpha^2} L( \alpha, {\sigma}(\alpha \mid \bm{x}) \mid \bm{x} ) 
  = - \frac{k}{\alpha^2} - k \frac{\sum_{i=1}^k x_i^{-\alpha} (\log(x_i))^2 \sum_{i=1}^k x_i^{-\alpha} - \left( \sum_{i=1}^k x_i^{-\alpha} \log(x_i) \right)^2}{\left( \sum_{i=1}^k x_i^{-\alpha} \right)^2}.
\end{align}
By the Cauchy--Schwartz inequality, the numerator of the big fraction is nonnegative, whence
\[
  \frac{d^2}{d \alpha^2} L( \alpha, {\sigma}(\alpha \mid \bm{x}) \mid \bm{x} )
  \le - \frac{k}{\alpha^2} < 0.
\]
Hence, $\alpha \mapsto \Psi_k(\alpha\mid \bm x)$ is strictly decreasing.
For $\alpha \to 0$, this function diverges to $\infty$, whereas for $\alpha \to \infty$, it converges to $\log( \min(x_1, \ldots, x_k) ) - k^{-1}\sum_{i=1}^k \log(x_i)$, which is less than zero given the assumptions on $x_1, \ldots, x_k$. Hence, there exists a unique $\hat{\alpha}( \bm{x} ) \in (0, \infty)$ such that this function is zero. We conclude that the function $\theta \mapsto L(\theta \mid \bm{x})$ admits a unique maximum at $\hat{\theta}( \bm{x} )$.
\end{proof}

Fix $\alpha_0 \in (0, \infty)$. Let $P$ denote the Fr\'echet distribution with parameter $\theta_0 = (\alpha_0, 1)$, with support $\mathcal{X} = (0, \infty)$. The tentative limit of the functions $\alpha \mapsto \Psi_k( \alpha \mid \bm{x} )$ is the function
\[
  \Psi( \alpha ) 
  = \frac{1}{\alpha} + \frac{\int_0^\infty x^{-\alpha} \log(x) \, \diff P(x)}{\int_0^\infty x^{-\alpha} \, \diff P(x)} 
  - \int_0^\infty \log(x)\, \diff P(x).
\]
Let $\Gamma$ be the gamma function and let $\psi = \Gamma' / \Gamma$ be the digamma function. 

\begin{lemma}
\label{lem:Psi}
Fix $\alpha_0 \in (0, \infty)$. We have
\begin{equation}
\label{eq:Psi}
  \Psi(\alpha) = \frac{1}{\alpha_0} \bigl( \psi(1) - \psi(\alpha/\alpha_0) \bigr), 
  \qquad \alpha \in (0, \infty).
\end{equation}
As a consequence, $\Psi : (0, \infty) \to \R$ is a decreasing bijection with a unique zero at $\alpha = \alpha_0$.
\end{lemma}

\begin{proof}[Proof of Lemma~\ref{lem:Psi}]
By Lemma~\ref{lem:moments},
\begin{align*}
  \Psi(\alpha)
  &= \frac{1}{\alpha} 
  + \frac{(-\alpha_0^{-1}) \Gamma'(1 + \alpha/\alpha_0)}{\Gamma(1 + \alpha/\alpha_0)} 
  - (-\alpha_0^{-1}) \Gamma'(1) 
  = \frac{1}{\alpha_0}
  \left(
    (\alpha/\alpha_0)^{-1} - \psi(1 + \alpha/\alpha_0) + \psi(1)
  \right).
\end{align*}
The digamma function satisfies the recurrence relation $\psi(x+1) = \psi(x) + \frac{1}{x}$. Equation~\eqref{eq:Psi} follows. The final statement follows from the fact that the digamma function $\psi : (0, \infty) \to \R$ is an increasing bijection.
\end{proof}

\begin{proof}[Proof of Theorem~\ref{theo:consistency}]
By Lemma~\ref{lem:existsUnique}, we only have to show the claimed convergence.
Define a random function $\Psi_n$ on $(0, \infty)$ by
\begin{equation}
\label{eq:Psin}
  \Psi_n( \alpha ) 
  = \Psi_{k_n}( \alpha \mid \bm{X}_n ) = \Psi_{k_n}( \alpha \mid \bm{X}_n / \sigma_n ),
\end{equation}
with $\Psi_k(\cdot|\cdot)$ as in \eqref{eq:Psik}. Recall $\Psi$ in \eqref{eq:Psi}. 
The hypotheses imply that, for each $\alpha \in (\alpha_-, \alpha_+)$,
\[
  \Psi_n( \alpha ) \weak \Psi( \alpha ), \qquad n \to \infty.
\]
By Lemma~\ref{lem:Psi}, the limit $\Psi( \alpha )$ is positive, zero, or negative according to whether $\alpha$ is less than, equal to, or greater than $\alpha_0$. Moreover, the function $\Psi_n$ is decreasing and $\Psi_n( \hat{\alpha}_n ) = 0$; see the proof of Lemma~\ref{lem:existsUnique}.

Let $\delta > 0$ be such that $\alpha_- < \alpha_0-\delta < \alpha_0 + \delta < \alpha_+$. Since $\Psi_n( \alpha_0 - \delta ) \weak \Psi( \alpha_0 - \delta ) > 0$ as $n \to \infty$, we find
\begin{align*}
  \Pr[ \hat{\alpha}_n \le \alpha_0 - \delta ]
  \le \Pr[ \Psi_n( \alpha_0 - \delta ) \le 0 ] 
  \to 0, \qquad n \to \infty.
\end{align*}
Similarly, $\Pr[ \hat{\alpha}_n \ge \alpha_0 + \delta ] \to 0$ as $n \to \infty$. We can choose $\delta > 0$ arbitrarily small, thereby concluding that $\hat{\alpha}_n \weak \alpha_0$ as $n \to \infty$.

Second, Condition~\ref{cond:LLN} also implies that, for each $\alpha \in (\alpha_-, \alpha_+)$ and as $n \to \infty$,
\begin{align*}
  \frac{1}{\sigma_n} \left( \frac{1}{k_n} \sum_{i=1}^{k_n} X_{n,i}^{-\alpha} \right)^{-1/\alpha}
  &= \left( \frac{1}{k_n} \sum_{i=1}^{k_n} (X_{n,i}/\sigma_n)^{-\alpha} \right)^{-1/\alpha}  \\
  &\weak \left( \int_0^\infty x^{-\alpha} \, p_{\alpha_0,1}(x) \, \diff x \right)^{-1/\alpha} 
  = \bigl( \Gamma(1 + \alpha/\alpha_0) \bigr)^{-1/\alpha},
\end{align*}
where we used Lemma~\ref{lem:moments} for the last identity. Both the left-hand and right-hand sides are continuous, nonincreasing functions of $\alpha$. Since $\hat{\alpha}_n \weak \alpha_0$ as $n \to \infty$ and since the right-hand side evaluates to unity at $\alpha = \alpha_0$, a standard argument then yields
\[
  \frac{\hat \sigma_n}{\sigma_n}=\frac{1}{\sigma_n} \left( \frac{1}{k_n} \sum_{i=1}^{k_n} X_{n,i}^{-\hat{\alpha}_n} \right)^{-1/\hat{\alpha}_n}
  \weak 1, \qquad n \to \infty. \qedhere
\]
\end{proof}

% \begin{subsubsection}{Proof of Theorem~\ref{theo:asydis}}

The proof of Theorem~\ref{theo:asydis} is decomposed into a sequence of lemmas.
Recall $\Psi_n$ and $\Psi$ in \eqref{eq:Psin} and \eqref{eq:Psi}, respectively, and define $\dot{\Psi}_n( \alpha ) = (\diff / \diff \alpha) \Psi_n( \alpha )$ and $\dot{\Psi}(\alpha) = (\diff / \diff \alpha) \Psi(\alpha)$. By \eqref{eq:ddLalpha},
\begin{equation}
\label{eq:dPsin}
  \dot{\Psi}_n(\alpha) =
  - \frac{1}{\alpha^2} - \frac{\Pn [x^{-\alpha} (\log x)^2] \, \Pn x^{-\alpha} - (\Pn x^{-\alpha} \log x)^2 }{(\Pn x^{-\alpha})^2},
\end{equation}
where $\Pn$ denotes the empirical distribution of the points $(X_{n,i}/\sigma_n)_{i=1}^{k_n}$ and where 
\[
 \Pn f = \frac{1}{k_n} \sum_{i=1}^{k_n} f(X_{n,i} / \sigma_n).
\]
The asymptotic distribution of $v_n( \hat \alpha_n - \alpha_0)$ can be derived from the asymptotic behavior of $\dot \Psi_n$ and $v_n \Psi_n$, which is the subject of the next two lemmas, respectively.

\begin{lemma}[Slope]
\label{lem:dPsin}
Let $\bm{X}_n = (X_{n,1}, \ldots, X_{n,k_n})$ be a sequence of random vectors in $(0, \infty)^{k_n}$, where $k_n \to \infty$. Suppose that Equation~\eqref{eq:noties} and Condition~\ref{cond:CLT}(i) are satisfied. If $\tilde{\alpha}_n$ is a random sequence in $(0, \infty)$ such that $\tilde{\alpha}_n \weak \alpha_0$ as $n \to \infty$, then
\[
  \dot{\Psi}_n( \tilde{\alpha}_n ) \weak \dot{\Psi}( \alpha_0 ) = - \frac{\pi^2}{6 \alpha_0^2}, \qquad n \to \infty.
\]
\end{lemma}

\begin{proof}
%Let $\Pn$ denote the empirical distribution of the points $(X_{n,i}/\sigma_n)_{i=1}^{k_n}$, and write
%\[
% \Pn f = \frac{1}{k_n} \sum_{i=1}^{k_n} f(X_{n,i} / \sigma_n)
%\]
%for $f : (0, \infty) \to \R$. 
%The random function $\Psi_n$ is differentiable and, by \eqref{eq:ddLalpha}, its derivative is given by
%\begin{equation}
%\label{eq:dPsin}
%  \dot{\Psi}_n(\alpha) =
%  - \frac{1}{\alpha^2} - \frac{\Pn [x^{-\alpha} (\log x)^2] \, \Pn x^{-\alpha} - (\Pn x^{-\alpha} \log x)^2 }{(\Pn x^{-\alpha})^2}.
%\end{equation}
For $\alpha \in (0, \infty)$ and $m \in \{0, 1, 2\}$, define
\[
  f_{m,\alpha}(x) = x^{-\alpha} (\log x)^m, \qquad x \in (0, \infty),
\]
with $(\log x)^0 = 1$ for all $x \in (0, \infty)$. Suppose that we could show that, for $m \in \{0, 1, 2\}$ and some $\eps > 0$,
\begin{equation}
\label{eq:uniformly}
  \sup_{\alpha : \abs{\alpha - \alpha_0} \le \eps} \abs{ \Pn f_{m,\alpha} - \int_0^\infty f_{m,\alpha}(x) \, p_{\alpha_0}(x) \, \diff x }
  \weak 0, \qquad n \to \infty.
\end{equation}
Then from weak convergence of $\tilde{\alpha}_n$ to $\alpha_0$, Slutsky's lemma (\citealp{Van98}, Lemma~2.8) and Lemma~\ref{lem:moments} below, it would follow that
\[
  \dot{\Psi}_n( \tilde{\alpha}_n )
  \weak - \frac{1}{\alpha_0^2} - \frac{\alpha_0^{-2} \Gamma''(2) \, \Gamma(2) - (\alpha_0^{-1} \Gamma'(2))^2}{(\Gamma(2))^2}, \qquad n \to \infty.
\]
Since $\Gamma(2) = 1$, $\Gamma'(2) = 1-\gamma$ and $\Gamma''(2) = (1-\gamma)^2 + \pi^2/6 - 1$, the conclusion would follow.

It remains to show \eqref{eq:uniformly}. We consider the three cases $m \in \{0, 1, 2\}$ separately. 
Let $\eps > 0$ be small enough such that $\alpha_- < \alpha_0-\eps < \alpha_0+\eps < \alpha_+$.
% The value of $\eps$ does not matter, as long as $\alpha_0 - \eps > 0$.

First, let $m = 0$. The maps $\alpha \mapsto (\Pn f_{0, \alpha})^{1/\alpha}$ and $\alpha \mapsto (\int_0^\infty f_{0,\alpha} \, p_{\alpha_0, 1})^{1/\alpha}$ are monotone by Lyapounov's inequality [i.e., $\|f\|_r \le \|f\|_s$ for $0<r<s$, where $\|f\|_r=(\int_{\Xc} |f|^r\, \diff \mu)^{1/r}$ denotes the $L_r$-norm of some real-valued function $f$ on a measurable space $(\Xc, \mu)$], and the second one is also continuous by Lemma~\ref{lem:moments}.
Pointwise convergence of monotone functions to a monotone, continuous limit implies locally uniform convergence \citep[Section~0.1]{Res87}. This property easily extends to weak convergence, provided the limit is nonrandom. We obtain
% Writing $\abs{x}=\max(x,-x)$ and using the bound  $(\Pn f_{0,\alpha_0-\eps})^{1/(\alpha_0-\eps)}  \le (\Pn f_{0,\alpha})^{1/\alpha} \le (\Pn f_{0,\alpha_0+\eps})^{1/(\alpha_0+\eps)}$, valid for all $\abs{\alpha-\alpha_0}\le \eps$, one obtains that pointwise weak convergence implies uniform weak convergence for sufficiently small $\eps>0$:
\[
  \sup_{\alpha : \abs{\alpha - \alpha_0} \le \eps} 
  \abs{ (\Pn f_{0,\alpha})^{1/\alpha} - \left(\int_0^\infty f_{0,\alpha}(x) \, p_{\alpha_0}(x) \, \diff x\right)^{1/\alpha} }
  \weak 0, \qquad n \to \infty.
\]
Uniform continuity of the map $(y, \alpha) \mapsto y^\alpha$ on compact subsets of $(0, \infty)^2$ then yields \eqref{eq:uniformly} for $m = 0$.

Second, let $m = 1$. The maps $\alpha \mapsto \Pn f_{1, \alpha}$ and $\alpha \mapsto \int_0^\infty f_{1, \alpha} \, p_{\alpha_0,1}$ are continuous and nonincreasing (their derivatives are nonpositive). Pointwise weak convergence at each $\alpha \in (\alpha_-, \alpha_+)$ then yields \eqref{eq:uniformly} for $m = 1$.

Finally, let $m = 2$. With probability tending to one, not all variables $X_{n,i}$ are equal to $\sigma_n$, and thus $\Pn (\log x)^2 > 0$. On the latter event, we have
\[
  \Pn x^{-\alpha} (\log x)^2
  = \Pn (\log x)^2 \, \left\{ \left( \frac{\Pn x^{-\alpha} (\log x)^2}{\Pn (\log x)^2} \right)^{1/\alpha} \right\}^\alpha.
\]
By Lyapounov's inequality, the expression in curly braces is nondecreasing in $\alpha$. For each $\alpha \in (\alpha_-, \alpha_+)$, it converges weakly to $\{\Gamma''(1 + \alpha/\alpha_0) / \Gamma''(1)\}^{1/\alpha}$, which is nondecreasing and continuous in $\alpha$; see Lemma~\ref{lem:moments}. It follows that
\[
  \sup_{\alpha : \abs{\alpha - \alpha_0} \le \eps} 
  \abs{ \left( \frac{\Pn x^{-\alpha} (\log x)^2}{\Pn (\log x)^2} \right)^{1/\alpha} - \left( \frac{\Gamma''(1 + \alpha/\alpha_0)}{\Gamma''(1)} \right)^{1/\alpha} }
  \weak 0, \qquad n \to \infty.  
\]
Equation~\eqref{eq:uniformly} for $m = 2$ follows.
\end{proof}

\begin{lemma}
\label{lem:Psinweak2}
Assume Condition~\ref{cond:CLT}. Then, as $n \to \infty$,
\begin{align}
\label{eq:Psinweak2}
  v_n \, \Psi_n( \alpha_0 ) 
  = \Gn x^{-\alpha_0} \log(x) + \frac{1-\gamma}{\alpha_0} \, \Gn x^{-\alpha_0} - \Gn \log(x) + o_p(1).
\end{align}
The expression on the right converges weakly to $Y \equiv Y_1 + \frac{1-\gamma}{\alpha_0} Y_2 - Y_3$.
\end{lemma}

\begin{proof} Recall that
\begin{equation*}
%\label{eq:Psin0}
  \Psi_n( \alpha_0 ) 
  = \Psi_{k_n}( \alpha_0 \mid \bm{X}_n / \sigma_n )
  = \frac{1}{\alpha_0} + \frac{\Pn x^{-\alpha_0} \log(x)}{\Pn x^{-\alpha_0}} - \Pn \log(x).
\end{equation*}
Define $\phi : \R \times (0, \infty) \times \R \to \R$ by
\[
   \phi(y_1, y_2, y_3) = \frac{1}{\alpha_0} + \frac{y_1}{y_2} - y_3.
\]
The previous two displays allow us to write
\[
  \Psi_n( \alpha_0 ) = \phi \bigl( \Pn x^{-\alpha_0} \log(x), \, \Pn x^{-\alpha_0}, \, \Pn \log(x) \bigr).
\]
Recall Lemma~\ref{lem:moments} and put
\begin{align*}
  \bm{y}_0 
  %&= \bigl( P x^{-\alpha_0} \log(x), \, P x^{-\alpha_0}, \, P \log(x) \bigr) \\
  &= \bigl( - \alpha_0^{-1} \Gamma'(2), \Gamma(2), - \alpha_0^{-1} \Gamma'(1) \bigr) 
  = \bigl( \alpha_0^{-1} (\gamma - 1), \, 1, \, \alpha_0^{-1} \gamma \bigr).
\end{align*}
As already noted in the proof of Lemma~\ref{lem:Psi}, we have
$
  \phi ( \bm{y}_0 )
  = \alpha_0^{-1} + \alpha_0^{-1} (\gamma-1) - \alpha_0^{-1} \gamma
  = 0.
$
As a consequence,
\[
  v_n \, \Psi_n( \alpha_0 )
  = v_n \bigl\{ 
    \phi \bigl( \Pn x^{-\alpha_0} \log(x), \, \Pn x^{-\alpha_0}, \, \Pn \log(x) \bigr) 
    -
    \phi ( \bm{y}_0 )
  \bigr\}.
\]
In view of Condition~\ref{cond:CLT} and the delta method, as $n \to \infty$,
\[
  v_n \, \Psi_n( \alpha_0 )
  = \dot{\phi}_1( \bm{y}_0 ) \, \Gn x^{-\alpha_0} \log(x)
  + \dot{\phi}_2( \bm{y}_0 ) \, \Gn x^{-\alpha_0}
  + \dot{\phi}_3( \bm{y}_0 ) \, \Gn \log(x)
  + o_p(1),
\]
where $\dot{\phi}_j$ denotes the first-order partial derivative of $\phi$ with respect to $y_j$ for $j \in \{1, 2, 3\}$. Elementary calculations yield
\begin{align*}
  \dot{\phi}_1( \bm{y}_0 ) &= 1, & 
  \dot{\phi}_2( \bm{y}_0 ) &= \alpha_0^{-1}(1 - \gamma), &
  \dot{\phi}_3( \bm{y}_0 ) &= -1.
\end{align*}
The conclusion follows by Slutsky's lemma.
\end{proof}

\begin{prop}[Asymptotic expansion for the shape parameter]
\label{prop:shape}
Assume that the conditions of Theorem~\ref{theo:asydis} hold. 
%If there exists $0 < v_n \to \infty$ such that, for some random variable $Y$,
%\begin{equation}
%\label{eq:Psinweak}
%  v_n \, \Psi_n( \alpha_0 ) \weak Y, \qquad n \to \infty.
%\end{equation}
Then, with $Y$ as defined in Lemma~\ref{lem:Psinweak2},
\begin{equation}
\label{eq:shape:delta}
  v_n \, \bigl( \hat{\alpha}_n - \alpha_0 \bigr)
  = \frac{6 \alpha_0^2}{\pi^2} \, v_n \, \Psi_n( \alpha_0 ) + o_p(1)
  \weak \frac{6 \alpha_0^2}{\pi^2} \, Y, \qquad n \to \infty.
\end{equation}
\end{prop}

\begin{proof}
Recall that, with probability tending to one, $\hat{\alpha}_n$ is the unique zero of the random function $\alpha \mapsto \Psi_n(\alpha)$. Recall that $\dot{\Psi}_n$ in \eqref{eq:dPsin} is the derivative of $\Psi_n$. With probability tending to one, we have, by virtue of the mean-value theorem,
\begin{align*}
  0
  = \Psi_n( \hat{\alpha}_n ) 
  &= \bigl( \Psi_n( \hat{\alpha}_n ) - \Psi_n( \alpha_0 ) \bigr) + \Psi_n( \alpha_0 ) 
  = ( \hat{\alpha}_n - \alpha_0 ) \, \dot{\Psi}_n( \tilde{\alpha}_n ) + \Psi_n( \alpha_0 );
\end{align*}
here $\tilde{\alpha}_n$ is a convex combination of $\hat{\alpha}_n$ and $\alpha_0$. Since $\dot{\Psi}_n( \alpha ) \le -1/\alpha^2 < 0$ (argument as in the proof of Lemma~\ref{lem:existsUnique}), we can write
\[
  v_n \, \bigl( \hat{\alpha}_n - \alpha_0 \bigr) = - \frac{1}{\dot{\Psi}_n( \tilde{\alpha}_n )} \, v_n \, \Psi_n( \alpha_0 ).
\]
By weak consistency of $\hat{\alpha}_n$, we have $\tilde{\alpha}_n \weak \alpha_0$ as $n \to \infty$. Lemma~\ref{lem:dPsin} then gives $\dot{\Psi}_n( \tilde{\alpha}_n ) \weak - \pi^2 / (6 \alpha_0^2)$ as $n \to \infty$. Apply Lemma~\ref{lem:Psinweak2} and Slutsky's lemma to conclude.
\end{proof}

\begin{proof}[Proof of Theorem~\ref{theo:asydis} and Addendum~\ref{add:asydis}]
%The theorem is a combination of Proposition~\ref{prop:shape} and Proposition~\ref{prop:scale} below. The addendum follows from from  a tedious but straightforward calculation.
%\end{proof}
%
%\begin{prop}
%\label{prop:scale}
%Let $\alpha_0 \in (0, \infty)$ and assume the set-up and conditions in Lemmas~\ref{lem:dPsin} and~\ref{lem:Psinweak2}. Then
%\begin{equation}
%\label{eq:joint:delta}
%  \begin{pmatrix} 
%    v_n ( \hat{\alpha}_n - \alpha_0 ) \\ 
%    v_n \, ( \hat{\sigma}_n / \sigma_n - 1 ) 
%  \end{pmatrix} 
%  = 
%  M(\alpha_0)
%  \begin{pmatrix}
%    \Gn x^{-\alpha_0} \log(x) \\
%    \Gn x^{-\alpha_0} \\
%    \Gn \log(x)
%  \end{pmatrix} + o_p(1)
%  \weak
%  M(\alpha_0) \bm{Y},
%  \qquad n \to \infty,
%%   \bigl( Y_1 + \frac{1-\gamma}{\alpha_0} Y_2 - Y_3, \; \text{TO BE COMPLETED} \bigr), 
%\end{equation}
%where $\bm{Y} = (Y_1, Y_2, Y_3)^T$ is the limit vector in \eqref{eq:Gnweak} and where
%\begin{equation}
%\label{eq:M}
%  M(\alpha_0)
%  =
%  \begin{pmatrix}
%    6\alpha_0^2/\pi^2 & 6\alpha_0(1-\gamma)/\pi^2 & - 6\alpha_0^2/\pi^2 \\
%    6(\gamma-1)/\pi^2 & -6(1-\gamma)^2/(\pi^2\alpha_0) - 1/\alpha_0 &  6(1-\gamma)/\pi^2 
%  \end{pmatrix}.
%\end{equation}
%\end{prop}
%
%\begin{proof}
% Recall the set-up and the notation of Proposition~\ref{prop:shape}. 
Combining Equations~\eqref{eq:shape:delta} and \eqref{eq:Psinweak2} yields
\[
  v_n \, ( \hat{\alpha}_n - \alpha_0 )
  = 
  \frac{6 \alpha_0^2}{\pi^2} 
  \left(
    \Gn x^{-\alpha_0} \log(x) + \frac{1-\gamma}{\alpha_0} \, \Gn x^{-\alpha_0} - \Gn \log(x)
  \right)
  + o_p(1)
\]
as $n \to \infty$. This yields the first row in \eqref{eq:joint:delta}.

By definition of $\hat{\sigma}_n$, we have
$
  ( \hat{\sigma}_n / \sigma_n )^{-\hat{\alpha}_n}
  = \Pn x^{-\hat{\alpha}_n}.
$
Consider the decomposition
\begin{equation}
\label{eq:scale:decomp}
  v_n \, \bigl( ( \hat{\sigma}_n / \sigma_n )^{-\hat{\alpha}_n} - 1 \bigr)
  = v_n \, \bigl( \Pn x^{-\hat{\alpha}_n} - \Pn x^{-\alpha_0} \bigr)
  + v_n \, \bigl( \Pn x^{-\alpha_0} - 1 \bigr).
\end{equation}
By the mean value theorem, there exists a convex combination, $\tilde{\alpha}_n$, of $\hat{\alpha}_n$ and $\alpha_0$ such that
\[
  \Pn x^{-\hat{\alpha}_n} - \Pn x^{-\alpha_0}
  = - (\hat{\alpha}_n - \alpha_0) \, \Pn x^{-\tilde{\alpha}_n} \log(x).
\]
By the argument for the case $m = 1$ in the proof of Lemma~\ref{lem:dPsin}, we have
\[
  \Pn x^{-\tilde{\alpha}_n} \log(x) \weak - \frac{1}{\alpha_0} \Gamma'(2) = - \frac{1-\gamma}{\alpha_0}, \qquad n \to \infty.
\]
By Proposition~\ref{prop:shape} and Lemma~\ref{lem:Psinweak2}, it follows that, as $n \to \infty$,
\begin{align*}
  v_n \, \bigl( \Pn x^{-\hat{\alpha}_n} - \Pn x^{-\alpha_0} \bigr)
  &= v_n \, ( \hat{\alpha}_n - \alpha_0 ) \, \frac{1-\gamma}{\alpha_0} + o_p(1) \\
  &= \frac{6 \alpha_0 \, (1-\gamma)}{\pi^2} \, v_n \, \Psi_n( \alpha_0 ) + o_p(1) \\
  &= \frac{6 \alpha_0 \, (1-\gamma)}{\pi^2} \, 
  \left( \Gn x^{-\alpha_0} \log(x) + \frac{1-\gamma}{\alpha_0} \, \Gn x^{-\alpha_0} - \Gn \log(x) \right) + o_p(1).
\end{align*}
This expression in combination with \eqref{eq:scale:decomp} yields, as $n \to \infty$,
\begin{multline}
\label{eq:scale:aux}
  v_n \, \bigl( ( \hat{\sigma}_n / \sigma_n )^{-\hat{\alpha}_n} - 1 \bigr) \\
  = \frac{6 \alpha_0 \, (1-\gamma)}{\pi^2} \, 
  \left( \Gn x^{-\alpha_0} \log(x) + \frac{1-\gamma}{\alpha_0} \, \Gn x^{-\alpha_0} - \Gn \log(x) \right)
  + \Gn x^{-\alpha_0} + o_p(1).
\end{multline}
% Weak convergence of the bivariate sequence $\bigl( v_n \, \Psi_n( \alpha_0 ), \; v_n \, (\Pn x^{-\alpha_0} - 1) \bigr)$ then yields joint weak convergence of
% \[
%   \bigl( v_n \, ( \hat{\alpha}_n - \alpha_0 ), \; v_n \, ( ( \hat{\sigma}_n / \sigma_n )^{-\hat{\alpha}_n} - 1 ) \bigr).
% \]
Write $Z_n = ( \hat{\sigma}_n / \sigma_n )^{-\hat{\alpha}_n}$, which converges weakly to 1 as $n \to \infty$. By the mean value theorem, \begin{align*}
  v_n \, ( \hat{\sigma}_n / \sigma_n - 1 )
  &= v_n \, (Z_n^{-1/\hat{\alpha}_n} - 1) 
  = v_n \, (Z_n - 1) \, (-1/\hat{\alpha}_n) \, \tilde{Z}_n^{-1/\hat{\alpha}_n - 1},
\end{align*}
where $\tilde{Z}_n$ is a random convex combination of $Z_n$ and $1$. But then $\tilde{Z}_n \weak 1$ as $n \to \infty$, whence, by consistency of $\hat{\alpha}_n$ and Slutsky's lemma,
\[
  v_n \, ( \hat{\sigma}_n / \sigma_n - 1 )
  = (-1/\alpha_0) \, v_n \, ( ( \hat{\sigma}_n / \sigma_n )^{-\hat{\alpha}_n} - 1 ) + o_p(1), \qquad n \to \infty.
\]
Combinining this with \eqref{eq:scale:aux}, we find
\begin{multline*}
%\label{eq:scale:delta}
  v_n \, ( \hat{\sigma}_n / \sigma_n - 1 )
  = - \frac{6 (1-\gamma)}{\pi^2} \, 
  \left( \Gn x^{-\alpha_0} \log(x) + \frac{1-\gamma}{\alpha_0} \, \Gn x^{-\alpha_0} - \Gn \log(x) \right) \\
  - \alpha_0^{-1} \Gn x^{-\alpha_0} + o_p(1)
\end{multline*}
as $n \to \infty$. This is the second row in \eqref{eq:joint:delta}. 

The proof of Addendum~\ref{add:asydis} follows from a tedious but straightforward calculation.
% Weak convergence of the random sequence
% \[
%   \bigl( v_n \, ( \hat{\alpha}_n - \alpha_0 ), \; v_n (\hat{\sigma}_n / \sigma_n - 1) \bigr)
% \]
% follows. Its bivariate limit distribution can be reconstructed from the previous computations.
\end{proof}

\subsection{Proofs for Section~\ref{sec:stat}}
\label{subsec:proofs:stat}

% We first introduce some additional notation. Next, we state and prove a number of auxiliary results. Finally, we give the proof of Theorem~\ref{theo:asydis:stat}.

% Let $\Pn$ and $\Pn^{[\ell]}$ be the following two empirical measures: for $f : (0, \infty) \to \R$, put
% \begin{align*}
%   \Pn f 
%   &= \frac{1}{k_n} \sum_{i=1}^{k_n} f \bigl( (M_{r_n,i} \vee c) / \sigma_{r_n} \bigr), \\
%   \Pn^{[\ell]} f 
%   &= \frac{1}{k_n} \sum_{i=1}^{k_n} f \bigl( (M_{r_n,i}^{[\ell]} \vee c) / \sigma_{r_n} \bigr).
% \end{align*}

% Consider the function class 
% \begin{equation}
% \label{eq:calF}
%   \mathcal{F} = \mathcal{F}_0 \cup \mathcal{F}_1 \cup \mathcal{F}_2, 
% \end{equation}
% where $\mathcal{F}_0$, $\mathcal{F}_1$ and $\mathcal{F}_2$ have been defined in equations~\eqref{eq:F0}, \eqref{eq:F1} and \eqref{eq:F2}, respectively.
% \begin{equation}
% \label{eq:calF}
% \left.
%   \begin{array}{lcl}
%     \mathcal{F}_0
%     &=&
%     \{ x \mapsto x^{-\alpha} : \alpha \in (0, \infty) \},  \\[1ex]
%     \mathcal{F}_1
%     &=&
%     \{ x \mapsto x^{-\alpha} \log(x) : \alpha \in [0, \infty) \}, \\[1ex]
%     \mathcal{F}_2 
%     &=&
%     \{ x \mapsto x^{-\alpha} (\log x)^2 : \alpha \in [0, \infty) \}.
%   \end{array}
% \right\}
% \end{equation}
% The range of the argument $x$ is $(0, \infty)$ in each case.

\begin{lemma}[Block maxima rarely show ties]
\label{lem:ties}
Under Conditions~\ref{cond:DA} and \ref{cond:alpha}, for every $c \in (0, \infty)$, we have $\Pr[ M_{r_n,1} \vee c = M_{r_n,3} \vee c ] \to 0$ as $n \to \infty$.
\end{lemma}

\begin{proof}[Proof of Lemma~\ref{lem:ties}]
By the domain-of-attraction condition combined with the strong mixing property, the sequence of random vectors $((M_{r_n,1} \vee c)/\sigma_{r_n}, (M_{r_n,3} \vee c)/\sigma_{r_n})$ converges weakly to the product of two independent $\Frechet(\alpha_0, 1)$ random variables. Apply the  Portmanteau lemma -- the set $\{ (x, y) \in \R^2 : x = y \}$ is closed and has zero probability in the limit. 
\end{proof}

\begin{lemma}[Moments of block maxima converge]
\label{lem:moment}
Under Conditions~\ref{cond:DA} and~\ref{cond:moment}, we have, for every $c \in (0, \infty)$,
\[
  \lim_{n \to \infty} \Exp[ f \bigl( ( M_n \vee c) / \sigma_n \bigr) ]
  = \int_0^\infty f(x) \, p_{\alpha_0,1}(x) \, \diff x
\]
for every measurable function $f : (0, \infty) \to \R$ which is continuous almost everywhere  and for which there exist $0<\eta<\mom$ such that $\abs{ f(x) } \le g_{\eta, \alpha_0}(x) = \{ x^{-\alpha_0} \ind(x\le e) +\log (x)  \ind(x>e)\}^{2+\eta}$.
\end{lemma}

\begin{proof}[Proof of Lemma~\ref{lem:moment}]
An elementary argument shows that we may replace $M_n \vee 1$  by $M_n\vee c$ in  \eqref{eq:mom1}.
Since $c / \sigma_n \to 0$ as $n \to \infty$, the sequence $(M_n \vee c) / \sigma_n$ converges weakly to the $\Frechet(\alpha_0, 1)$ distribution in view of Condition~\ref{cond:DA}. The result follows from Example 2.21 in \cite{Van98}.
\end{proof}

In order to separate maxima over consecutive blocks by a time lag of at least $\ell$, we clip off the final $\ell - 1$ variables within each block:
\begin{equation}
\label{eq:max:clipped}
  M_{r,i}^{[\ell]}
  = \max \{ \xi_t : (i-1)r+1 \le t \le ir - \ell + 1 \}.
\end{equation}
Clearly, $M_{r,i} \ge M_{r,i}^{[\ell]}$. 
The probability that the maximum over a block of size $r$ is attained by any of the final $\ell-1$ variables should be small; see Lemma~\ref{lem:clipping:2} below.

\begin{lemma}[Short blocks are small]
\label{lem:clipping}
Assume Condition~\ref{cond:DA}. If $\ell_n = o(r_n)$ 
and if $\alpha({\ell_n}) = o(\ell_n / r_n)$ 
as $n \to \infty$, then for all $\eps > 0$,
\begin{equation}
\label{eq:smallblock:eps}
  \Pr[ M_{\ell_n} \ge \eps \sigma_{r_n} ] = O( \ell_n / r_n ), \qquad n \to \infty.
\end{equation}
\end{lemma}

\begin{proof}[Proof of Lemma~\ref{lem:clipping}]
Let $F_r$ be the cumulative distribution function of $M_r$. By \citet[Lemma~7.1]{BucSeg14}, for every $u > 0$,
\begin{equation}
\label{eq:FrMl}
  \Pr[ F_{r_n} ( M_{\ell_n} ) \ge u ] = O( \ell_n / r_n ), \qquad n \to \infty.
\end{equation}
Fix $\eps > 0$. By assumption,
\[
  \lim_{n \to \infty} F_{r_n}( \eps \sigma_{r_n} ) = \exp( - \eps^{-\alpha_0} ).
\]
For sufficiently large $n$, we have
\begin{align*}
  \Pr[ M_{\ell_n} \ge \eps \sigma_n ]
  &\le \Pr[ F_{r_n}( M_{\ell_n} ) \ge F_{r_n} (\eps \sigma_n ) ] 
  \le \Pr[ F_{r_n}( M_{\ell_n} ) \ge \exp( - \eps^{-\alpha_0} ) / 2 ].
\end{align*}
Set $u = \exp(- \eps^{-\alpha_0} ) / 2$ in \eqref{eq:FrMl} to arrive at \eqref{eq:smallblock:eps}.
\end{proof}

% The following lemma provides a weakened version of Condition~\ref{cond:clipping}. The weakened version will turn out to be good enough.

\begin{lemma}[Clipping doesn't hurt]
\label{lem:clipping:2}
Assume Condition~\ref{cond:DA}.  If $\ell_n = o(r_n)$ 
and if $\alpha({\ell_n}) = o(\ell_n / r_n)$ 
as $n \to \infty$,
%, $\sigma_{r_n - \ell_n} / \sigma_{r_n} \to 1$ 
% \begin{equation}
%   \label{eq:mixing:ellr}
%   \frac{r_n}{\ell_n} \alpha(\ell_n) \to 0, \qquad n \to \infty,
% \end{equation}
then
\begin{equation}
\label{eq:smallbigblock}
  \Pr[ M_{r_n} > M_{r_n-\ell_n} ] \to 0, \qquad n \to \infty.
\end{equation}
\end{lemma}

\begin{proof}[Proof of Lemma~\ref{lem:clipping:2}]
Recall Lemma~\ref{lem:clipping}. For every $\eps > 0$ we have, by stationarity,
\[
  \Pr[ M_{r_n} > M_{r_n-\ell_n} ]
  \le \Pr[ M_{r_n - \ell_n} \le \eps \sigma_{r_n} ] + \Pr[ M_{\ell_n} > \eps \sigma_{r_n} ].
\]
Since $\sigma_{r_n - \ell_n} / \sigma_{r_n} \to 1$ as a consequence of Condition~\ref{cond:DA} and the fact that $\ell_n = o(r_n)$ as $n \to \infty$, the first term converges to $\exp( - \eps^{-\alpha_0} )$ as $n \to \infty$, whereas the second one converges to $0$ by Lemma~\ref{lem:clipping}. Since $\eps > 0$ was arbitrary, Equation~\eqref{eq:smallbigblock} follows.
\end{proof}

\begin{proof}[Proof of Theorem~\ref{theo:asydis:stat}]
We apply Theorem~\ref{theo:asydis} and Addendum~\ref{add:asydis} to the array $X_{n,i} = M_{r_n,i} \vee c$ and $v_n = \sqrt{k_n}$, where $c \in (0, \infty)$ is arbitrary and $i \in \{1, \ldots, k_n\}$. By Condition~\ref{cond:small}, we have $\lim_{n \to \infty} \Pr[ \forall i = 1, \ldots, k_n : X_{n,i} = M_{r_n,i} ] = 1$.

The not-all-tied property~\eqref{eq:noties} has been established in Lemma~\ref{lem:ties}.

We need to check Condition~\ref{cond:CLT}, and in particular that the distribution of the random vector $\bm{Y}$ in \eqref{eq:Y} is $\mathcal{N}_3( B, \Sigma_{\bm{Y}} )$ with $B$ as in the statement of Theorem~\ref{theo:asydis:stat} and $\Sigma_{\bm{Y}}$ as in \eqref{eq:Sigma}. Essentially, the proof employs the Bernstein big-block-small-block method in combination with the Lindeberg central limit theorem.

Let $\ell_n= \max\{ s_n, \ip{ r_n \sqrt{ \alpha(s_n)}} \}$, where $s_n= \ip{\sqrt{r_n}}$. Clearly, 
\begin{align} 
\label{eq:mixing:ellr}
\ell_n \to \infty, \quad \ell_n=o(r_n) \quad \text{and} \quad \alpha(\ell_n) = o(\ell_n/r_n), \text{ as } n \to \infty. \end{align}

Consider the truncated and rescaled block maxima
\begin{align*}
  Z_{r,i} &= (M_{r_n,i} \vee c) / \sigma_r,  \qquad
  Z_{r,i}^{[\ell_n]} = (M_{r_n,i}^{[\ell_n]} \vee c) / \sigma_r,
\end{align*}
with $M_{r,i}^{[\ell_n]}$ as in \eqref{eq:max:clipped}. Consider  the following empirical and population probability measures:
\begin{align*}
  \Pn f &= \frac{1}{k_n} \sum_{i=1}^{k_n} f(Z_{r_n,i}), &
  P_n f &= \Exp [f(Z_{r_n,i})], \\
  \Pn^{[\ell_n]} f &= \frac{1}{k_n} \sum_{i=1}^{k_n} f(Z_{r_n,i}^{[\ell_n]}), &
  P_n^{[\ell_n]} f &= \Exp [f(Z_{r_n,i}^{[\ell_n]})].
\end{align*}
Abbreviate the tentative limit distribution by $P = \Frechet(\alpha_0, 1)$. We will also need the following empirical processes:
\begin{align*}
  \Gn &= \sqrt{k_n} ( \Pn - P ) && \text{(uncentered)}, \\
  \tilde{\G}_n &= \sqrt{k_n} ( \Pn - P_n ) && \text{(centered)}, \\
  \tilde{\G}_n^{[\ell_n]} &= \sqrt{k_n} ( \Pn^{[\ell_n]} - P_n^{[\ell_n]} ) && \text{(centered)}.
\end{align*}
Finally, the bias arising from the finite block size is quantified by the operator
\[
  B_n = \sqrt{k_n} (P_n - P).
\]

\textit{Proof of Condition~\ref{cond:CLT}(i).}
%\change{Choose $\eta\in(2,p)$ and $0<\alpha_-<\alpha_0<\alpha_+$ such that
%%\[
%%\eta \alpha_+ < p \alpha_0, \qquad \omega > \frac{2}{\eta-2}.
%%\]
%\[
%\omega > \frac{2}{\eta-2} , \qquad 
%\alpha_+ < \frac{2\eta}{2+\eta} \times \alpha_0.
%\]
%}
Choose $\eta\in(2/\omega,\mom)$ and $0<\alpha_-<\alpha_0<\alpha_+$. Additional constraints on $\alpha_+$ will be imposed below, while the values of $\eta$ and $\alpha_-$ do not matter.
Recall the function class $\mathcal{F}_2(\alpha_-, \alpha_+)$ in~\eqref{eq:F2}. For every $f \in \mathcal{F}_2(\alpha_-, \alpha_+)$, %not of the form $x^{-\alpha_0}$, $\log(x)$ or $x^{-\alpha_0} \log x$, 
we just need to show that 
\[ 
  \Pn f = P f + o_p(1), \qquad n \to \infty. 
\]
The domain-of-attraction property (Condition~\ref{cond:DA}) and the asymptotic moment bound (Condition~\ref{cond:moment}) yield
\[
  \Exp[ \Pn f ] = P_n f \to P f, \qquad n \to \infty,
\]
by uniform integrability, see Lemma~\ref{lem:moment} (note that $|f|$ is bounded by a multiple of $g_{0, \alpha_0}$ if $\alpha_+$ is chosen suitably small: $\alpha_+< 2 \alpha_0$ must be satisfied). Further,
\[
  \Pn f - P_n f = \frac{1}{\sqrt{k_n}} \tilde{\G}_n f.
\]
Below, see \eqref{eq:Delta}, we will show that
\begin{equation}
\label{eq:tildeGnf}
  \tilde{\G}_n f = \tilde{\G}_n^{[\ell_n]} f + o_p(1) = O_p(1) + o_p(1) = O_p(1), \qquad n \to \infty.
\end{equation}
It follows that, as required,
\[
  \Pn f = (\Pn f - P_n f) + P_n f = o_p(1) + Pf + o(1) = Pf + o_p(1), \qquad n \to \infty.
\]

\textit{Proof of Condition~\ref{cond:CLT}(ii).}
We can decompose the empirical process $\Gn$ in a stochastic term and a bias term:
\begin{align*}
  \Gn 
  &= \sqrt{k_n} ( \Pn - P_n ) + \sqrt{k_n} ( P_n - P ) 
  = \tilde{\G}_n + B_n.
\end{align*}
For $f \in \mathcal{H} = \{ f_1, f_2, f_3 \}$, the bias term $B_n f$ converges to $B(f)$ thanks to Condition~\ref{cond:bias}. It remains to treat the stochastic term $\tilde{\G}_n f$, for all $f \in \mathcal{F}_2(\alpha_-, \alpha_+)$ [in view of the proof of item (i); see~\eqref{eq:tildeGnf} above]. We will show that the finite-dimensional distributions of $\tilde{\G}_n$ converge to those of a $P$-Brownian bridge, $\G$, i.e., a zero-mean, Gaussian stochastic process with covariance function given by
\[
  \cov ( \G f, \G g ) 
  = 
  P \bigl( (f - Pf) (g - Pg) \bigr)
  = 
  \cov_P \bigl( f(X), g(X) \bigr), \qquad f, g \in \mathcal{F}_2 ( \alpha_-, \alpha_+ ).
\]
%\change{for $f, g \in \mathcal{F}_2 ( \alpha_-, \alpha_+ )$}.

Decompose the stochastic term in two parts:
\begin{equation} \label{eq:Delta}
  \tilde{\G}_n = \tilde{\G}_n^{[\ell_n]} + \Delta_n.
\end{equation}
We will show that $\Delta_n$ converges to zero in probability and that the finite-dimensional distributions of $\tilde{\G}_n^{\scriptscriptstyle [\ell_n]}$ converge to those of $\G$.

First, we treat the main term, $\tilde{\G}_n^{\scriptscriptstyle [\ell_n]}$. By the Cram\'er--Wold device, it suffices to show that $\tilde{\G}_n^{\scriptscriptstyle [\ell_n]} g \weak \Gb g$ as $n\to\infty$, where $g$ is an arbitrary linear combination of functions $f\in \Fc_2(\alpha_-, \alpha_+)$.  Define
\[
\phi_{ni}(t) = \exp\big[ -\iu t k_n^{-1/2} \{ g(Z_{r_n,i}^{[\ell_n]}) - P_n^{[\ell_n]} g \} \big],
\]
with $\iu$ the imaginary unit. Note that the characteristic function  of $\tilde{\G}_n^{\scriptscriptstyle [\ell_n]} g$ can be written as $ t \mapsto \Eb[ \prod_{i=1}^{k_n} \phi_{ni}(t)]$. Successively applying Lemma 3.9 in \cite{DehPhi02}, we obtain that
\[
\left| \Eb\left[\prod_{i=1}^{k_n} \phi_{ni}(t)\right] - \prod_{i=1}^{k_n} \Eb[\phi_{ni}(t)] \right|
\le
2\pi k_n \max_{i=1}^{k_n} \alpha\left( \sigma \{ \phi_{ni}(t) \}, \sigma\left\{ \prod_{j=i+1}^{k_n} \phi_{nj}(t)\right\} \right),
\]
where $\alpha(\Ac_1, \Ac_2)$ denotes the alpha-mixing coefficient between the sigma-fields $\Ac_1$ and $\Ac_2$. Since the maxima $Z_{r,i}^{\scriptscriptstyle[\ell_n]}$ over different blocks are based on observations that are at least $\ell_n$ observations apart, the expression on the right-hand side of the last display is of the order $O(k_n \alpha(\ell_n))$, which converges to $0$ as a consequence of Equation~\eqref{eq:mixing:kdelta}. We can conclude that the weak limit of $\tilde{\G}_n^{\scriptscriptstyle [\ell_n]} g$ is the same as the one of
\[
\tilde {\Hb}_n^{ [\ell_n]} g = \sqrt{k_n}\left\{ \frac{1}{k_n} \sum_{i=1}^{k_n} g(\bar Z_{r_n,i}^{[\ell_n]}) - P_n^{[\ell_n]}g \right\},
\]
where $\bar Z_{r_n,i}^{\scriptscriptstyle[\ell_n]}$ are independent over $i\in \N$ and have the same distribution as $Z_{r_n,i}^{\scriptscriptstyle[\ell_n]}$. By the classical central limit theorem for row wise independent triangular arrays, the weak limit of $\tilde {\Hb}_n^{\scriptscriptstyle [\ell]} g$ is $\Gb g$: first, its variance 
\[
\Var(\tilde {\Hb}_n^{ [\ell_n]} g) = P_n^{[\ell_n]}g^2 - (P_n^{[\ell_n]} g)^2
\]
converges to $\Var(\Gb g)$ by Lemma~\ref{lem:moment}. Note that the square of any linear combination $g$ of functions $f\in \Fc_2(\alpha_-, \alpha_+)$ can be bounded by a multiple of $g_{\eta, \alpha_0}$, after possibly decreasing the value of $\alpha_+>\alpha_0$. Second, 
%the Lindeberg Condition follows from the H\"older  and Markov inequalities: for any $\eps>0$ and $\delta>0$,  
%\begin{align*}
%\frac{1}{k_n} \sum_{i=1}^{k_n}&\, \Exp\left[ | g(\bar Z_{r_n,i}^{[\ell_n]}) - P_n^{[\ell_n]}g | ^2 \ind( | g(\bar Z_{r_n,i}^{[\ell_n]}) - P_n^{[\ell_n]}g | > \eps k_n^{1/2}) \right] \\
%& \le 
%\frac{1}{k_n} \sum_{i=1}^{k_n} \Exp[| g(\bar Z_{r_n,i}^{[\ell_n]}) - P_n^{[\ell_n]}g |^{2+\delta} ]^{\frac2{2+\delta}} \Prob(| g(\bar Z_{r_n,i}^{[\ell_n]}) - P_n^{[\ell_n]}g | > \eps k_n^{1/2}) ^{\frac{\delta}{2+\delta}}\ \\
%& \le 
%\frac{1}{k_n} \sum_{i=1}^{k_n} \Exp[| g(\bar Z_{r_n,i}^{[\ell_n]}) - P_n^{[\ell_n]}g |^{2+\delta} ] (\eps k_n^{1/2}) ^{-\delta},
%\end{align*}
%which converges to $0$ as $n\to \infty$ again as a consequence of Lemma~\ref{lem:moment}, as $|g|^{2+\delta}$ can also be bounded by a multiple of $g_{\eta,\alpha_0}$ if $\delta\in(0,\eta)$ and $\alpha_+>\alpha_0$ are chosen sufficiently small.
the Lyapunov Condition is satisfied: for  all $\delta>0$,  
\begin{align*}
\frac{1}{k_n^{1+\delta/2}} \sum_{i=1}^{k_n} \Exp\big[| g(\bar Z_{r_n,i}^{[\ell_n]}) - P_n^{[\ell_n]}g |^{2+\delta} \big] 
\end{align*}
converges to $0$ as $n\to \infty$ again as a consequence of Lemma~\ref{lem:moment}, as $|g|^{2+\delta}$ can also be bounded by a multiple of $g_{\eta,\alpha_0}$ if $\delta\in(0,\eta)$ and $\alpha_+>\alpha_0$ are chosen sufficiently small.

Now, consider the remainder term $\Delta_n$ in \eqref{eq:Delta}. Since $\tilde{\G}_n f$ and $\tilde{\G}_n^{[\ell_n]} f$ are centered, so is $\Delta_n f$, and
\begin{align*}
  \Exp[ (\Delta_n f)^2 ]
  &= \var ( \Delta_n f ) 
  = \frac{1}{k_n} \var \left( \sum_{i=1}^{k_n} \Delta_{r_n,i}^{[\ell_n]} f \right),
\end{align*}
where
$
  \Delta_{r,i}^{[\ell_n]} f = f(Z_{r,i}) - f(Z_{r,i}^{[\ell_n]}).
$
By stationarity and the Cauchy--Schwartz inequality,
\begin{align}
\nonumber
  \Exp[ (\Delta_n f)^2 ]
  &= 
  \var \left( \Delta_{r_n,1}^{[\ell_n]} f \right) 
  + 
  \frac{2}{k_n} \sum_{h=1}^{k_n-1} (k_n - h) \, 
  \cov \left( \Delta_{r_n,1}^{[\ell_n]} f, \Delta_{r_n,1+h}^{[\ell_n]} f \right) \\
\label{eq:MSEdecomp}
  &\le
  3 \var \left( \Delta_{r_n,1}^{[\ell_n]} f \right)
  + 2 \sum_{h=2}^{k_n-1} \abs{ \cov \left( \Delta_{r_n,1}^{[\ell_n]} f, \Delta_{r_n,1+h}^{[\ell_n]} f \right) }.
\end{align}
Please note that we left the term $h = 1$ out of the sum; whence the factor three in front of the variance term.

Since $\ell_n = o(r_n)$ as $n \to \infty$ by Condition~\ref{cond:alpha}, we have $\sigma_{r_n - \ell_n + 1} / \sigma_{r_n} \to 1$ as $n \to \infty$ by Condition~\ref{cond:DA}. The asymptotic moment bound in Condition~\ref{cond:moment} then ensures that we may choose $\delta\in(2/\omega, \mom)$ and $\alpha_+>\alpha_0$ such that, 
for  every $f \in \Fc_2(\alpha_-, \alpha_+)$, we have, by Lemma~\ref{lem:moment}, 
\begin{equation}
\label{eq:pnormp}
  \limsup_{n \to \infty} \Exp \left[ \abs{ \Delta_{r_n,1}^{[\ell_n ]} f }^{2+\delta} \right] < \infty.
\end{equation}

On the event that $M_{r_n,1} = M_{r_n-\ell_n+1}$, we have $\Delta_{r_n,1}^{[\ell_n]} f = 0$. The mixing rate in~\eqref{eq:mixing:ellr} together with Lemma~\ref{lem:clipping:2} then imply
\[
  \Delta_{r_n,1}^{[\ell_n]} f = o_p(1), \qquad n \to \infty.
\]
Lyapounov's inequality and the asymptotic moment bound \eqref{eq:pnormp} then ensure that
\begin{equation}
\label{eq:Delta:Lp}
  \lim_{n \to \infty} \Exp \left[ \abs{ \Delta_{r_n,1}^{[\ell_n]} f }^{2+\delta} \right] = 0,
  \qquad f \in \mathcal{F}_2(\alpha_-, \alpha_+).
\end{equation}

Recall Lemma~3.11 in \cite{DehPhi02}: for random variables $\xi$ and $\eta$ and for numbers $p, q \in [1, \infty]$ such that $1/p + 1/q < 1$,
\begin{equation}
\label{eq:alphaCov}
  \abs{ \cov(\xi, \eta) } 
  \le 
  10 \, \norm{\xi}_p \, \norm{\eta}_q \, \{\alpha(\sigma(\xi), \sigma(\eta))\}^{1 - 1/p - 1/q},
\end{equation}
%where $\sigma(\,\cdot\,)$ denotes the $\sigma$-field generated by its argument and 
where $\alpha( \mathcal{A}_1, \mathcal{A}_2 )$ denotes the strong mixing coefficient between two $\sigma$-fields $\mathcal{A}_1$ and $\mathcal{A}_2$. Use inequality \eqref{eq:alphaCov} with $p=q=2+\delta$ to bound the covariance terms in \eqref{eq:MSEdecomp}: 
\[
  \Exp[ (\Delta_n f)^2 ]
  \le 
    3 \, \norm{ \Delta_{r_n,1}^{[\ell_n]} f }_2^2 
  + 20 \, k_n \, \norm{ \Delta_{r_n,1}^{[\ell_n]} f }_{2+\delta}^2 \, \{ \alpha(r_n) \}^{\delta/(2+\delta)}.
\]
In view of \eqref{eq:Delta:Lp} and Condition~\ref{cond:alpha}, the right-hand side converges to zero since $\omega<2/\delta$.
%Recall $\omega > 0$ in Condition~\ref{cond:alpha}. In view of \eqref{eq:Delta:Lp}, a sufficient condition for the right-hand side to converge to zero is that $p$ is big enough: $1 / (1 - 2/p)$ should be smaller than $1 + \omega$.
\end{proof}

\subsection{Proof of Theorem~\ref{theo:ml}}%{Proofs for Section~\ref{sec:iid}}
\label{subsec:proofs:iid}

\begin{proof}[Proof of Theorem~\ref{theo:ml}]
We apply Theorem~\ref{theo:asydis:stat}. To this end, we verify its conditions.

\textit{Proof of Condition~\ref{cond:DA}.}
The second-order regular variation condition \eqref{eq:SV:2} implies the first-order one in \eqref{eq:RV}, which is in turn equivalent to weak convergence of partial maxima as in \eqref{eq:DA}. Condition~\ref{cond:DA} follows with scaling sequence $\sigma_n = a_n$. The latter sequence is regularly varying \citep[Proposition~1.11]{Res87} with index $1/\alpha_0$, which implies that $\lim_{n \to \infty} a_{m_n} / a_n = 1$ whenever $\lim_{n \to \infty} m_n / n = 1$.

\textit{Proof of Condition~\ref{cond:small}.}
For any real $c$ we have, since $\log F(c) < 0$ and since $\log(k_n) = o(r_n)$ by \eqref{eq:klogk},
\begin{align*}
  \Pr[ \min( M_{r_n,1}, \ldots, M_{r_n,k_n} ) \le c ] 
  \le k_n \, F^{r_n}( c )
  = \exp \{ \log(k_n) + r_n \log F(c) \}
  \to 0, \qquad n \to \infty.
\end{align*}

\textit{Proof of Condition~\ref{cond:alpha}.}
Trivial, since $\alpha(\ell) = 0$ for integer $\ell \ge 1$.

\textit{Proof of Condition~\ref{cond:moment}.}
This follows from Lemma~\ref{lem:blockmoments} in the supplementary material (which in turn is a variant of Proposition~2.1(i) in \citealp{Res87}), where we prove that the sufficient Condition~\eqref{eq:momall} is satisfied.

\textit{Proof of Condition~\ref{cond:bias}.}
Recall Remark~\ref{rem:secor} and therein the functions $L$ and $g(u)=A(u)L(u)$. We begin by collecting some non-asymptotic bounds on the function~$L$. Fix $\delta\in(0,\alpha_0)$. Potter's theorem (\citealp{BGT87}, Theorem 1.5.6) implies that there exists some constant $x'(\delta)>0$  such that, for all $u \ge x'(\delta)$ and $x \ge x'(\delta)/u$, 
\begin{align} \label{eq:L1b}
\frac{L(u)}{L(ux)} \le (1+\delta)\, \max( x^{-\delta}, x^{\delta}) . %\quad \Longleftrightarrow \quad \frac{L(uy)}{L(u)} \ge c(\delta) y^{\delta} 
\end{align}

As a consequence of Theorem B.2.18 in \cite{dHF06}, accredited to \cite{Dre98}, there exists some further constant $x''(\delta)>0$ such that, for all $u \ge x''(\delta)$ and $x \ge x''(\delta)/u$, 
\begin{align} \label{eq:L2b}
\left| \frac{L(ux) - L(u)}{g(u)} \right| 
\le 
\ c(\delta) \, \max( x^{\rho - \delta}, x^{\rho+\delta}) ,
\end{align}
for some constant $c(\delta)>0$. Define $x(\delta)=\max\{x'(\delta), x''(\delta),1\}$.

We are going to show Condition~\ref{cond:bias} for $c=x(\delta)$ and $\sigma_{r_n}=a_{r_n}$. For $i=1, \dots, k_n$, define $X_{n,i} = M_{r_n,i} \vee x(\delta)$. 
Let $P_n$ denote the common distribution of the rescaled, truncated block maxima $X_{n,i}/a_{r_n}$ and let $P$ denote the Fr\'echet($\alpha_0, 1$) distribution. Write $B_n = \sqrt{k_n} (P_n - P)$ and define the three-by-one vector $\bm{\beta}$ by
\begin{equation} \label{eq:beta}
  \bm{\beta}
%   = 
%   B_{\bm Y}(\rho)
  = 
    \frac{\lambda }{\abs{\rho} \alpha_0} 
      \begin{pmatrix}
	2 - \gamma - \Gamma(2+\tfrac{\abs{\rho}}{\alpha_0}) - \Gamma'(2+\tfrac{\abs{\rho}}{\alpha_0}) \\[1ex]
	\alpha_0  \Gamma(2+\tfrac{\abs{\rho}}{\alpha_0}) - \alpha_0 \\[1ex]
	1 - \Gamma(1+\tfrac{\abs{\rho}}{\alpha_0})
      \end{pmatrix} 
\end{equation}
if $\rho < 0$ and by
\begin{equation*}
  \bm{\beta}
  =
    \frac{\lambda }{\alpha_0^2} 
      \begin{pmatrix}  
	\gamma - (1-\gamma)^2 - \pi^2/6 \\
	\alpha_0(1 - \gamma) \\
	\gamma
      \end{pmatrix} 
\end{equation*}
if $\rho = 0$.
% \[
% \Gb_n = \sqrt {k_n} (\Pb_n - P_n) + \sqrt{k_n} (P_n - P) \equiv \Ub_n + B_n
% \] 
% By the classical Lindeberg central limit theorem for row wise independent triangular arrays, the vector $( \Ub_nx^{-\alpha_0} \log x, \Ub_n x^{-\alpha_0} , \Ub_n \log x \bigr)$ converges to a trivariate centered normal distribution with covariance $\Sigma_{\bm Y}$ as specified in \eqref{eq:Sigma}. Here, the convergence of the second moments follows from Theorem 2.20 in \cite{Van98} together with Lemma~\ref{lem:blockmoments} and Lemma~\ref{lem:cov} below. H\"older's inequality can be used to deduce the Lindeberg condition. 
%Consider the bias term $B_n$. 
We will show that
\begin{equation}
\label{eq:bias:Y:iid}
  \lim_{n \to \infty} 
  \bigl( B_n x^{-\alpha_0} \log x, \, B_n x^{-\alpha_0}, \, B_n \log x \bigr)^T
  = \bm{\beta}.
\end{equation}
Elementary calculations yield that $M(\alpha_0) \, \bm{\beta} = \lambda \, B( \alpha_0, \rho)$ as required in \eqref{eq:bias:iid}.

Equation~\eqref{eq:bias:Y:iid} can be shown coordinatewise. 
We begin by some generalities. For any $f \in \Hc$ as in \eqref{eq:H}, we can write, for arbitrary $x,x_0\in(0,\infty)$,
\[
  f(x) = 
  \begin{cases}
    f(x_0) - \int_x^{x_0}  f'(y) \, dy, & \text{if $0 < x \le x_0$}, \\[1ex]
    f(x_0) + \int_{x_0}^x f'(y) \, dy, & \text{if $x_0 < x < \infty$}.
  \end{cases}
\]
By Fubini's theorem, with $G_n$ and $G$ denoting the cdf-s of $P_n$ and $P$, respectively, 
\begin{align*}
  P f 
  &= \int_{(0, x_0]} f(x) \, dP(x) + \int_{(x_0, \infty)} f(x) \, dP(x) \\
  &= f(x_0) 
  - \int_{x \in (0, x_0]} \int_{y=x}^{x_0} f'(y) \, dy \, dP(x)  + \int_{x \in (x_0, \infty)} \int_{y=x_0}^x f'(y) \, dy \, dP(x) \\
  &= f(x_0)
  - \int_{y=0}^{x_0} \int_{x \in (0, y]} \, dP(x) \, f'(y) \, dy+ \int_{y=x_0}^\infty \int_{x \in (y, \infty)} \, dP(x) \, f'(y) \, dy \\
  &= f(x_0) - \int_0^{x_0} G(y) \, f'(y) \, dy + \int_{x_0}^\infty \left\{ 1 - G(y) \right \} \, f'(y) \, dy,
\end{align*}
and the same formula holds with $P$ and $G$ replaced by $P_n$ and $G_n$, respectively.  We find that
\[
   B_n f = \sqrt {k_n}(P_n-P) f
  = - \int_0^\infty  \sqrt{k_n} \left\{ G_n(y) - G(y) \right \} \, f'(y) \, dy.
\]
Note that
\begin{align*}
  G(y) &= \exp ( - y^{-\alpha_0} ) \, \1_{(0, \infty)}(y),  \qquad
  G_n(y) = F^{r_n}(a_{r_n} y) \, \1_{[x(\delta)/a_{r_n}, \infty)}(y),
\end{align*}
From the definition of $L$ in \eqref{eq:defL}, we can write, for  $y \ge x(\delta) /a_{r_n}$,
\begin{equation*}
%\label{eq:GnL}
  G_n(y) = \exp \left( - y^{-\alpha_0} r_n \{ - \log F(a_{r_n}) \}  \, \frac{L(a_{r_n}y)}{L(a_{r_n})} \right).
\end{equation*}
For the sake of brevity, we will only carry out the subsequent parts of the proof in the case where $F$ is ultimately continuous, so that ${r_n} \, \{ -\log F(a_{r_n}) \} = 1$ for all sufficiently large $n$. In that case,  $B_n f =J_{n1}(f)+J_{n2}(f)$ %+I_{nj3}(f)$ 
where
\begin{align*}
J_{n1}(f) &=
\sqrt{k_n} \int_0^{x(\delta)/a_{r_n}} \exp(-y^{-\alpha_0})f'(y)\, dy, \\
J_{n2}(f) &= - \sqrt{k_n} \int_{x(\delta)/a_{r_n}} ^\infty \left[ \exp\left( -y^{-\alpha_0} \frac{L(a_{r_n}y)}{L(a_{r_n})} \right) - \exp(-y^{-\alpha_0})  \right]
f'(y)\, dy, 
%I_{nj3}(f) &= - \sqrt{k_n} \int_1 ^\infty \left[ \exp\left( -y^{-\alpha_0} \frac{L(a_{r_n}y)}{L(a_{r_n})} \right) - \exp(y^{-\alpha_0})  \right]
%f'(y)\, dy.
\end{align*}

Let us first show that $J_{n1}(f)$ converges to $0$ for any $f \in  \Hc$. For that purpose, note that any $f \in  \Hc$ satisfies $|f'(x)| \le K x^{-\alpha_0-\eps-1}$ for any $\eps<1$ and for some constant $K=K(\eps)>0$. As a consequence, by \eqref{eq:ka}, for sufficiently large $n$,
\[
\max_{ f\in \Hc } |J_{n1}(f) | \le \{ \lambda+o(1) \} \frac{K}{A(a_{r_n})} \int_0^{x(\delta) /a_{r_n}}  \exp(-y^{-\alpha_0}) y^{-\alpha_0 - \eps - 1} \, dy.
\]
Since $A(x)$ is bounded from below by a multiple of $x^{\rho - \eps}$ for sufficiently large $x$ (by Remark~\ref{rem:secor} and Potter's theorem), the expression on the right-hand side of the last display can be easily seen to converge to $0$ for $n\to\infty$.

For the treatment of $J_{n2}$, note that
\begin{align*}
J(f,\rho) &\equiv \int_0^\infty  h_\rho(y) \exp\left( -y^{-\alpha_0}  \right) y^{-\alpha_0}f'(y)  \, dy \\
&=
\begin{cases} 
\int_0^\infty  h_\rho(y) \exp\left( -y^{-\alpha_0}  \right) y^{-2\alpha_0-1} (1-\alpha_0\log y)  \, dy &, f(y)=y^{-\alpha_0} \log y\\
\int_0^\infty  h_\rho(y) \exp\left( -y^{-\alpha_0}  \right)  (-\alpha_0y^{-2\alpha_0-1}) \, dy &, f(y)=y^{-\alpha_0} \\
\int_0^\infty  h_\rho(y) \exp\left( -y^{-\alpha_0}  \right)  y^{-\alpha_0-1} \, dy &, f(y)= \log y
\end{cases} \\
&=
\begin{cases} 
\Exp[h_\rho(Y) Y^{-\alpha_0}(\alpha_0^{-1}- \log Y) ] &, f(y)=y^{-\alpha_0} \log y\\
- \Exp[h_\rho(Y) Y^{-\alpha_0}]  &, f(y)=y^{-\alpha_0} \\
\alpha_0^{-1} \Exp[h_\rho(Y) ]  &, f(y)= \log y,
\end{cases}
\end{align*}
%For the treatment of $J_{n2}$, define
%\begin{align*}
%J(f,\rho) &\equiv \int_0^\infty  h_\rho(y) \exp\left( -y^{-\alpha_0}  \right) y^{-\alpha_0}f'(y)  \, dy 
%\end{align*}
%and note that
%\begin{align*}
%\begin{pmatrix}
%J(x^{-\alpha_0} \log x, \rho) \\
%J(x^{-\alpha_0} , \rho) \\
%J(\log x, \rho)
%\end{pmatrix}
%&=
%\begin{pmatrix} 
%\int_0^\infty  h_\rho(y) \exp\left( -y^{-\alpha_0}  \right) y^{-2\alpha_0-1} (1-\alpha_0\log y)  \, dy \\
%\int_0^\infty  h_\rho(y) \exp\left( -y^{-\alpha_0}  \right)  (-\alpha_0y^{-2\alpha_0-1}) \, dy  \\
%\int_0^\infty  h_\rho(y) \exp\left( -y^{-\alpha_0}  \right)  y^{-\alpha_0-1} \, dy 
%\end{pmatrix} \\
%&=
%\begin{pmatrix} 
%\Exp[h_\rho(Y) Y^{-\alpha_0}(\alpha_0^{-1}- \log Y) ] \\
%- \Exp[h_\rho(Y) Y^{-\alpha_0}]   \\
%\alpha_0^{-1} \Exp[h_\rho(Y) ]  
%\end{pmatrix}
%\end{align*}
where $Y$ denotes a Fr\'echet$(\alpha_0,1)$ random variable. By Lemma~\ref{lem:moments} this implies
\begin{align*}
J(x^{-\alpha_0} \log x, \rho)
&=
\frac{1}{\rho\alpha_0} \left\{ \Gamma(2+\tfrac{\abs{\rho}}{\alpha_0})+\Gamma'(2+\tfrac{\abs{\rho}}{\alpha_0})-1 - \Gamma'(2) \right\} \\
&=
\frac{1}{\abs{\rho}\alpha_0} \left\{ 2 - \gamma -  \Gamma(2+\tfrac{\abs{\rho}}{\alpha_0}) - \Gamma'(2+\tfrac{\abs{\rho}}{\alpha_0}) \right\},  \\
J(x^{-\alpha_0}, \rho) 
&= 
\frac{1}{\rho} \left\{ \Gamma(2) -  \Gamma(2+\tfrac{\abs{\rho}}{\alpha_0})  \right\} 
=
\frac{1}{\abs \rho} \left\{  \Gamma(2+\tfrac{\abs{\rho}}{\alpha_0})  -1 \right\}, \\
J(\log x, \rho)
&=
\frac{1}{\rho \alpha_0} \left\{ \Gamma(1+\tfrac{\abs{\rho}}{\alpha_0}) -1 \right\} 
=
\frac{1}{\abs{\rho} \alpha_0} \left\{ 1- \Gamma(1+\tfrac{\abs{\rho}}{\alpha_0}) \right\}
\end{align*}
for $\rho<0$ and
\begin{align*}
J(x^{-\alpha_0} \log x, 0)
&=
- \frac{1}{\alpha_0^2} \left\{ \Gamma'(2)+ \Gamma''(2) \right\} 
= 
 \frac{1}{\alpha_0^2} \left\{ \gamma - (1-\gamma)^2 - \pi^2/6 \right\}, \\
J(x^{-\alpha_0}, 0) 
&= 
\frac{1}{\alpha_0}  \Gamma'(2)  
= \frac{1-\gamma}{\alpha_0},
\\
J(\log x, 0)&=
- \frac{1}{ \alpha_0^2} \Gamma'(1) 
=
 \frac{\gamma}{\alpha_0^2}.
\end{align*}
Hence, $\bm\beta= \lambda \bigl( J(x^{-\alpha_0} \log x, \rho), J(x^{-\alpha_0} , \rho), J( \log x, \rho)\bigr)^T$ and it is therefore sufficient to show that, for any $f\in \Hc$,
\begin{align} \label{eq:in2}
%&\left| J_{n2}(f) + \lambda\, \kappa \, \int_0^\infty  h_\rho(y) \exp\left( -y^{-\alpha_0}  \right) y^{-\alpha_0} f'(y) \, dy  \right| 
%=o(1), 
& J_{n2}(f) \to  \lambda\, J(f, \rho)
%\label{eq:in3}
%&\left| J_{n3}(f) + \lambda\, \kappa \, \int_1^\infty  h_\rho(y) \exp\left( -y^{-\alpha_0}  \right) y^{-\alpha_0} f'(y) \, dy  \right| 
%=o(1),
\end{align}
as $n\to\infty$. By the mean value theorem, we can write $J_{n2}(f)$ as
\[
J_{n2}(f) =   
\sqrt{k_n} A(a_{r_n}) \int_{x(\delta)/a_{r_n}}^\infty
%\frac{1}{A(a_{r_n}) }\left( \frac{L(a_{r_n}y)}{L(a_{r_n})}  - 1\right) 
\frac{L(a_{r_n}y) - L(a_{r_n})}{A(a_{r_n})L(a_{r_n})} 
 \exp\left( -y^{-\alpha_0} \xi_n(y)  \right) y^{-\alpha_0} f'(y)\, dy
\]
for some $\xi_n(y)$ between ${L(a_{r_n}y)}/{L(a_{r_n})}$ and $1$. For $n\to\infty$, the factor in front of this integral converges to $\lambda$ by assumption \eqref{eq:ka}, while the integrand in this integral converges to
\[
 h_\rho(y) \exp\left( -y^{-\alpha_0}  \right) y^{-\alpha_0} f'(y),
\]
pointwise in $y\in(0,\infty)$, by Condition~\ref{cond:secor}. Hence, the convergence in \eqref{eq:in2} follows from dominated convergence if we show that 
\[
f_n(y) =  \ind\left(y>\tfrac{x(\delta)}{a_{r_n}}\right) \left|
\frac{L(a_{r_n}y) - L(a_{r_n})}{A(a_{r_n})L(a_{r_n})} \right|
 \exp\left( -y^{-\alpha_0} \xi_n(y)  \right) y^{-\alpha_0}f'(y)
\]
can be bounded by an integrable function on $(0,\infty)$. We split the proof into two cases. 

First,  for any $1 \ge y \ge x(\delta)/a_{r_n}$, 
\[
\left| \frac{L(a_{r_n}y) - L(a_{r_n})}{A(a_{r_n})L(a_{r_n})} \right| \le c(\delta) y^{\rho-\delta}
\]
from \eqref{eq:L2b} and 
\begin{align*} 
%\label{eq:xin}
\xi_n(y)  
\ge
 \min\left(1, \frac{L(a_{r_n}y)}{L(a_{r_n})}\right) 
\ge
(1+\delta)^{-1} y^{\delta} 
\end{align*}
from \eqref{eq:L1b}. Moreover, for any $f\in \Hc$, the function $f'(y)$ is bounded by a multiple of $y^{-\alpha_0 - \delta - 1}$ for $y\le 1$.
Therefore, for any $y\in(0,1)$,
\[
f_n(y) \le c'(\delta) \exp\{ - (1+\delta)^{-1} y^{-\alpha_0+\delta} \} y^{-2\alpha_0-2\delta-1+\rho} 
\]
and the function on the right is integrable on $(0,1)$ since $\delta< \alpha_0$.

Second, for $y\in[1,\infty)$, we have 
\[
\left| \frac{L(a_{r_n}y) - L(a_{r_n})}{A(a_{r_n})L(a_{r_n})} \right| \le c(\delta) y^{\rho+\delta}
\]
from \eqref{eq:L2b} and 
\begin{align*} 
%\label{eq:xin}
\xi_n(y)  
\ge
 \min\left(1, \frac{L(a_{r_n}y)}{L(a_{r_n})}\right) 
\ge
(1+\delta)^{-1} y^{-\delta} 
\end{align*}
from \eqref{eq:L1b}. Moreover, $f'(y)$ is bounded by a multiple of $y^{-1}$ for any $y\ge1$ and any $f\in  \Hc$. Therefore,
\[
f_{n}(y) \le c''(\delta) \, y^{-\alpha_0-1+\rho+\delta}
\]
which is easily integrable on $[1,\infty)$. %\qed
\end{proof}

\section{Auxiliary results}
\label{sec:aux}

Let $\Gamma(x) = \int_0^\infty t^{x-1}e^{-t} \, \diff t$ be the gamma function and let $\Gamma'$ and $\Gamma''$ be its first and second derivative, respectively. 
All proofs for this section are given in Section~\ref{sec:proofs:aux} in the supplementary material.

\begin{lemma}[Moments]
\label{lem:moments}
Let $P$ denote the Fr\'echet distribution with parameter vector $(\alpha_0, 1)$, for some $\alpha_0 \in (0, \infty)$. For all $\alpha \in (-\alpha_0, \infty)$,
\begin{align*}
  \int_0^\infty x^{-\alpha} \, \diff P(x) &= \Gamma(1+\alpha/\alpha_0), \\
  \int_0^\infty x^{-\alpha} \log(x) \, \diff P(x) &= -\frac{1}{\alpha_0} \Gamma'(1 + \alpha/\alpha_0), \\
  \int_0^\infty x^{-\alpha} (\log(x))^2 \, \diff P(x) &= \frac{1}{\alpha_0^2} \, \Gamma''(1 + \alpha/\alpha_0).
\end{align*}
\end{lemma}

\begin{lemma}[Covariance matrix]
\label{lem:cov}
Let $X$ be a random variable whose distribution is Fr\'echet with parameter vector $(\alpha_0, 1)$. The covariance matrix of the random vector
$
  \bm{Y} = (Y_1, Y_2, Y_3)^T = \bigl( X^{-\alpha_0} \log(X), \, X^{-\alpha_0}, \, \log(X) \bigr)^T
$
is equal to
\begin{align*}
  \cov( \bm{Y} )
  &=
  \frac{1}{\alpha_0^2}
  \begin{pmatrix}
    1-4\gamma+\gamma^2+\pi^2/3 & \alpha_0(\gamma - 2) & \pi^2/6-\gamma \\
    \alpha_0(\gamma - 2) & \alpha_0^2 & -\alpha_0 \\
    \pi^2/6-\gamma &-\alpha_0 &  \pi^2/6
  \end{pmatrix}.
\end{align*}
\end{lemma}

%\begin{proof}
%Applying the formulas in Lemma~\ref{lem:moments} with $\alpha \in \{ 0, \alpha_0, 2 \alpha_0 \}$, we find
%\begin{align*}
%  \var( Y_1 ) &= \alpha_0^{-2} \, \bigl\{ \Gamma''(3) - (\Gamma'(2))^2 \bigr\} =  \alpha_0^{-2} (1-4\gamma+\gamma^2+\pi^2/3)  , \\
%  \var( Y_2 ) &= \Gamma(3) - (\Gamma(2))^2 = 1, \\
%  \var( Y_3 ) &= \alpha_0^{-2} \, \bigl( \Gamma''(1) - (\Gamma'(1))^2 \bigr) = \alpha_0^{-2} \pi^2/6,
%\end{align*}
%as well as
%\begin{align*}
%  \cov( Y_1, Y_2 ) &= \alpha_0^{-1} \, \bigl( (-\Gamma'(3)) - (-\Gamma'(2)) \Gamma(2) \bigr) = \alpha_0^{-1} \, (\gamma - 2), \\
%  \cov( Y_1, Y_3 ) &= \alpha_0^{-2} \, \bigl( \Gamma''(2) - (-\Gamma'(2)) (-\Gamma'(1)) \bigr) = \alpha_0^{-2} (\pi^2/6 - \gamma), \\
%  \cov( Y_2, Y_3 ) &= \alpha_0^{-1} \, \bigl( (-\Gamma'(2)) - \Gamma(2) (-\Gamma'(1)) \bigr) = -\alpha_0^{-1}.\qedhere
%\end{align*}
%\end{proof}

\begin{lemma}[Fisher information]
\label{lem:fisher}
Let $P_\theta$ denote the Fr\'echet distribution with parameter $\theta=(\alpha,\sigma)\in(0,\infty)^2$. The Fisher information $I_{\theta} = P_\theta (\dot{\ell}_{\theta} \dot{\ell}_{\theta}^T)$ is given by 
\[
I_{\theta} =
\left( 
\begin{array}{cc}  \iota_{11} & \iota_{12} \\ \iota_{21} & \iota_{22} \end{array} \right)
=
\left( 
%\begin{array}{cc} \frac{5-2\gamma+\gamma^2+\pi^2/6}{\alpha^2} & \frac{1-\gamma}{\sigma} \\ \frac{1-\gamma}{\sigma} & (\frac{\alpha}{\sigma})^2 \end{array} \right)
\begin{array}{cc} \{ (1-\gamma)^2+\pi^2/6\}/{\alpha^2} & (1-\gamma)/{\sigma} \\ (1-\gamma)/{\sigma} & \alpha^2/\sigma^2 \end{array} \right).
\]
Its inverse is given by
\[
I_{\theta}^{-1} 
= \frac{6}{\pi^2} \left( 
\begin{array}{cc} \alpha^2 & (\gamma-1) {\sigma} \\ (\gamma-1){\sigma} &  (\sigma/\alpha)^2 \{ (1-\gamma)^2 +\pi^2/6\} \end{array} \right).
\]
\end{lemma}

%\begin{proof} If $X\sim P_{(\alpha,\sigma)}$, then $Z=X/\sigma \sim P_{(\alpha,1)}$. Therefore, by \eqref{eq:score:alpha} and  Lemma~\ref{lem:moments},
%\begin{align*}
%\iota_{11} &= \Exp\big[\{ \alpha^{-1} + (Z^{-\alpha} - 1) \log Z \}^2 \big] \\
%&= 
%\frac{1}{\alpha^2}\big[ 1 - 2 \{ \Gamma'(2) - \Gamma'(1) \} + \{ \Gamma''(3) - 2 \Gamma''(2) + \Gamma''(1) \} \big]  \\
%&= 
%\frac{1}{\alpha^2} \{ (1-\gamma)^2 + \pi^2/6\}
%\end{align*}
%Similarly, by \eqref{eq:score:alpha} and  \eqref{eq:score:sigma}, 
%\begin{align*}
%\iota_{12} 
%&= 
%\frac{\alpha}{\sigma} \Exp\big[(1-Z^{-\alpha}) \{ \alpha^{-1} + (Z^{-\alpha} - 1) \log Z \} \big] \\
%&=  
%\frac{\alpha}{\sigma} \left[ \alpha^{-1} \{ \Gamma(1) - \Gamma(2) \} + \alpha^{-1} \{ \Gamma'(1) - 2\Gamma'(2) + \Gamma'(3) \} \right\}  \\
%&= 
%\frac{1-\gamma}{\sigma}. 
%\end{align*}
%Finally, 
%\begin{align*}
%\iota_{22} 
%&= 
%\frac{\alpha^2}{\sigma^2} \Exp[(1-Z^{-\alpha})^2 ]  
%= 
%\frac{\alpha^2}{\sigma^2} \{ \Gamma(1)-2\Gamma(2)+\Gamma(3)\} 
%=
%\frac{\alpha^2}{\sigma^2}.\qedhere
%\end{align*}
%\end{proof}

\section*{Acknowledgments}
The authors would like to thank two anonymous referees and an Associate Editor for their constructive comments on an earlier version of this manuscript, and in particular for suggesting a sharpening of Conditions~\ref{cond:LLN} and~\ref{cond:CLT} and for pointing out the connection between Equations~\eqref{eq:ka} and~\eqref{eq:klogk}.

The research by A.\ B\"ucher  has been supported by the Collaborative Research Center ``Statistical modeling of nonlinear dynamic processes'' (SFB 823, Project A7) of the German Research Foundation, which is gratefully acknowledged. Parts of this paper were written when A.\ B\"ucher was a visiting professor at TU Dortmund University.

J. Segers gratefully acknowledges funding by contract ``Projet d'Act\-ions de Re\-cher\-che Concert\'ees'' No.\ 12/17-045 of the ``Communaut\'e fran\c{c}aise de Belgique'' and by IAP research network Grant P7/06 of the Belgian government (Belgian Science Policy).

\bibliographystyle{imsart-nameyear}

\bibliography{biblio}

\begin{thebibliography}{25}
% BibTex style file: imsart-nameyear.bst, 2013-01-28
% Default style options (sort=1,type=nameyear).
% Used options (sort=1,type=nameyear).

\bibitem[\protect\citeauthoryear{Balakrishnan and Kateri}{2008}]{BalKat08}
\begin{barticle}[author]
\bauthor{\bsnm{Balakrishnan},~\bfnm{N.}\binits{N.}} \AND
  \bauthor{\bsnm{Kateri},~\bfnm{M.}\binits{M.}}
(\byear{2008}).
\btitle{On the maximum likelihood estimation of parameters of {Weibull}
  distribution based on complete and censored data}.
\bjournal{Statistics \& Probability Letters}
\bvolume{78}
\bpages{2971--2975}.
\end{barticle}
\endbibitem

\bibitem[\protect\citeauthoryear{Bingham, Goldie and Teugels}{1987}]{BGT87}
\begin{bbook}[author]
\bauthor{\bsnm{Bingham},~\bfnm{N.~H.}\binits{N.~H.}},
  \bauthor{\bsnm{Goldie},~\bfnm{C.~M.}\binits{C.~M.}} \AND
  \bauthor{\bsnm{Teugels},~\bfnm{J.~L.}\binits{J.~L.}}
(\byear{1987}).
\btitle{Regular Variation}.
\bpublisher{Cambridge University Press}, \baddress{Cambridge}.
\end{bbook}
\endbibitem

\bibitem[\protect\citeauthoryear{B\"ucher and Segers}{2014}]{BucSeg14}
\begin{barticle}[author]
\bauthor{\bsnm{B\"ucher},~\bfnm{Axel}\binits{A.}} \AND
  \bauthor{\bsnm{Segers},~\bfnm{Johan}\binits{J.}}
(\byear{2014}).
\btitle{Extreme value copula estimation based on block maxima of a multivariate
  stationary time series}.
\bjournal{Extremes}
\bvolume{17}
\bpages{495--528}.
\end{barticle}
\endbibitem

\bibitem[\protect\citeauthoryear{{B{\"u}cher} and {Segers}}{2016}]{BucSeg16}
\begin{barticle}[author]
\bauthor{\bsnm{{B{\"u}cher}},~\bfnm{A.}\binits{A.}} \AND
  \bauthor{\bsnm{{Segers}},~\bfnm{J.}\binits{J.}}
(\byear{2016}).
\btitle{{On the maximum likelihood estimator for the Generalized Extreme-Value
  distribution}}.
\bjournal{ArXiv e-prints}.
\end{barticle}
\endbibitem

\bibitem[\protect\citeauthoryear{Cai, de~Haan and Zhou}{2013}]{CaiDehZho13}
\begin{barticle}[author]
\bauthor{\bsnm{Cai},~\bfnm{Juan-Juan}\binits{J.-J.}}, \bauthor{\bparticle{de}
  \bsnm{Haan},~\bfnm{Laurens}\binits{L.}} \AND
  \bauthor{\bsnm{Zhou},~\bfnm{Chen}\binits{C.}}
(\byear{2013}).
\btitle{Bias correction in extreme value statistics with index around zero}.
\bjournal{Extremes}
\bvolume{16}
\bpages{173-201}.
\bdoi{10.1007/s10687-012-0158-x}
\end{barticle}
\endbibitem

\bibitem[\protect\citeauthoryear{de~Haan and Ferreira}{2006}]{dHF06}
\begin{bbook}[author]
\bauthor{\bparticle{de} \bsnm{Haan},~\bfnm{Laurens}\binits{L.}} \AND
  \bauthor{\bsnm{Ferreira},~\bfnm{Ana}\binits{A.}}
(\byear{2006}).
\btitle{Extreme Value Theory}.
\bseries{Springer Series in Operations Research and Financial Engineering}.
\bpublisher{Springer, New York}
\bnote{An introduction}.
\bdoi{10.1007/0-387-34471-3}
\bmrnumber{2234156 (2007g:62008)}
\end{bbook}
\endbibitem

\bibitem[\protect\citeauthoryear{Dehling and Philipp}{2002}]{DehPhi02}
\begin{bincollection}[author]
\bauthor{\bsnm{Dehling},~\bfnm{Herold}\binits{H.}} \AND
  \bauthor{\bsnm{Philipp},~\bfnm{Walter}\binits{W.}}
(\byear{2002}).
\btitle{Empirical process techniques for dependent data}.
In \bbooktitle{Empirical process techniques for dependent data}
\bpages{3--113}.
\bpublisher{Birkh\"auser Boston}, \baddress{Boston, MA}.
\bdoi{10.1007/978-1-4612-0099-4_1}
\bmrnumber{1958777 (2003k:62155)}
\end{bincollection}
\endbibitem

\bibitem[\protect\citeauthoryear{Dombry}{2015}]{Dom15}
\begin{barticle}[author]
\bauthor{\bsnm{Dombry},~\bfnm{Cl\'ement}\binits{C.}}
(\byear{2015}).
\btitle{Existence and consistency of the maximum likelihood estimators for the
  extreme value index within the block maxima framework}.
\bjournal{Bernoulli}
\bvolume{21}
\bpages{420--436}.
\end{barticle}
\endbibitem

\bibitem[\protect\citeauthoryear{Drees}{1998}]{Dre98}
\begin{barticle}[author]
\bauthor{\bsnm{Drees},~\bfnm{Holger}\binits{H.}}
(\byear{1998}).
\btitle{On smooth statistical tail functionals}.
\bjournal{Scand. J. Statist.}
\bvolume{25}
\bpages{187--210}.
\bdoi{10.1111/1467-9469.00097}
\bmrnumber{1614276 (99c:62059)}
\end{barticle}
\endbibitem

\bibitem[\protect\citeauthoryear{Drees}{2000}]{Dre00}
\begin{barticle}[author]
\bauthor{\bsnm{Drees},~\bfnm{Holger}\binits{H.}}
(\byear{2000}).
\btitle{Weighted approximations of tail processes for $\beta$-mixing random
  variables}.
\bjournal{Ann. Appl. Probab.}
\bvolume{10}
\bpages{1274--1301}.
\bdoi{10.1214/aoap/1019487617}
\end{barticle}
\endbibitem

\bibitem[\protect\citeauthoryear{Ferreira and de~Haan}{2015}]{dHF15}
\begin{barticle}[author]
\bauthor{\bsnm{Ferreira},~\bfnm{Ana}\binits{A.}} \AND \bauthor{\bparticle{de}
  \bsnm{Haan},~\bfnm{Laurens}\binits{L.}}
(\byear{2015}).
\btitle{On the block maxima method in extreme value theory: PWM estimators}.
\bjournal{Ann. Statist.}
\bvolume{43}
\bpages{276--298}.
\bdoi{10.1214/14-AOS1280}
\end{barticle}
\endbibitem

\bibitem[\protect\citeauthoryear{Gnedenko}{1943}]{gnedenko:1943}
\begin{barticle}[author]
\bauthor{\bsnm{Gnedenko},~\bfnm{B.}\binits{B.}}
(\byear{1943}).
\btitle{Sur la distribution limite du terme maximum d'une s\'erie al\'eatoire}.
\bjournal{Ann. of Math. (2)}
\bvolume{44}
\bpages{423--453}.
\bmrnumber{0008655 (5,41b)}
\end{barticle}
\endbibitem

\bibitem[\protect\citeauthoryear{Gumbel}{1958}]{Gum58}
\begin{bbook}[author]
\bauthor{\bsnm{Gumbel},~\bfnm{E.~J.}\binits{E.~J.}}
(\byear{1958}).
\btitle{Statistics of extremes}.
\bpublisher{Columbia University Press}, \baddress{New York}.
\bmrnumber{0096342 (20 \#\#2826)}
\end{bbook}
\endbibitem

\bibitem[\protect\citeauthoryear{Hosking, Wallis and
  Wood}{1985}]{hosking+w+w:1985}
\begin{barticle}[author]
\bauthor{\bsnm{Hosking},~\bfnm{J.~R.~M.}\binits{J.~R.~M.}},
  \bauthor{\bsnm{Wallis},~\bfnm{J.~R.}\binits{J.~R.}} \AND
  \bauthor{\bsnm{Wood},~\bfnm{E.~F.}\binits{E.~F.}}
(\byear{1985}).
\btitle{Estimation of the generalized extreme-value distribution by the method
  of probability-weighted moments}.
\bjournal{Technometrics}
\bvolume{27}
\bpages{251--261}.
\bdoi{10.2307/1269706}
\bmrnumber{797563}
\end{barticle}
\endbibitem

\bibitem[\protect\citeauthoryear{Hsing}{1991}]{Hsi91}
\begin{barticle}[author]
\bauthor{\bsnm{Hsing},~\bfnm{Tailen}\binits{T.}}
(\byear{1991}).
\btitle{On Tail Index Estimation Using Dependent Data}.
\bjournal{Ann. Statist.}
\bvolume{19}
\bpages{1547--1569}.
\bdoi{10.1214/aos/1176348261}
\end{barticle}
\endbibitem

\bibitem[\protect\citeauthoryear{Leadbetter}{1983}]{leadbetter:1983}
\begin{barticle}[author]
\bauthor{\bsnm{Leadbetter},~\bfnm{M.~R.}\binits{M.~R.}}
(\byear{1983}).
\btitle{Extremes and local dependence in stationary sequences}.
\bjournal{Z. Wahrsch. Verw. Gebiete}
\bvolume{65}
\bpages{291--306}.
\bdoi{10.1007/BF00532484}
\bmrnumber{722133 (85b:60033)}
\end{barticle}
\endbibitem

\bibitem[\protect\citeauthoryear{Marohn}{1994}]{marohn:1994}
\begin{bincollection}[author]
\bauthor{\bsnm{Marohn},~\bfnm{F.}\binits{F.}}
(\byear{1994}).
\btitle{On testing the {Exponential} and {Gumbel} distribution}.
In \bbooktitle{Extreme Value Theory and Applications}
\bpages{159--174}.
\bpublisher{Kluwer Academic Publishers}.
\end{bincollection}
\endbibitem

\bibitem[\protect\citeauthoryear{Mikosch and
  St{\u{a}}ric{\u{a}}}{2000}]{MikSta00}
\begin{barticle}[author]
\bauthor{\bsnm{Mikosch},~\bfnm{Thomas}\binits{T.}} \AND
  \bauthor{\bsnm{St{\u{a}}ric{\u{a}}},~\bfnm{C{\u{a}}t{\u{a}}lin}\binits{C.}}
(\byear{2000}).
\btitle{Limit theory for the sample autocorrelations and extremes of a {GARCH}
  {$(1,1)$} process}.
\bjournal{Ann. Statist.}
\bvolume{28}
\bpages{1427--1451}.
\bdoi{10.1214/aos/1015957401}
\bmrnumber{1805791 (2002c:62156)}
\end{barticle}
\endbibitem

\bibitem[\protect\citeauthoryear{Peng}{1998}]{Pen98}
\begin{barticle}[author]
\bauthor{\bsnm{Peng},~\bfnm{L.}\binits{L.}}
(\byear{1998}).
\btitle{Asymptotically unbiased estimators for the extreme-value index}.
\bjournal{Statistics \& Probability Letters}
\bvolume{38}
\bpages{107 - 115}.
\bdoi{http://dx.doi.org/10.1016/S0167-7152(97)00160-0}
\end{barticle}
\endbibitem

\bibitem[\protect\citeauthoryear{Pickands}{1975}]{Pic75}
\begin{barticle}[author]
\bauthor{\bsnm{Pickands},~\bfnm{James}\binits{J.}}
(\byear{1975}).
\btitle{Statistical Inference Using Extreme Order Statistics}.
\bjournal{Ann. Statist.}
\bvolume{3}
\bpages{119--131}.
\bdoi{10.1214/aos/1176343003}
\end{barticle}
\endbibitem

\bibitem[\protect\citeauthoryear{Prescott and Walden}{1980}]{prescott+w:1980}
\begin{barticle}[author]
\bauthor{\bsnm{Prescott},~\bfnm{P.}\binits{P.}} \AND
  \bauthor{\bsnm{Walden},~\bfnm{A.~T.}\binits{A.~T.}}
(\byear{1980}).
\btitle{Maximum likelihood estimation of the parameters of the generalized
  extreme-value distribution}.
\bjournal{Biometrika}
\bvolume{67}
\bpages{723--724}.
\bdoi{10.1093/biomet/67.3.723}
\bmrnumber{601119 (81m:62046)}
\end{barticle}
\endbibitem

\bibitem[\protect\citeauthoryear{Resnick}{1987}]{Res87}
\begin{bbook}[author]
\bauthor{\bsnm{Resnick},~\bfnm{Sidney~I.}\binits{S.~I.}}
(\byear{1987}).
\btitle{Extreme values, regular variation, and point processes}.
\bseries{Applied Probability. A Series of the Applied Probability Trust}
\bvolume{4}.
\bpublisher{Springer-Verlag, New York}.
\bdoi{10.1007/978-0-387-75953-1}
\bmrnumber{900810 (89b:60241)}
\end{bbook}
\endbibitem

\bibitem[\protect\citeauthoryear{Rootzén}{2009}]{Roo09}
\begin{barticle}[author]
\bauthor{\bsnm{Rootzén},~\bfnm{Holger}\binits{H.}}
(\byear{2009}).
\btitle{Weak convergence of the tail empirical process for dependent
  sequences}.
\bjournal{Stochastic Processes and their Applications}
\bvolume{119}
\bpages{468 - 490}.
\bdoi{http://dx.doi.org/10.1016/j.spa.2008.03.003}
\end{barticle}
\endbibitem

\bibitem[\protect\citeauthoryear{Smith}{1985}]{Smi85}
\begin{barticle}[author]
\bauthor{\bsnm{Smith},~\bfnm{Richard~L.}\binits{R.~L.}}
(\byear{1985}).
\btitle{Maximum Likelihood Estimation in a Class of Nonregular Cases}.
\bjournal{Biometrika}
\bvolume{72}
\bpages{67-90}.
\end{barticle}
\endbibitem

\bibitem[\protect\citeauthoryear{van~der Vaart}{1998}]{Van98}
\begin{bbook}[author]
\bauthor{\bparticle{van~der} \bsnm{Vaart},~\bfnm{A.~W.}\binits{A.~W.}}
(\byear{1998}).
\btitle{Asymptotic Statistics}.
\bseries{Cambridge Series in Statistical and Probabilistic Mathematics}
\bvolume{3}.
\bpublisher{Cambridge University Press}, \baddress{Cambridge}.
\bmrnumber{1652247 (2000c:62003)}
\end{bbook}
\endbibitem

\end{thebibliography}

% For some reason, the presence of numberless sections (Acknowledgments, References) disrupts the section numbering. Without the command below, the first section of the supplement would be numbered Appendix B again, just like the second appendix of the paper itself.
\addtocounter{section}{1}

\newpage

%\appendix
% \begin{center}
{\LARGE\noindent Supplementary Material on  \\ [2mm] ``Maximum likelihood estimation for the Fr\'echet distribution  \\ [3mm] based on block maxima extracted from a time series''}

\bigskip
\noindent AXEL B\"UCHER and JOHAN SEGERS
\smallskip

\noindent \textit{Ruhr-Universit\"at Bochum and Universit\'e catholique de Louvain}
\bigskip

% \end{center}

\noindent This supplementary material contains a lemma on moment convergence of block maxima used in the proof of Theorem~\ref{theo:ml} (in Section~\ref{sec:lemmom}),  the proof of Lemma~\ref{lem:movmax} (in Section~\ref{subsec:proofs:ex}) and  the proofs of auxiliary lemmas from Section~\ref{sec:aux} (in Section~\ref{sec:proofs:aux}) from the main paper.
Furthermore, we present additional Monte Carlo simulation results to quantify the finite-sample bias and variance of the maximum likelihood estimator (in Section~\ref{sec:simulextra}).

\section{Moment convergence of block maxima}\label{sec:lemmom}

The following Lemma is a variant of Proposition~2.1(i) in \cite{Res87}. It is needed in the proof of Theorem~\ref{theo:ml}.

\begin{lemma} 
\label{lem:blockmoments}
Let $\xi_1, \xi_2, \ldots$ be independent random variables with common distribution function $F$ satisfying \eqref{eq:RV}. Let $M_n = \max(\xi_1, \ldots, \xi_n)$. For every $\beta \in (-\infty, \alpha_0)$ and any constant $c >0$, we have
\[
  \limsup_{n \to \infty} 
  \Exp \bigl[ \bigl( (M_n \vee c) / a_n \bigr)^\beta \bigr]
  < 
  \infty.
\]
\end{lemma}

\begin{proof}[Proof of Lemma~\ref{lem:blockmoments}]
Since the case $\beta = 0$ is trivial, there are two cases to be considered: $\beta \in (-\infty, 0)$ and $\beta \in (0, \alpha_0)$. Write
$
  Z_n = (M_n \vee c) / a_n
$
and note that
\[
  \Pr[ Z_n < y ] = \Pr[ M_n \vee c < a_n y ] = F^n(a_n y) \, \1_{(c/a_n, \infty)}(y).
\]
%For simplicity, write $\alpha$ rather than $\alpha_0$ (sorry, dear reader).

% \medskip
% \noindent
\textit{Case $\beta \in (-\infty, 0)$.} 
% Since $\int_{[1, \infty)} x^\beta \, dP_n(x) \le 1$ for every $n$, we can restrict the range of integration to $(0, 1)$. 
We have
\begin{align*}
  \Exp [ Z_n^\beta ]
  = \int_0^\infty \Pr[ Z_n^\beta > x ] \, dx 
  = \int_0^\infty \Pr[ Z_n < x^{1/\beta} ] \, dx 
  &= \int_0^\infty \Pr[ Z_n < y ] \, \abs{\beta} \, y^{\beta - 1} \, dy \\
  &= \int_{c/a_n}^\infty F^n(a_n y) \, \abs{\beta} \, y^{\beta - 1} \, dy.
\end{align*}
We split the integration domain in two pieces. For $y \in (1, \infty)$, the integrand is bounded by $\abs{\beta} \, y^{\beta - 1}$, which integrates to unity. Hence we only need to consider the integral over $y \in (c/a_n, 1]$.
We have
\begin{align*}
  F^n(a_n y)
  &= \exp \{ n \log F(a_n y) \} 
%  &\le& \exp [ - n \{ 1 - F(a_n y) \} ] \\
  = \exp \left( - n \{ - \log F(a_n ) \} \, \frac{ - \log F(a_n y) }{  - \log F(a_n ) } \right).
\end{align*}
Fix $\delta \in (0, \alpha_0)$. By \eqref{eq:scaling}, we have $n \{ - \log F(a_n) \} \ge 1 - \delta$ for all $n$ larger than some $n(\delta)$. By Potter's theorem (\citealp{BGT87}, Theorem 1.5.6), there exists $x(\delta) > 0$ such that, for all $n$ such that $a_n \ge x(\delta)$ and for all $y \in (x(\delta) / a_n, 1]$,
\[
  \frac{- \log F(a_n)}{- \log F(a_n y)} \le (1 + \delta) \, y^{\alpha_0 - \delta}.
\]
Without loss of generality, assume $x(\delta) > c$. For all $y \in ( c / a_n, \, x(\delta) / a_n ]$, we have
\begin{align*}
  \frac{- \log F(a_n)}{ - \log F(a_n y)} 
  &\le \frac{- \log F(a_n)}{ - \log F(x(\delta))} 
  \le (1 + \delta) \, ( x(\delta) / a_n )^{\alpha_0 - \delta} 
  \le (1 + \delta) \, x(\delta)^{\alpha_0 - \delta} c^{\delta-\alpha_0} \, y^{\alpha_0 - \delta}.
\end{align*}
Combining the previous two displays, we see that there exists a constant $c(\delta) > 0$ such that
\[
  \frac{-\log F(a_n y)}{ - \log F(a_n)} \ge c(\delta) \, y^{-\alpha_0 + \delta}
\]
for all $y \in (c/a_n, 1]$.
We conclude that, for all sufficiently large $n$ and all $y \in (c/a_n, 1]$,
\[
  F^n(a_n y) \le \exp \left( - c(\delta) \, y^{-\alpha_0 + \delta} \right),
\]
where $c(\delta)$ is a positive constant, possibly different from the one in the previous equation. For such $n$, we have
\[
  \int_{c/a_n}^1 F^n(a_n y) \, \abs{\beta} \, y^{\beta - 1} \, dy
  \le \int_0^1 \exp \left( - c(\delta) \, y^{-\alpha_0 + \delta} \right) \, \abs{\beta} \, y^{\beta - 1} \, dy
  < \infty.
\]

% \medskip
% \noindent
\textit{Case $\beta \in (0, \alpha_0)$.}
Let $\delta > 0$ be sufficiently small such that $\beta + \delta < \alpha$. Let $x(\delta) > 0$ be as in Potter's theorem. Let $n(\delta)$ be sufficiently large such that $a_n \ge x(\delta) \vee c$ for all $n \ge n(\delta)$. Put $K = \sup_{n \ge 1} n \{  1- F(a_n) \}$, which is finite by \eqref{eq:scaling} and the fact that $-\log F(x) \sim 1-F(x)$ for $x\to\infty$. For $n \ge n(\delta)$, we have
\begin{align*}
  \Exp[ Z_n^\beta ]
%  &= \int_0^\infty \Pr[ Z_n^\beta > x ] \, dx \\
  = \int_0^\infty \Pr[ Z_n > x^{1/\beta} ] \, dx 
  &= \int_0^\infty \Pr[ M_n \vee c > a_n x^{1/\beta} ] \, dx \\
  &\le 1 + \int_1^\infty \Pr[ M_n > a_n x^{1/\beta} ] \, dx \\
  &\le 1 + \int_1^\infty n \{ 1 - F(a_n x^{1/\beta}) \} \, dx \\
  &\le 1 + K \int_1^\infty \frac{1 - F(a_n x^{1/\beta})}{1 - F(a_n)} \, dx.
\end{align*}
By Potter's theorem, the integral on the last line is bounded by
\[
  (1 + \delta) \int_c^\infty (x^{1/\beta})^{-\alpha_0 + \delta} \, dx.
\]
The latter integral is finite, since $(-\alpha_0 + \delta)/\beta < -1$.
\end{proof}

\section{Proofs for Section~\ref{sec:ex}}
 \label{subsec:proofs:ex}

% \ab{Supplement.}

\begin{proof}[Proof of Lemma~\ref{lem:movmax}] We only give a sketch proof for the case $p=2$, the general case being similar, but notationally more involved. 
Set $b_1=b$ and $b_2=1-b$, so that $b_{(2)}=b\vee (1-b)$. Clearly, 
\begin{align*}
\Pr( M_n \le x) 
%&= 
%\Prob\{ Z_1 \le x b^{-1}, Z_{0} \le x (1-b)^{-1}, 
%\dots, Z_{n}  \le xb^{-1}, Z_{n-1} \le x(1-b)^{-1}  \} \\
&= 
\Pr\{ Z_0 \le x(1-b)^{-1}, Z_1 \le x b_{(2)}^{-1}, \dots, Z_{n-1}\le xb_{(2)}^{-1}, Z_n \le xb^{-1}\} \\
&=
F( x(1-b)^{-1} ) \cdot F ( xb^{-1} )  \cdot F^{n-1}(xb_{(2)}^{-1}).
\end{align*}
As a consequence, with $b_{(1)}=b \wedge (1-b)$,
\begin{align*} %\label{eq:hn}
H_n(x) 
%&= \Prob( M_n/\sigma_n \le x ) \nonumber \\
&=
\Pr( M_n \le x b_{(2)} a_n )\nonumber  \\
&= F( a_nx \tfrac{b_{(2)}}{1-b} )\cdot  F(  a_nx \tfrac{b_{(2)}}{b} ) \cdot 
F^{n-1} (a_nx)\nonumber \\
&=
F( a_n x \tfrac{b_{(2)}}{b_{(1)}} ) \cdot F^n(a_nx).
\end{align*}
Since, by assumption, $F^n(xa_n)\to \exp(-x^{-\alpha_0})$, Condition~\ref{cond:DA} is satisfied.

Condition~\ref{cond:alpha} is trivial, since the process is $p$-dependent.

The proof of Condition~\ref{cond:moment} can be be carried out along the lines of the proof of Lemma~\ref{lem:blockmoments}.
For $\beta<0$, simply use that 
\begin{align*}
\Pr\{ (M_n\vee c)/\sigma_n \le x \} 
&=
H_n(x) \, \ind (x \ge c/\sigma_n)
\le 
F^n(xa_n) \, \ind (x \ge c/\sigma_n),
\end{align*}
while, for $\beta>0$,
\[
\Pr(M_n>\sigma_n x^{1/\beta} )  \le 2 n \cdot \Pr( Z_1 > \sigma_n  x^{1/\beta} b_{(2)}) 
=
2n\{ 1- F(a_nx^{1/\beta}) \},
\]
for any $x>1$.

Since $\log k_n = o(r_n)$, Condition~\ref{cond:small} follows from 
\begin{align*}
\Pr[ \min( M_{r_n,1}, \ldots, M_{r_n, k_n} ) \le c ] 
&\le k_n \Pr(M_{r_n} \le c) \\
& = \exp\{ \log k_n + (r_n-1) \log F(c b_{(2)}^{-1}) \} \cdot F( c(1-b)^{-1} ) \cdot F ( cb^{-1} ) .
\end{align*}

Finally, consider Condition~\ref{cond:bias}. As in the proof of Theorem~\ref{theo:ml}, write
\[
\sqrt{k_n} \left( \Exp\bigl[ f \bigl( (M_{r_n} \vee c) / \sigma_{r_n} \bigr)\bigr] - P f \right)
=
-\int_0^\infty \sqrt k_n \{ \tilde H_n(y)  - G(y) \} f'(y) \, dy,
\]
where $G(y) = \exp(-y^{-\alpha_0})$ and where 
\[
\tilde H_n(y) = \Prob\{ (M_{r_n} \vee c) / \sigma_{r_n}  \le y \} 
= A_n(y) G_n(y) 
\]
with 
\[
A_n(y) = F( a_{r_n} y\tfrac{b_{(2)}}{b_{(1)}} ), 
\qquad
G_n(y) = F^{r_n}(ya_{r_n}) \ind (y \ge c/\sigma_{r_n}).
\]
Write
\begin{multline} \label{eq:biasint}
\int_0^\infty \sqrt k_n \{ \tilde H_n(y)   - G(y) \} f'(y) \, dy
= 
- \int_0^{c/\sigma_{r_n}} \sqrt k_n G(y)  f'(y) \, dy  \\ 
+ \int_{c/\sigma_{r_n}}^\infty \sqrt k_n A_n(y) \{ G_n (y) - G(y) \} f'(y) \, dy  \\
+\int_{c/\sigma_{r_n}}^\infty \sqrt k_n \{A_n (y)  - 1 \}  G(y) f'(y) \, dy.
\end{multline}
The first integral converges to $0$ as shown in the proof of Theorem~\ref{theo:ml}, treatment of $J_{n1}(f)$.
The integrand of the second integral converges pointwise to the same limit as in the iid case. The integrand can further be bounded by an integrable function as shown in the treatment of  $J_{n2}$ in the proof of Theorem~\ref{theo:ml}, after splitting the integration domain at $1$. Hence, the limit of that integral is the same as in the iid case by dominated convergence.

Consider the last integral in the latter display.  Decompose
\[
\sqrt {k_n} | A_n(y) - 1 | = 
\frac{\sqrt{k_n}}{r_n}\cdot \frac{1-A_n(y)}{-\log A_n(y)}   \cdot  \frac{-\log A_n(y)}{-\log F(a_{r_n})} ,
\]
where we used the fact that $r_n\{-\log F(a_{r_n}) \}= 1$.
The second factor is bounded by $1$, since $\log(x) \le x-1$ for all $x>0$. Consider the third factor. With $L(x) = - \log \{F(x) \} x^{\alpha_0}$, we have
\[
\frac{-\log A_n(y)}{-\log F(a_{r_n})}
=
(yb_{(2)}/b_{(1)})^{-\alpha_0} \frac{L(a_{r_n}y b_{(2)}/b_{(1) } )}{L(a_{r_n})}.
\]
The fraction on the right-hand side is bounded by a multiple of $y^\delta \vee y^{-\delta}$ by Potter's theorem, for some $0<\delta<\alpha_0$. Further note that, up to a factor, $f'(y) \le y^{-\alpha_0-\delta-1}$ for $y\le 1$ and $f'(y) \le y^{-1}$ for $y>1$. We obtain that the integrand of the third integral on the right-hand side of  \eqref{eq:biasint} is bounded by a multiple of
\[
\sqrt k_n /r_n \cdot \exp(-y^{-\alpha_0}) y^{-2\alpha_0-2\delta-1}
\]
for $y\le 1$ and by a multiple of
\[
\sqrt k_n /r_n \cdot  y^{-\alpha_0-1+\delta}
\]
for $y>1$. Both functions are integrable on its respective domains. Since $k_n=o(n^{2/3})$ is equivalent to $\sqrt k_n=o(r_n)$, the third integral converges to $0$. Hence, Condition~\ref{cond:bias} is satisfied.
\end{proof}

\section{Proofs for Section~\ref{sec:aux}}
\label{sec:proofs:aux}

\begin{proof}[Proof of Lemma~\ref{lem:moments}]
If $Y$ is a unit exponential random variable, then the law of $Y^{-1/\alpha_0}$ is equal to $P$. The integrals stated in the lemma are equal to $\Exp[Y^{\alpha/\alpha_0}]$, $-\alpha_0^{-1} \, \Exp[Y^{\alpha/\alpha_0} \log(Y)]$, and $\alpha_0^{-2} \, \Exp[Y^{\alpha/\alpha_0} (\log Y)^2]$, respectively.
% Recall that $dP(x) / dx = \alpha_0 \, x^{-\alpha_0-1} \exp(-x^{-\alpha_0})$. 
First,
\begin{align*}
  \int_0^\infty x^{-\alpha} \, \diff P(x)
%   &= \int_0^\infty \alpha_0 \, x^{-\alpha-\alpha_0-1} \exp(-x^{-\alpha_0}) \, dx \\
%   &= \int_0^\infty x^{-\alpha} \, \exp(-x^{-\alpha_0}) \, d(-x^{-\alpha_0}) \\
  &= \int_0^\infty y^{\alpha/\alpha_0} \, \exp(-y) \, \diff y 
  = \Gamma(1 + \alpha/\alpha_0).
\end{align*}
Second,
\begin{align*}
  \int_0^\infty x^{-\alpha} \log(x) \, \diff P(x)
%   &= \int_0^\infty \log(x) \, \alpha_0 \, x^{-\alpha-\alpha_0-1} \, \exp(-x^{-\alpha_0}) \, dx \\
%   &= \int_0^\infty \log(x) \, x^{-\alpha} \, \exp(-x^{-\alpha_0}) \, d(-x^{-\alpha_0}) \\
%   &= \int_0^\infty \log(y^{-1/\alpha_0}) \, y^{\alpha/\alpha_0} \, \exp(-y) \, dy \\
  &= -\frac{1}{\alpha_0} \int_0^\infty \log(y) \, y^{\alpha/\alpha_0} \, \exp(-y) \, \diff y 
  = -\frac{1}{\alpha_0} \Gamma'(1 + \alpha/\alpha_0).
\end{align*}
Third,
\begin{align*}
  \int_0^\infty x^{-\alpha} (\log x)^2 \, \diff P(x)
  &= \frac{1}{\alpha_0^2} \int_0^\infty (\log y)^2 \, y^{\alpha/\alpha_0} \, \exp(-y) \, \diff y 
  = \frac{1}{\alpha_0^2} \, \Gamma''(1 + \alpha/\alpha_0). \qedhere
\end{align*}
\end{proof}

%\begin{lemma}[Covariance matrix]
%\label{lem:cov}
%Let $X$ be a random variable whose distribution is Fr\'echet with parameter vector $(\alpha_0, 1)$. The covariance matrix of the random vector
%\[
%  \bm{Y} = (Y_1, Y_2, Y_3)^T = \bigl( X^{-\alpha_0} \log(X), \, X^{-\alpha_0}, \, \log(X) \bigr)^T
%\]
%is equal to
%\begin{align*}
%  \cov( \bm{Y} )
%  &=
%  \frac{1}{\alpha_0^2}
%  \begin{pmatrix}
%    1-4\gamma+\gamma^2+\pi^2/3 & \alpha_0(\gamma - 2) & \pi^2/6-\gamma \\
%    \alpha_0(\gamma - 2) & \alpha_0^2 & -\alpha_0 \\
%    \pi^2/6-\gamma &-\alpha_0 &  \pi^2/6
%  \end{pmatrix}.
%\end{align*}
%% where the lower triangular part has been omitted for convenience.
%\end{lemma}

\begin{proof}[Proof of Lemma~\ref{lem:cov}]
Recall a few special values of the first two derivatives of the Gamma function:
\begin{align*}
  \Gamma'(1) &= -\gamma, &
  \Gamma''(1) &= \gamma^2 + \pi^2/6, \\
  \Gamma'(2) &= 1-\gamma, &
  \Gamma''(2) &= (1-\gamma)^2 + \pi^2/6 - 1, \\
  \Gamma'(3) &= 3-2\gamma, &
  \Gamma''(3) &= 2((3/2-\gamma)^2 + \pi^2/6 - 5/4).
\end{align*}
Applying the formulas in Lemma~\ref{lem:moments} with $\alpha \in \{ 0, \alpha_0, 2 \alpha_0 \}$, we find
\begin{align*}
  \var( Y_1 ) &= \alpha_0^{-2} \, \bigl\{ \Gamma''(3) - (\Gamma'(2))^2 \bigr\} =  \alpha_0^{-2} (1-4\gamma+\gamma^2+\pi^2/3)  , \\
  \var( Y_2 ) &= \Gamma(3) - (\Gamma(2))^2 = 1, \\
  \var( Y_3 ) &= \alpha_0^{-2} \, \bigl( \Gamma''(1) - (\Gamma'(1))^2 \bigr) = \alpha_0^{-2} \pi^2/6,
\end{align*}
as well as
\begin{align*}
  \cov( Y_1, Y_2 ) &= \alpha_0^{-1} \, \bigl( (-\Gamma'(3)) - (-\Gamma'(2)) \Gamma(2) \bigr) = \alpha_0^{-1} \, (\gamma - 2), \\
  \cov( Y_1, Y_3 ) &= \alpha_0^{-2} \, \bigl( \Gamma''(2) - (-\Gamma'(2)) (-\Gamma'(1)) \bigr) = \alpha_0^{-2} (\pi^2/6 - \gamma), \\
  \cov( Y_2, Y_3 ) &= \alpha_0^{-1} \, \bigl( (-\Gamma'(2)) - \Gamma(2) (-\Gamma'(1)) \bigr) = -\alpha_0^{-1}. \qedhere
\end{align*}
\end{proof}

%\begin{lemma}[Fisher information]
%\label{lem:fisher}
%Let $P_\theta$ denote the Fr\'echet distribution with parameter $\theta=(\alpha,\sigma)\in(0,\infty)^2$. The Fisher information $I_{\theta} = P_\theta (\dot{\ell}_{\theta} \dot{\ell}_{\theta}^T)$ is given by 
%\[
%I_{\theta} =
%\left( 
%\begin{array}{cc}  \iota_{11} & \iota_{12} \\ \iota_{21} & \iota_{22} \end{array} \right)
%=
%\left( 
%%\begin{array}{cc} \frac{5-2\gamma+\gamma^2+\pi^2/6}{\alpha^2} & \frac{1-\gamma}{\sigma} \\ \frac{1-\gamma}{\sigma} & (\frac{\alpha}{\sigma})^2 \end{array} \right)
%\begin{array}{cc} \{ (1-\gamma)^2+\pi^2/6\}/{\alpha^2} & (1-\gamma)/{\sigma} \\ (1-\gamma)/{\sigma} & \alpha^2/\sigma^2 \end{array} \right).
%\]
%Its inverse is given by
%\[
%I_{\theta}^{-1} 
%= \frac{6}{\pi^2} \left( 
%\begin{array}{cc} \alpha^2 & (\gamma-1) {\sigma} \\ (\gamma-1){\sigma} &  (\sigma/\alpha)^2 \{ (1-\gamma)^2 +\pi^2/6\} \end{array} \right).
%\]
%\end{lemma}

\begin{proof}[Proof of Lemma~\ref{lem:fisher}]
 If $X\sim P_{(\alpha,\sigma)}$, then $Z=X/\sigma \sim P_{(\alpha,1)}$. Therefore, by \eqref{eq:score:alpha} and  Lemma~\ref{lem:moments},
\begin{align*}
\iota_{11} &= \Exp\big[\{ \alpha^{-1} + (Z^{-\alpha} - 1) \log Z \}^2 \big] \\
&= 
\frac{1}{\alpha^2}\big[ 1 - 2 \{ \Gamma'(2) - \Gamma'(1) \} + \{ \Gamma''(3) - 2 \Gamma''(2) + \Gamma''(1) \} \big]  \\
&= 
\frac{1}{\alpha^2} \{ (1-\gamma)^2 + \pi^2/6\}.
\end{align*}
Similarly, by \eqref{eq:score:alpha} and  \eqref{eq:score:sigma}, 
\begin{align*}
\iota_{12} 
&= 
\frac{\alpha}{\sigma} \Exp\big[(1-Z^{-\alpha}) \{ \alpha^{-1} + (Z^{-\alpha} - 1) \log Z \} \big] \\
&=  
\frac{\alpha}{\sigma} \left[ \alpha^{-1} \{ \Gamma(1) - \Gamma(2) \} + \alpha^{-1} \{ \Gamma'(1) - 2\Gamma'(2) + \Gamma'(3) \} \right\}  \\
&= 
\frac{1-\gamma}{\sigma}. 
\end{align*}
Finally, 
\begin{align*}
\iota_{22} 
&= 
\frac{\alpha^2}{\sigma^2} \Exp[(1-Z^{-\alpha})^2 ]  
= 
\frac{\alpha^2}{\sigma^2} \{ \Gamma(1)-2\Gamma(2)+\Gamma(3)\} 
=
\frac{\alpha^2}{\sigma^2}. \qedhere
\end{align*}
\end{proof}

\section{Finite-sample bias and variance}
\label{sec:simulextra}

We work out the second-order Condition~\ref{cond:secor} and the expressions for the asymptotic bias and variance of the maximum likelihood estimator of the Fr\'echet shape parameter for the case of block maxima extracted from an independent random sample from the absolute value of a Cauchy distribution. Furthermore, we compare these expressions to those obtained in finite samples from Monte Carlo simulations.

If the random variable $\xi$ is Cauchy-distributed, then $\lvert\xi\rvert$ has distribution function 
\[
  F(x) 
  = \Prob\lbrace|\xi| \le x\rbrace 
  = \frac2\pi \arctan(x) \, \mathds{1}(x>0),
  \qquad x \in \R.
\]
Based on the asymptotic expansion
\[
  - \log \bigl( \arctan (x) \bigr) 
  =  
  \log \Bigl( \frac2\pi \Bigl) + \frac{2}{\pi x } + \frac{2}{\pi^2 x^2} + O\Bigl(\frac1{x^3}\Bigr), \qquad x \to \infty,
\]
one can show that $-\log F$ is regularly varying at infinity with index $-\alpha_0 = -1$ and that the limit relation
\begin{equation*}
  \lim_{u \to \infty} 
  \frac{1}{A(u)} 
  \left( 
    \frac{-\log F(ux)}{-\log F(u)} - x^{-1} 
  \right) 
  = 
  x^{-1} \, h_\rho(x)
\end{equation*}
is satisfied for 
\[
\rho=-1 \quad \text{ and } \quad  A(u) = - \frac{1}{1+\pi u}.
\]
In addition, the normalizing sequence $(a_n)_{n \in \mathbb{N}}$ can be chosen as $a_n=\frac{2n}{\pi}$.

By Theorem~\ref{theo:ml}, these facts imply that the theoretical bias and variance of $\hat \alpha_n$ are given by
\begin{align*}
  \operatorname{Bias} 
  &= -A(a_{r_n}) \frac{6}{\pi^2} b_1(|\rho|) =  \frac{12}{\pi^2(1+2r_n)}, & \operatorname{Variance}
  &= \frac{1}{k_n} \frac{\pi^2}{6}.
\end{align*}
In particular, the mean squared error is of the order $O(1/r_n^2)+O(1/k_n)$, which can be minimized by balancing the block size $r_n$ and the number of blocks $k_n$, that is, by choosing $r_n=O(n^{1/3})$ and $k_n=O(n^{2/3})$ so that $r_n^2\approx k_n$. More precisely, the equations $n = kr$ and $(\frac{12}{\pi^2(1+2r)})^2 = \frac{1}{k} \frac{\pi^2}{6}$ imply that $\frac{864}{\pi^6} n = r(1+2r)^2$, which for $n=1\,000$ implies that $r \approx 6$ and $k \approx 174$. These values are quite close to the optimal finite-sample values of $r=4$ and $k=250$ to be observed in the upper-left panel of Figure~\ref{fig:mse}.

\begin{figure}
\centering
\vspace{-0.8cm}
\includegraphics[width=0.48\textwidth]{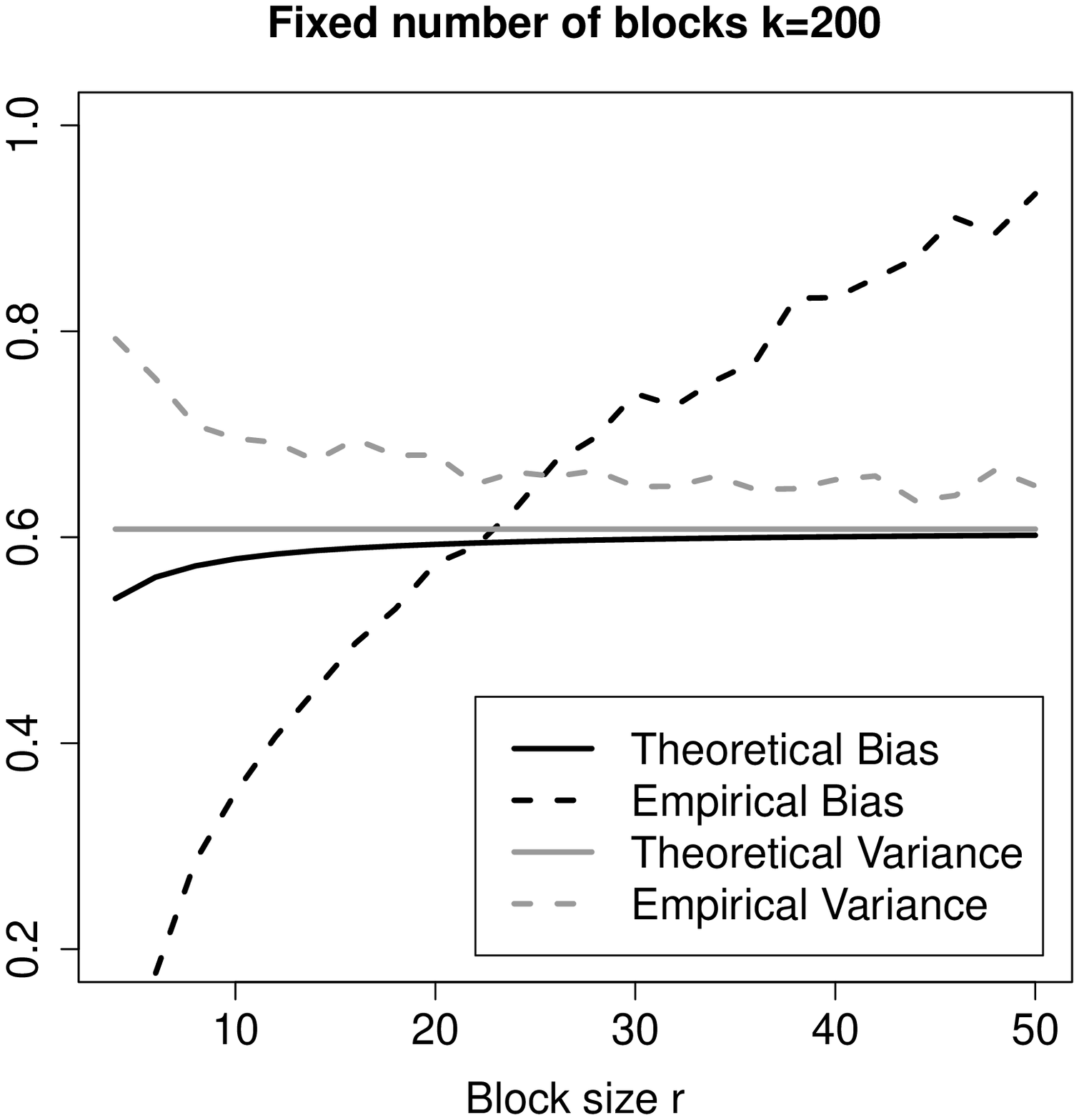}
\includegraphics[width=0.48\textwidth]{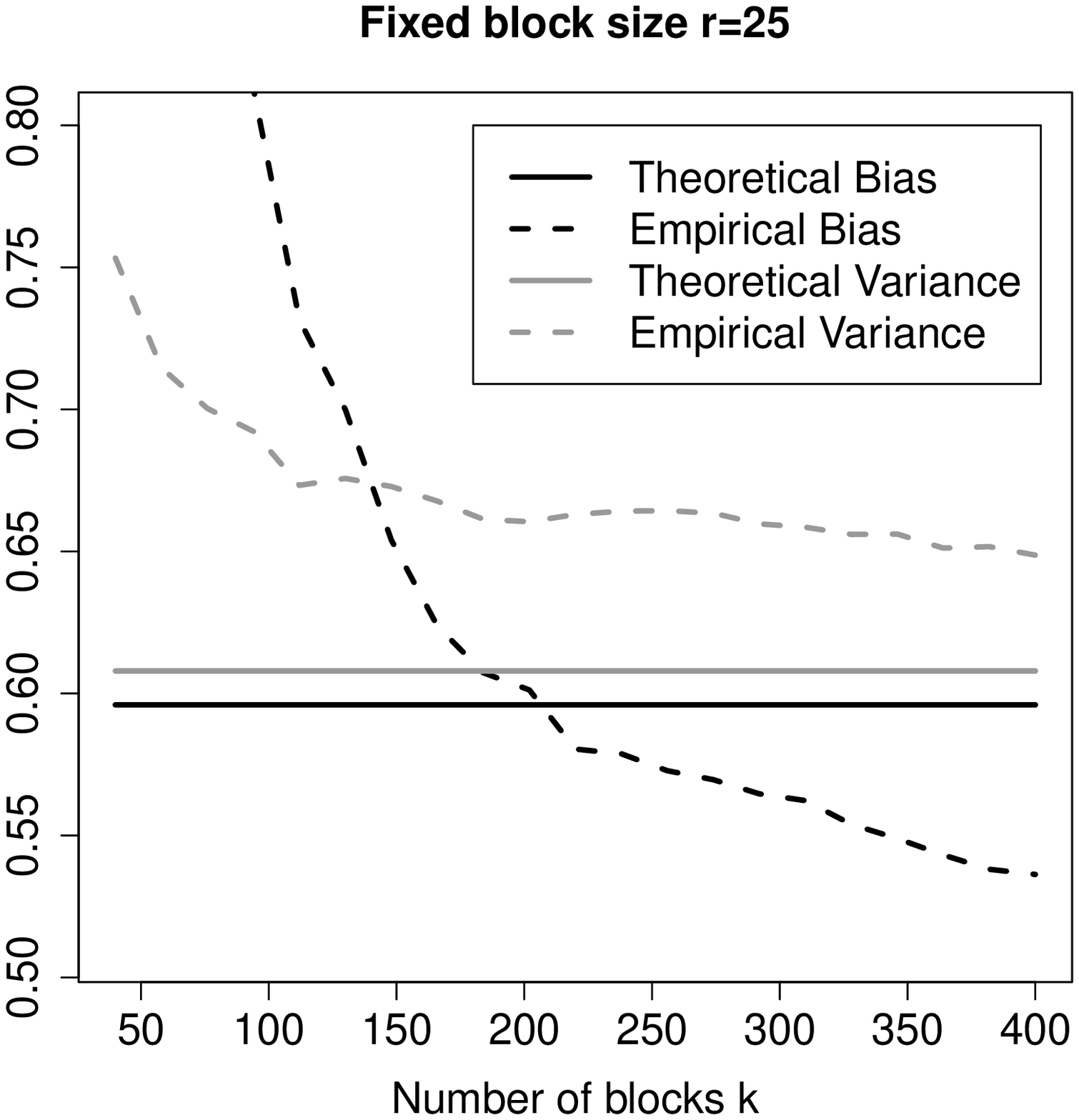}
\includegraphics[width=0.48\textwidth]{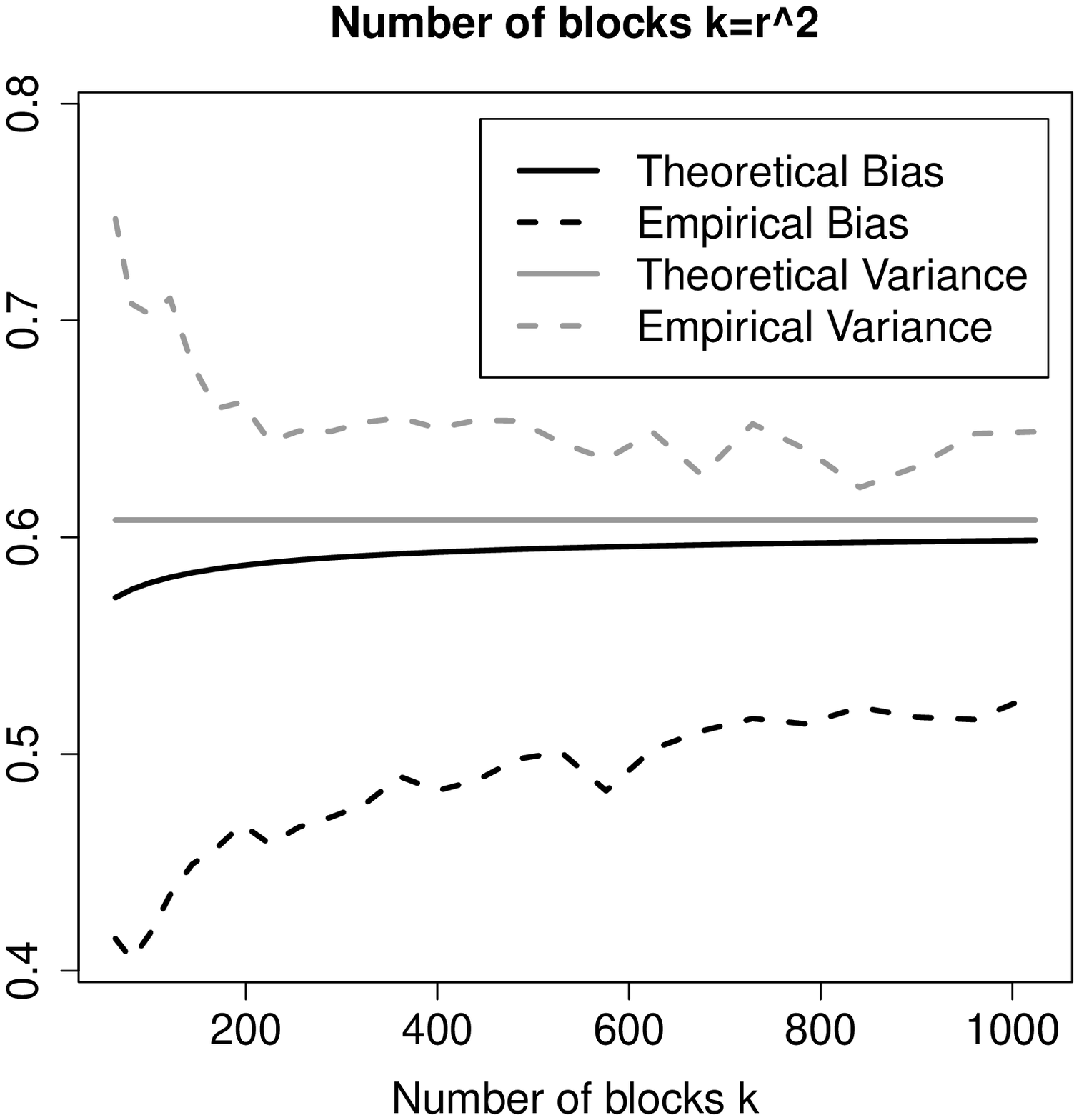}
\caption{Simulation results in the iid Cauchy-model (see Section~\ref{sec:simulextra}). Theoretical bias multiplied with $r$ and theoretical variance multiplied with $k$, together with finite-sample approximations based on $N=5\,000$ simulation runs. In the upper left picture, the number of blocks is fixed at $k=200$; in the upper right picture, the size of the blocks is fixed at $r=25$; in the lower picture, finally, the number of blocks $k$ and the block size $r$ satisfy $r^2=k$, as suggested by (approximately) minimizing the mean squared error. }
\label{fig:biasvarapprox}
\end{figure}

In Figure~\ref{fig:biasvarapprox}, we depict results of a Monte-Carlo simulation study on the finite-sample approximation of the theoretical bias, multiplied by $r$, and of the theoretical variance, multiplied by $k$. Three scenarios have been considered: 
\begin{itemize}
\item fixed number of blocks $k=200$ and block sizes $r= 4,\dots, 50$;
\item fixed block size $r=25$ and number of blocks $k=40, \dots, 400$;
\item block sizes $r=8,9, \dots, 32$ and number of blocks $k=r^2$.
\end{itemize}

We find that the variance approximation improves with increasing $r$ or $k$. For the bias approximation to improve, both $r$ and $k$ must increase.

\end{document}